\documentclass{article}
\usepackage{graphicx} 
\usepackage{amsmath,amsfonts,amssymb,amsthm}
\usepackage{hyperref}
\usepackage{bm}
\usepackage{booktabs}
\usepackage{dutchcal}
\usepackage{mathrsfs}
\DeclareMathAlphabet{\mathpzc}{T1}{pzc}{mb}{n}
\newtheorem{theorem}{Theorem}
\newtheorem{lemma}{Lemma}
\newtheoremstyle{myplain}    
  {2pt}                      
  {2pt}                      
  {}                         
  {}                         
  {\bfseries}               
  {.}                        
  {.5em}                     
  {}                         

\theoremstyle{myplain}
\newtheorem{remark}{Remark}
\newcommand{\norm}[1]{\left\|#1\right\|}

\title{\textbf{A discontinuous Galerkin pressure correction scheme for the Oldroyd model of order one}}
\author{Pratyay Mondal\thanks{\tt m.pratyay@iitg.ac.in} and Rajen Kumar Sinha\thanks{{\tt rajen@iitg.ac.in} \\ Department of Mathematics, Indian Institute of Technology Guwahati, Guwahati-781039, India.}}
\date{}

\begin{document}

\maketitle

\textbf{Abstract:}
We develop and analyze a discontinuous Galerkin pressure correction scheme for the Oldroyd model of order one. The existence and uniqueness of the discrete solution as well as the consistency of the scheme are proved. The stability of the discrete velocity and pressure are established. We derive optimal \textit{a priori} error bounds for the fully discrete velocity in the discontinuous discrete space. In addition, an improved error estimate for the velocity is derived in the $L^2$ norm which is optimal with respect to space and time. Furthermore, the error bound for the pressure is obtained via the estimates of discrete time derivative of the velocity. Finally, numerical experiments confirm the optimal convergence rates.

\vspace{0.2cm}

\vspace{0.2cm}

\section{Introduction}
We consider the time-dependent Oldroyd flow of order one, which is described by the following nonlinear integro-differential equations
\begin{equation}
\frac{\partial \mathbf{u}}{\partial t} - \mu \Delta \mathbf{u} + \mathbf{u} \cdot \nabla \mathbf{u} - \int_0^t \beta(t - s) \Delta \mathbf{u}(x, s) ds + \nabla p = \mathbf{f}(x, t), \quad x \in \mathcal{O}, \, t > 0, \label{eq3.1}
\end{equation}
along with the continuity equation
\begin{equation}
\nabla \cdot \mathbf{u} = 0, \quad x \in \mathcal{O}, \, t > 0, \label{eq3.2}
\end{equation}
and the initial-boundary conditions
\begin{equation}
\mathbf{u}(x, 0) = \mathbf{u}_0 \quad \text{in } \mathcal{O}, \qquad \mathbf{u} = \mathbf{0} \quad \text{on } \partial \mathcal{O}, \, t \geq 0, \label{eq:3.3}
\end{equation}
along with a zero average constraint for the pressure
\begin{equation}
\int_\mathcal{O} p = 0,   \label{eq3.4}
\end{equation}
where $\mathcal{O}$ is a convex polyhedral domain in $\mathbb{R}^{d}$, $d\in\{2,3\}$, which is open and bounded, $\partial \mathcal{O}$ is the boundary of the domain $\mathcal{O}$. In the above, $\mathbf{u}=\mathbf{u}(x,t)$ represents the velocity of the fluid, $p=p(x,t)$ denotes the pressure, and $\mathbf{f}=\mathbf{f}(x,t)$ stands for the external force. Here $\beta(t)=\gamma \exp(-\eta t)$ with $\gamma =2 \lambda^{-1}(\mu-\kappa \lambda^{-1})>0$ and $\eta=\lambda^{-1}>0$, where $\lambda>0$ is the relaxation time, $\kappa>0$ is the retardation time, and $\mu=2\kappa \lambda^{-1}>0$ is the kinetic coefficient of viscosity.

The model \eqref{eq3.1} originally considered as an integral perturbation of the Navier–Stokes equations (NSEs), represents linear viscoelastic fluids, with the integral term accounting for the flow's history. It is widely applied in the modeling of dilute polymeric fluids, polymeric suspensions, biological flows like blood circulation, and industrial processes such as inkjet printing and lubrication. The model is also valuable in oil recovery and microfluidics. It is derived under the assumption that the material can be treated as a single stationary macroscopic element that experiences low stress and strain rates. For more details on the physical background and its mathematical modeling, we refer to previous studies \cite{joseph1990fluid,oldroyd1956non}.

Early studies on the model and its finite element approximations are available in previous research \cite{goswami2011priori,pani2005semidiscrete}. For time-discrete schemes, we refer to other works \cite{bir2022backward,guo2022crank,pani2006linearized} and the references listed therein. Previous studies have examined various numerical approaches for this model, including the stability of time-discrete schemes \cite{guo2018euler,zhang2018stability}, projection methods \cite{liu2019incremental,zhao2018stability}, stabilized techniques \cite{wang2012stabilized}, the characteristics finite element method \cite{yang2020unconditional}, the penalty method \cite{bir2022finite}, the two-grid method \cite{bir2021three} and the discontinuous Galerkin (dG) method \cite{ray2024discontinuous}. To address complex computational challenges, the development of splitting schemes focuses on separating the non-linearity in the convection term from the pressure term. An in-depth analysis of these methods can be found in \cite{glowinski2003finite} and \cite{guermond2006overview}. This study specifically focuses on the pressure correction schemes. Chorin and Temam first introduced the concept of a non-incremental pressure correction scheme over time in \cite{chorin1968numerical,temam1969approximation}. Over the years, this approach has been refined by several researchers, resulting in two prominent variants: (1) the incremental scheme, which incorporates the previous pressure gradient value \cite{goda1979multistep,van1986second}, and (2) the rotational scheme, which employs the rotational form of the Laplacian to manage non-physical boundary conditions for pressure \cite{timmermans1996approximate}. The semi-discrete analysis of pressure correction schemes has been extensively studied, as seen in the work by Shen et al. \cite{guermond2004error,shen1996error}. \textit{A priori} error estimates for the finite element approximation of the pressure correction scheme can be found in the works of Guermond et al. \cite{guermond1998stability,guermond1998approximation}, and Nochetto et al. \cite{nochetto2005gauge}.

More recently, several research articles have explored the use of dG spatial approximations combined with pressure correction formulations to solve the incompressible NSEs. These dG methods offer key advantages, such as high-order accuracy and local conservation. In the following, we present some relevant articles in this field. Botti et al. \cite{botti2011pressure}, for instance, employ a dG approximation for the velocity, coupled with a continuous Galerkin method for the pressure. Liu et al. apply an interior penalty dG method combined with a pressure correction approach \cite{liu2019interior}. Piatkowski et al. use a modified upwind scheme, followed by postprocessing of the projected velocity to ensure discrete divergence-free properties \cite{piatkowski2018stable}. Fehn et al. \cite{fehn2017stability} investigate the stability of pressure correction and velocity correction dG methods numerically for small time steps. Masri et al. establish the stability and convergence of a pressure correction scheme combined with dG approximation to solve the incompressible NSEs \cite{masri2022discontinuous,masri2023improved}.

We study here an \textit{a priori} error analysis for a pressure-correction dG scheme applied to the first-order Oldroyd model, which is missing from the existing literature to the best of our knowledge. Besides the advantages of dG methods, this splitting method stabilizes the solution by ensuring that the pressure and velocity fields are consistent with each other. Also, this pressure correction method addresses the coupling between velocity and pressure by introducing a correction step that updates the pressure field to ensure that the continuity equation is satisfied, which helps the solver converge to a more accurate solution. The techniques involved in the analysis of this model problem differ from the derivation techniques of NSEs due to the additional integral term, and it is worth mentioning here that a certain positivity property of this integral term plays a crucial role in the analysis. For the fully discrete scheme, the backward Euler method is employed, and we have chosen the right rectangle rule to approximate the integral term. Then, we discuss the existence of unique discrete solutions and the consistency of the scheme. The discrete velocities are approximated by polynomials of degree $r_1$ and the discrete potential and pressure by polynomials of degree $r_2$. Introducing an auxiliary function appropriately, we obtain the stability of the scheme under the constraint $r_1-1 \leq r_2 \leq r_1+1$. The convergence results of the scheme are obtained for the case $r_2=r_1-1$ in order to use the approximation properties. The error estimate for the velocity in the dG norm is obtained, which is optimal in space but suboptimal in time. Then, we derive an improved error estimate for the velocity by carefully considering the dual Stokes problem and assuming the domain to be convex. Optimal error estimates for the velocity in the $L^2$ norm are obtained in both time and space without any condition for linear polynomials and with a reverse CFL condition for polynomials of degree two or higher. Error estimates for pressure are obtained by obtaining the stability and error estimates for discrete time derivative of the velocity, and with the help of the discrete inf-sup condition. These pressure estimates are optimal in space and suboptimal in time. The proofs are technical and rely on several tools, including special lift operators. Finally, we perform numerical computations to validate our results.

The paper is arranged as follows. In Section 2, the pressure correction scheme in time is introduced for the model problem, and it also contains the notation and discrete function spaces. In Section 3, we present the fully discrete dG scheme and show the existence, uniqueness of solutions, and consistency of the scheme. The stability result is derived in Section 4. The standard approximation properties, along with the optimal energy norm error estimate of the discrete velocity, are presented in Section 5. The optimal $L^2(L^2)$-norm error estimates for the velocity are derived in Section 6. We present the bounds for the discrete time derivative of the velocity in Section 7. The error estimates for the pressure are provided in Section 8. Section 9 presents the numerical experiments and results, while Section 10 concludes the article with a summary of the findings.

\section{Pressure correction scheme}
We apply a uniform partition of the time interval $(0,T]$ into $N$ subintervals with $\tau$ denoting the time step size. We denote $\bm\varphi^n = \bm\varphi(t_n)$ and $q^n = q(t_n)$, where $t_n = n\tau$ for given functions $\bm\varphi$ and $q$. Formulation of this splitting scheme is inspired by the established literature \cite{masri2022discontinuous} and references therein. The method is based on the splitting of the operators which introduce intermediate velocity $\widetilde{\mathbf{u}}^n$ and potential $v^n$. The pressure correction scheme for the model problem \eqref{eq3.1}-\eqref{eq3.4} is stated as follows. Set $p^0 = 0$, for $n = 1, \ldots, N$, given $\mathbf{u}^{n-1}$ and $p^{n-1}$, compute an intermediate velocity $\widetilde{\mathbf{u}}^n$ such that
\begin{equation}
\widetilde{\mathbf{u}}^n - \tau \mu \Delta \widetilde{\mathbf{u}}^n+\tau \mathbf{u}^{n-1} \cdot \nabla \widetilde{\mathbf{u}}^{n} - \tau \int_0^{t_{n}} \beta(t_{n} - s) \Delta \widetilde{\mathbf{u}}^{n}(x, s) ds + \tau \nabla p^{n-1} = \mathbf{u}^{n-1} + \tau \mathbf{f}^n \quad \text{in } \mathcal{O}, \label{eq3.5}
\end{equation}
\begin{equation}
\widetilde{\mathbf{u}}^n = \mathbf{0} \quad \text{on } \partial \mathcal{O}.
\end{equation}
For given $\widetilde{\mathbf{u}}^n$, compute the potential $v^n$ such that,
\begin{align}
- \Delta v^n &= -\frac{1}{\tau} \nabla \cdot \widetilde{\mathbf{u}}^n \quad \text{in } \mathcal{O}, \\
\nabla v^n \cdot \mathbf{n} &= 0 \quad \text{on } \partial \mathcal{O}, \\
\int_\mathcal{O} v^n &= 0.
\end{align}
Given $(\widetilde{\mathbf{u}}^n, p^{n-1}, v^n)$, update the pressure $p^n$ and the velocity $\mathbf{u}^n$ by solving
\begin{equation}
p^n = p^{n-1} + v^n - \delta \mu \nabla \cdot \widetilde{\mathbf{u}}^n - \delta \int_0^{t_{n}} \beta(t_{n} - s) \nabla \cdot \widetilde{\mathbf{u}}^{n}(x, s) ds, \label{eq:3.1}
\end{equation}
\begin{equation}
\mathbf{u}^n = \widetilde{\mathbf{u}}^n - \tau \nabla v^n. \label{eq3.11}
\end{equation}
The vector $\mathbf{n}$ denotes the unit normal vector pointing outward to $\partial \mathcal{O}$. In \eqref{eq:3.1}, $\delta$ is a positive parameter that will be specified later in the paper.

Throughout this paper, $C, \tilde{C}$ and $\tilde{C}_T$ are positive constants. The constant $C$ is independent of $\mu, \gamma, h$ and $\tau$, the constant $\tilde{C}$ depend on $\mu, \gamma $, but independent of $h$ and $\tau$, whereas the constant $\tilde{C}_T$ is independent of $h$ and $\tau$ but depends on $\mu, \gamma$ and $T$.
\subsection{Notation and preliminaries}
For a domain $\mathcal{O} \subset \mathbb{R}^d$ and for a given integer $s \geq 0$ and a real number $k \geq 1$, we consider the standard Sobolev space $W^{s,k}(\mathcal{O})$ (cf. \cite{adams2003sobolev}) equipped with the usual norm denoted by $\| \cdot \|_{W^{s,k}(\mathcal{O})}$ and the semi-norm by $| \cdot |_{W^{s,k}(\mathcal{O})}$. If $k = 2$, we denote the Hilbert space $H^s(\mathcal{O}) = W^{s,2}(\mathcal{O})$, the norm and the semi-norm are denoted by $\| \cdot \|_{H^s(\mathcal{O})}$ and $| \cdot |_{H^s(\mathcal{O})}$, respectively. The norm and the inner-product on $L^2(\mathcal{O})$ is denoted by $\|\cdot\|$ and $(\cdot, \cdot)$, respectively.

Let $\mathcal{K}_h = \{ K \}$ denote a family of shape-regular and uniform partition of the domain $\mathcal{O}$ \cite{ciarlet2002finite}. That means, there is a parameter $\rho > 0$, independent of $h$, such that
\[
    \frac{h_K}{\rho_K} \leq \rho, \quad \forall K \in \mathcal{K}_h,
\]
where $\rho_K$ is the radius of the largest ball inscribed in $K$, $h_K = \text{diam}(K)$, and the mesh-size $h=\max_{K \in \mathcal{K}_h} h_K$.
We introduce the following broken Sobolev spaces:
\begin{align*}
    \mathbf{M} &= \{ \bm\varphi \in (L^2(\mathcal{O}))^d \mid \forall K \in \mathcal{K}_h, \quad \bm\varphi|_K \in (W^{2,2d/(d+1)}(K))^d \}, \\
    P &= \{ g \in L^2(\mathcal{O}) \mid \forall K \in \mathcal{K}_h, \quad g|_K \in W^{1,2d/(d+1)}(K) \}.
\end{align*}

Let $\mathcal{F}_h^i$ and $\mathcal{F}_h^b$ denote the set of all interior faces (or edges) and set of all boundary faces (or edges), respectively and $\mathcal{F}_h$ is the union of all the faces (or edges) of $\mathcal{K}_h$. For $F \in \mathcal{F}_h^i$, we associate a unit normal vector $\mathbf{n}_F$ and define $K^1_F$ and $K^2_F$ as the two elements sharing $F$, ensuring that $\mathbf{n}_F$ directs from $K^1_F$ to $K^2_F$. The average and jump of a function $\mathbf{w} \in \mathbf{M}$ are defined as:
\[
\{\mathbf{w}\} = \frac{1}{2} (\mathbf{w}|_{K^1_F} + \mathbf{w}|_{K^2_F}),
\quad 
[\mathbf{w}] = \mathbf{w}|_{K^1_F} - \mathbf{w}|_{K^2_F}, \quad \forall F = \partial K^1_F \cap \partial K^2_F.
\]
For a boundary face $F \in \mathcal{F}_h^b$, the normal vector $\mathbf{n}_F$ is chosen as the outward unit normal to $\mathcal{F}_h^b$. In this case, the definitions of average and jump extend to:
\[
\{\mathbf{w}\} = [\mathbf{w}] = \mathbf{w}|_{K_F}, \quad \forall F \in \mathcal{F}_h^b.
\]
Similar definitions hold for scalar-valued functions $g \in P$. 

Let $\mathbf{n}_K$ represent the normal pointing outward to $K$. The notations $\bm\varphi^{\text{int}}$ and $\bm\varphi^{\text{ext}}$ denote the trace of a function $\bm\varphi$ on the boundary of $K$ coming from the interior and exterior, respectively. For boundary faces $F$ ($F \subset \mathcal{F}_h^b$), we take $\bm\varphi^{\text{ext}}|_F = 0$. \\
To formulate the dG method for the spatial discretization of \eqref{eq3.5}–\eqref{eq3.11}, we follow the same discretization as in \cite{girault2005discontinuous} for the convection term. Using above notations, we now define $\forall \bm{\theta}, \mathbf{z}, \bm\varphi, \mathbf{w} \in \mathbf{M}$,
\begin{align*}
\mathcal{A}_c(\bm{\theta}; \mathbf{z}, \bm\varphi, \mathbf{w}) = \sum_{K \in \mathcal{K}_h} &\left( \int_K (\mathbf{z} \cdot \nabla \bm\varphi) \cdot \mathbf{w} + \frac{1}{2} \int_K (\nabla \cdot \mathbf{z}) \bm\varphi \cdot \mathbf{w} \right) \\
&         - \frac{1}{2} \sum_{F \in \mathcal{F}_h} \int_F [\mathbf{z}] \cdot \mathbf{n}_F \{ \bm\varphi \cdot \mathbf{w} \} 
+ \sum_{K \in \mathcal{K}_h} \int_{\partial K^{\bm{\theta}}_{-}} |\{ \mathbf{z} \} \cdot \mathbf{n}_K| (\bm\varphi^{\text{int}} - \bm\varphi^{\text{ext}}) \cdot \mathbf{w}^{\text{int}},
\end{align*}
where the inflow boundary of $K$ with respect to a function $\bm{\theta} \in \mathbf{M}$ is defined as:
\[
\partial K^{\bm{\theta}}_{-} = \{ x \in \partial K \mid \{\bm{\theta}(x)\} \cdot \mathbf{n}_K < 0 \}.
\]
$\mathcal{A}_c$ satisfies the following positivity property:
\begin{equation}
\mathcal{A}_c(\mathbf{z}; \mathbf{z}, \bm\varphi, \bm\varphi) \geq 0, \quad \forall \mathbf{z}, \bm\varphi \in \mathbf{M}. \label{eq4.12}
\end{equation}
To facilitate the analysis, it is beneficial to define the following for $\bm{\theta}, \mathbf{z}, \bm\varphi, \mathbf{w} \in \mathbf{M}$:
\begin{equation*}
\mathscr{C}(\mathbf{z}, \bm\varphi, \mathbf{w}) = \sum_{K \in \mathcal{K}_h} \left( \int_K (\mathbf{z} \cdot \nabla \bm\varphi) \cdot \mathbf{w} + \frac{1}{2} \int_K (\nabla \cdot \mathbf{z}) \bm\varphi \cdot \mathbf{w} \right) 
- \frac{1}{2} \sum_{F \in \mathcal{F}_h} \int_F [\mathbf{z}] \cdot \mathbf{n}_F \{ \bm\varphi \cdot \mathbf{w} \}.
\end{equation*}
\begin{equation*}
\mathscr{U}(\bm{\theta}; \mathbf{z}, \bm\varphi, \mathbf{w}) = \sum_{K \in \mathcal{K}_h} \int_{\partial K_-^{\bm{\theta}}} \{ \mathbf{z} \} \cdot \mathbf{n}_K (\bm\varphi^{\text{int}} - \bm\varphi^{\text{ext}}) \cdot \mathbf{w}^{\text{int}}.
\end{equation*}
Consequently, this leads to:
\begin{equation*}
\mathcal{A}_c(\mathbf{z}; \mathbf{z}, \bm\varphi, \mathbf{w}) = \mathscr{C}(\mathbf{z}, \bm\varphi, \mathbf{w}) - \mathscr{U}(\mathbf{z}; \mathbf{z}, \bm\varphi, \mathbf{w}).
\end{equation*}
For the discretization of the elliptic operator $-\Delta \widetilde{\mathbf{u}}$, we define for $\widetilde{\mathbf{u}}, \mathbf{w} \in \mathbf{M}$:
\begin{align*}
    \mathcal{A}_\epsilon(\widetilde{\mathbf{u}}, \mathbf{w}) = \sum_{K \in \mathcal{K}_h}& \int_K \nabla \widetilde{\mathbf{u}} \cdot \nabla \mathbf{w} 
    - \sum_{F \in \mathcal{F}_h} \int_F \{\nabla \widetilde{\mathbf{u}}\} \mathbf{n}_F \cdot [\mathbf{w}] \\
    + &\epsilon \sum_{F \in \mathcal{F}_h} \int_F \{\nabla \mathbf{w}\} \mathbf{n}_F \cdot [\widetilde{\mathbf{u}}] 
    + \sum_{F \in \mathcal{F}_h}\frac{\sigma}{h_F} \int_F [\widetilde{\mathbf{u}}] \cdot [\mathbf{w}].
\end{align*}
In the above expression, $h_F = |F|^{1/(d-1)}$, $\epsilon \in \{-1, 0, 1\}$, and $\sigma > 0$ is a user-defined penalty parameter.
For $\epsilon = -1$, the above bilinear form reduces to the symmetric interior penalty Galerkin scheme. For the elliptic operator $-\Delta \widetilde{\mathbf{u}}$, we define that bilinear form as 
\begin{equation}
\begin{aligned}
    \mathcal{A}_d(\widetilde{\mathbf{u}}, \mathbf{w}) = \sum_{K \in \mathcal{K}_h}& \int_K \nabla \widetilde{\mathbf{u}} \cdot \nabla \mathbf{w} 
    - \sum_{F \in \mathcal{F}_h} \int_F \{\nabla \widetilde{\mathbf{u}}\} \mathbf{n}_F \cdot [\mathbf{w}] \\
    - & \sum_{F \in \mathcal{F}_h} \int_F \{\nabla \mathbf{w}\} \mathbf{n}_F \cdot [\widetilde{\mathbf{u}}] 
    + \sum_{F \in \mathcal{F}_h}\frac{\sigma}{h_F} \int_F [\widetilde{\mathbf{u}}] \cdot [\mathbf{w}].
    \label{18}
\end{aligned}
\end{equation}
\\
The pressure gradient term $-\nabla p$ is discretized as follows. For $\mathbf{w} \in \mathbf{M}$ and $g \in P$, define:
\begin{equation}
    \mathcal{b}(\mathbf{w}, g) = \sum_{K \in \mathcal{K}_h} \int_K (\nabla \cdot \mathbf{w}) g 
    - \sum_{F \in \mathcal{F}_h} \int_F \{g\} \mathbf{n}_F \cdot \mathbf{w}.
    \label{eq4.13}
\end{equation}
The above bilinear form $\mathcal{b}(\mathbf{w},g)$ can be written in  the following equivalent form \cite{masri2022discontinuous}:
\begin{equation}
    \mathcal{b}(\mathbf{w}, g) = -\sum_{K \in \mathcal{K}_h} \int_K \mathbf{w} \cdot \nabla g + \sum_{F \in \mathcal{F}_h^i} \int_F \{\mathbf{w}\} \cdot \mathbf{n}_F [g], \quad \forall (\mathbf{w}, g) \in \mathbf{M} \times P. \label{eq:4.24}
\end{equation}
The following expression discretizes the elliptic operator $-\Delta v$. For $v, g \in P$:
\begin{align*}
    \mathcal{A}_{\text{sip}}(v, g) = \sum_{K \in \mathcal{K}_h} &\int_K \nabla v \cdot \nabla g
    - \sum_{F \in \mathcal{F}_h^i} \int_F \{\nabla v\} \cdot \mathbf{n}_F [g] \\
    &- \sum_{F \in \mathcal{F}_h^i} \int_F \{\nabla g\} \cdot \mathbf{n}_F [v] 
    + \sum_{F \in \mathcal{F}_h^i}  \frac{\tilde{\sigma}}{h_F} \int_F [v] [g],
\end{align*}
where $\tilde{\sigma} > 0$ is a penalty parameter. For $\mathbf{w} \in \mathbf{M}$, the energy norm is defined as:
\begin{equation*}
    \|\mathbf{w}\|^2_\text{dG} = \sum_{K \in \mathcal{K}_h} \|\nabla \mathbf{w}\|^2_{L^2(K)}
    + \sum_{F \in \mathcal{F}_h}\frac{\sigma}{h_F} \|[\mathbf{w}]\|^2_{L^2(F)}.
\end{equation*}
For $g \in P$, the energy semi-norm is defined as:
\begin{equation*}
    |g|^2_\text{dG} = \sum_{K \in \mathcal{K}_h} \|\nabla g\|^2_{L^2(K)}
    + \sum_{F \in \mathcal{F}_h^i}  \frac{\tilde{\sigma}}{h_F} \|[g]\|^2_{L^2(F)}.
\end{equation*}
We define the following discontinuous discrete spaces $\mathbf{M}_h \subset \mathbf{M}$ and $P_h^0 \subset P_h \subset P$ to approximate velocity and pressure, respectively. For any integers $r_1 \geq 1$, $r_2 \geq 0$, we define:
\begin{equation*}
    \mathbf{M}_h = \{ \bm\varphi_h \in (L^2(\mathcal{O}))^d \mid \forall K \in \mathcal{K}_h, \ \bm\varphi_h|_K \in (\mathbb{P}_{r_1}(K))^d \},
\end{equation*}
\begin{equation*}
    P_h = \{ g_h \in L^2(\mathcal{O}) \mid \forall K \in \mathcal{K}_h, \ g_h|_K \in \mathbb{P}_{r_2}(K) \},
\end{equation*}
\begin{equation*}
    P_h^0 = \{ g_h \in P_h \mid \int_\mathcal{O} g_h = 0 \},
\end{equation*}
where $\mathbb{P}_r(K)$ represents the space of polynomials of degree at most $r$, for $r \in \mathbb{N}$. We assume that $r_1 - 1 \leq r_2 \leq r_1 + 1$. \\
Observing the fact $|\cdot|_\text{dG}$ is a norm on the space $P_h^0$, the following coercivity properties hold:
\begin{align}
    \mathcal{A}_\epsilon(\mathbf{w}_h, \mathbf{w}_h) &\geq \omega \|\mathbf{w}_h\|^2_\text{dG}, \quad \forall \mathbf{w}_h \in \mathbf{M}_h, \label{eq3.20}
\\
    \mathcal{A}_{\text{sip}}(g_h, g_h) &\geq \frac{1}{2} |g_h|^2_\text{dG}, \quad \forall g_h \in P_h. \label{eq3.21}
\end{align}
It is established that \eqref{eq3.20} holds with $\omega = 1$ if $\epsilon = 1$ and with $\omega = 1/2$ if $\epsilon \in \{-1,0\}$ and $\sigma$ is sufficiently large. Property \eqref{eq3.21} holds for sufficiently large $\tilde{\sigma}$. Hence, we assume that \eqref{eq3.20} and \eqref{eq3.21} are satisfied (cf. \cite{riviere2008discontinuous}). \\
The bilinear forms $\mathcal{A}_\epsilon$ and $\mathcal{A}_{\text{sip}}$ are continuous on $\mathbf{M}_h \times \mathbf{M}_h$ and $P_h \times P_h$, respectively (cf. \cite{riviere2008discontinuous}), i.e., we have 
\begin{align}
    \mathcal{A}_\epsilon(\mathbf{w}_h, \bm{\theta}_h) &\leq C \|\mathbf{w}_h\|_\text{dG} \|\bm{\theta}_h\|_\text{dG}, \quad \forall \mathbf{w}_h, \bm{\theta}_h \in \mathbf{M}_h,
    \label{eq4.23}
\\
    \mathcal{A}_{\text{sip}}(g_h, \xi_h) &\leq C |g_h|_\text{dG} |\xi_h|_\text{dG}, \quad \forall g_h, \xi_h \in P_h. \label{eq4.24}
\end{align}
Recalling from \cite{di2010discrete,masri2022discontinuous}, we define for any $F \in \mathcal{F}_h$, the lift operator $r_F: (L^2(F))^d \to P_h$, by:

\begin{equation*}
    \int_\mathcal{O} r_F(\bm\varsigma) g_h = \int_F \{g_h\} \bm\varsigma \cdot \mathbf{n}_F, \quad \forall g_h \in P_h,
\end{equation*}
and for any
$F \in \mathcal{F}^i_h$, the lift operator $\mathbf{g}_F: L^2(F) \to \mathbf{M}_h$ by:
\begin{equation*}
    \int_\mathcal{O} \mathbf{g}_F(\varsigma) \cdot \mathbf{w}_h = \int_F \{\mathbf{w}_h\} \cdot \mathbf{n}_F \varsigma, \quad \forall \mathbf{w}_h \in \mathbf{M}_h.
\end{equation*}
Using the above definitions, we define the linear operators $R_h: \mathbf{M}_h \to P_h$ and $\mathbf{G}_h: P_h \to \mathbf{M}_h$ by
\begin{equation}
    R_h([\mathbf{w}_h]) = \sum_{F \in \mathcal{F}_h}r_F([\mathbf{w}_h]), \quad \mathbf{w}_h \in \mathbf{M}_h, \label{eqn20}
\end{equation}
\begin{equation}
    \mathbf{G}_h([\beta_h]) = \sum_{F \in \mathcal{F}_h^i}  \mathbf{g}_F([\beta_h]), \quad \beta_h \in P_h. \label{eqn21}
\end{equation}
Furthermore, these lift operators are bounded.
There exist constants $B_{r_2}, \tilde{B}_{r_1} > 0$, depending on the polynomial degrees $r_2$ and $r_1$ respectively, but independent of $h$ such that:
\begin{align}
    \| R_h([\mathbf{w}_h]) \| &\leq B_{r_2} \left( \sum_{F \in \mathcal{F}_h}h_F^{-1} \|[\mathbf{w}_h]\|_{L^2(F)}^2 \right)^{1/2}, \quad \forall \mathbf{w}_h \in \mathbf{M}_h, \label{eq4.29} \\
    \| \mathbf{G}_h([\beta_h]) \| &\leq \tilde{B}_{r_1} \left( \sum_{F \in \mathcal{F}_h^i}  h_F^{-1} \|[\beta_h]\|_{L^2(F)}^2 \right)^{1/2}, \quad \forall \beta_h \in P_h. \label{eq4.30}
\end{align}
Let $\nabla_h$ and $\nabla_h \cdot$ be the broken gradient and broken divergence operators, respectively. Then, with the definitions of the lift operators \eqref{eqn20} and \eqref{eqn21}, we can write \eqref{eq4.13} and \eqref{eq:4.24} in the equivalent forms given by:
\begin{align}
    \mathcal{b}(\mathbf{w}_h, g_h) &= (\nabla_h \cdot \mathbf{w}_h, g_h) - (R_h([\mathbf{w}_h]), g_h), \quad \forall \mathbf{w}_h \in \mathbf{M}_h, \ \forall g_h \in P_h, \label{eq4.31} \\
    \mathcal{b}(\mathbf{w}_h, g_h) &= -(\nabla_h g_h, \mathbf{w}_h) + (\mathbf{G}_h([g_h]), \mathbf{w}_h), \quad \forall \mathbf{w}_h \in \mathbf{M}_h, \ \forall g_h \in P_h. \label{eq4.32}
\end{align}
Accordingly, with the help of the Cauchy-Schwarz's inequality (the CS inequality) and \eqref{eq4.32}, it can be easily shown that
\begin{equation}
    \mathcal{b}(\mathbf{w}_h,g_h) \leq C\|\mathbf{w}_h\| |g_h|_\text{dG}, \quad \forall \mathbf{w}_h \in \mathbf{M}_h, \forall g_h \in P_h. \label{eq4.37}
\end{equation}

\section{Fully discrete scheme}
Since the backward Euler method is a first-order difference
scheme, the right rectangle rule is chosen here to approximate the integral term
\begin{equation}
    q_{r}^{n}(\bm\varphi) = \tau \sum_{j=1}^{n} \beta_{n-j} \bm\varphi^{j} \approx \int_{0}^{t_{n}} \beta(t_{n} - s) \bm\varphi(s) \, ds, \label{40}
\end{equation}
where \(\beta_{n-j} = \beta(t_{n} - t_{j})\). For $\bm\varphi \in C^1[0,t_n]$, the error associated with the rule \eqref{40} is given by
\begin{equation}
\left| \int_{0}^{t_n} \beta(t_n - s) \bm\varphi(s) ds - q^n_r (\bm\varphi) \right|
\leq C \tau \sum_{j=1}^{n} \int_{t_{j-1}}^{t_j} \left| \frac{\partial}{\partial s} (\beta(t_n - s) \bm\varphi(s)) \right| ds. \label{41}
\end{equation}
For large $\sigma$ and for arbitrary $\alpha_0 > 0$, the following positivity property holds for the symmetric form $\mathcal{A}_d(\cdot, \cdot)$ (cf. \cite{ray2024discontinuous}),
\begin{equation}
\tau \sum_{n=1}^m \tau \sum_{i=1}^n e^{-\alpha_0 (t_n - t_i)} \mathcal{A}_d(\bm\varphi_h^i, \bm\varphi_h^i) \geq 0, \quad \forall \bm\varphi_h \in \mathbf{M}_h. \label{5.38}
\end{equation}
Initially, we set $p_h^0 = v_h^0 = 0$. Let $\mathbf{u}_h^0$ be the local $L^2$ projection of $\mathbf{u}^0$ onto $\mathbf{M}_h$.
\begin{equation*}
    \int_K (\mathbf{u}_h^0 - \mathbf{u}^0) \cdot \mathbf{w}_h = 0, \quad \forall \mathbf{w}_h \in (\mathbb{P}_r(K))^d, \quad \forall K \in \mathcal{K}_h.
\end{equation*}
For $n = 1, \ldots, N$, given $(\mathbf{u}_h^{n-1}, p_h^{n-1}) \in \mathbf{M}_h \times P_h$, compute $\widetilde{\mathbf{u}}_h^n \in \mathbf{M}_h$ such that $\forall \mathbf{w}_h \in \mathbf{M}_h$,
\begin{equation}
\begin{aligned}
    (\widetilde{\mathbf{u}}_h^n, \mathbf{w}_h) + \tau \mathcal{A}_c(\mathbf{u}_h^{n-1};\mathbf{u}_h^{n-1}, \widetilde{\mathbf{u}}_h^n, \mathbf{w}_h) + \tau \mu \mathcal{A}_\epsilon(\widetilde{\mathbf{u}}_h^n, \mathbf{w}_h) + \tau \mathcal{A}_\epsilon(q_r^n (\widetilde{\mathbf{u}}_{h}),\mathbf{w}_h) \\
    =  (\mathbf{u}_h^{n-1}, \mathbf{w}_h) + \tau \mathcal{b}(\mathbf{w}_h, p_h^{n-1}) + \tau (\mathbf{f}^n, \mathbf{w}_h). \label{eq5.36}
\end{aligned}
\end{equation}
Next, compute $v_h^n \in P_h^0$ such that $\forall g_h \in P_h^0$,
\begin{equation}
    \mathcal{A}_{\text{sip}}(v_h^n, g_h) = -\frac{1}{\tau} \mathcal{b}(\widetilde{\mathbf{u}}_h^n, g_h). \label{eq:5.41}
\end{equation}
At the last step, compute $p_h^n \in P_h$ and $\mathbf{u}_h^n \in \mathbf{M}_h$ such that  $\forall g_h \in P_h$ and $\forall \mathbf{w}_h \in \mathbf{M}_h$,
\begin{equation}
    (p_h^n, g_h) = (p_h^{n-1}, g_h) + (v_h^n, g_h) - \delta \mu \mathcal{b}(\widetilde{\mathbf{u}}_h^n, g_h) - \delta \mathcal{b}(q_r^n (\widetilde{\mathbf{u}}_{h}), g_h), \label{eq5.38}
\end{equation}
\begin{equation}
    (\mathbf{u}_h^n, \mathbf{w}_h) = (\widetilde{\mathbf{u}}_h^n, \mathbf{w}_h) + \tau \mathcal{b}(\mathbf{w}_h, v_h^n). \label{eq5.39}
\end{equation}
Using \eqref{eq4.31} and \eqref{eq4.32}, equations \eqref{eq5.38} and \eqref{eq5.39} can be expressed as follows, $\forall g_h \in P_h$ and $\mathbf{w}_h \in \mathbf{M}_h$,
\begin{equation}
(p_h^n, g_h) = (p_h^{n-1}, g_h) + (v_h^n, g_h) - \delta \mu (\nabla_h \cdot \widetilde{\mathbf{u}}_h^n - R_h([\widetilde{\mathbf{u}}_h^n]), g_h) - \delta (\nabla_h \cdot q_r^n(\widetilde{\mathbf{u}}_h) - R_h([q_r^n(\widetilde{\mathbf{u}}_h)]), g_h), \label{eq5.40}
\end{equation}
\begin{equation}
(\mathbf{u}_h^n, \mathbf{w}_h) = (\widetilde{\mathbf{u}}_h^n, \mathbf{w}_h) - \tau (\nabla_h v_h^n - \mathbf{G}_h([v_h^n]), \mathbf{w}_h).
\label{eq5.41}
\end{equation}
\begin{lemma}
Let $p_h^n \in P_h$ be defined by \eqref{eq5.40}. Then $p_h^n \in P_h^0$, $\forall n \geq 0$.
\end{lemma}
\begin{proof}
We take the help of mathematical induction on $n$ to prove this lemma. The result holds trivially for $n = 0$. Assume that $p_h^{n-1} \in P_h^0$. Let $g_h = 1$ in \eqref{eq5.38} and use \eqref{eq4.32} to obtain:
\begin{equation*}
    \int_\mathcal{O} p_h^n = \int_\mathcal{O} p_h^{n-1} + \int_\mathcal{O} v_h^n.
\end{equation*}
We use the fact that $v_h^n \in P_h^0$ to conclude the result.
\end{proof}
\begin{lemma}
For the symmetric bilinear form $\mathcal{A}_d(\cdot,\cdot)$, for given $(\widetilde{\mathbf{u}}_h^{n-1}, \mathbf{u}_h^{n-1}, p_h^{n-1}) \in \mathbf{M}_h \times \mathbf{M}_h \times P_h^0$, there exists a unique solution $(\widetilde{\mathbf{u}}_h^n, \mathbf{u}_h^n, p_h^n) \in \mathbf{M}_h \times \mathbf{M}_h \times P_h^0$ to the fully discrete scheme given by \eqref{eq5.36}--\eqref{eq5.39}. Moreover, for non-symmetric form of $\mathcal{A}_\epsilon$, for $\tau$ small enough, the solution to the fully discrete scheme given by \eqref{eq5.36}--\eqref{eq5.39} is unique.
\end{lemma}
\begin{proof}
Since the problem is linear in finite dimensions, it is enough to prove the uniqueness of the solution. Assume that, there exist two solutions $\widetilde{\mathbf{u}}_{h,1}^n$ and $\widetilde{\mathbf{u}}_{h,2}^n$ to \eqref{eq5.36}. Set $\bm{\mathcal{X}}_h^n = \widetilde{\mathbf{u}}_{h,1}^n - \widetilde{\mathbf{u}}_{h,2}^n$.\\
Then, recalling the linearity of $\mathcal{A}_c$ in the third argument, and choosing $\mathbf{w}_h = \bm{\mathcal{X}}_h^n$, we obtain
\begin{equation*}
    (\bm{\mathcal{X}}_h^n, \bm{\mathcal{X}}_h^n) + \tau \mathcal{A}_c(\mathbf{u}_h^{n-1}; \mathbf{u}_h^{n-1}, \bm{\mathcal{X}}_h^n, \bm{\mathcal{X}}_h^n) + \tau \mu \mathcal{A}_\epsilon(\bm{\mathcal{X}}_h^n, \bm{\mathcal{X}}_h^n) + \tau \mathcal{A}_\epsilon(q_r^n(\bm{\mathcal{X}}_h),\bm{\mathcal{X}}_h^n) = 0.
\end{equation*}
Use of the positivity property of $\mathcal{A}_c$ \eqref{eq4.12} and the coercivity of $\mathcal{A}_\epsilon$ \eqref{eq3.20}, leads above to
\begin{equation}
    \|\bm{\mathcal{X}}_h^n\|^2 + \omega \tau \mu \|\bm{\mathcal{X}}_h^n\|_\text{dG}^2 + \tau \mathcal{A}_\epsilon(q_r^n(\bm{\mathcal{X}}_h),\bm{\mathcal{X}}_h^n) \leq 0. \label{51}
\end{equation}
We first prove the lemma for the symmetric form $\mathcal{A}_d(\cdot,\cdot)$. Summing over $n=1$ to $m$, we have
\begin{equation*}
\sum_{n=1}^{m} \|\bm{\mathcal{X}}_h^n\|^2 + \sum_{n=1}^{m} \omega \tau \mu \|\bm{\mathcal{X}}_h^n\|_\text{dG}^2 + \sum_{n=1}^{m} \tau \mathcal{A}_d(q_r^n(\bm{\mathcal{X}}_h),\bm{\mathcal{X}}_h^n) \leq 0.
\end{equation*}
Apply positivity property \eqref{5.38} to have
\begin{equation*}
\sum_{n=1}^{m} \|\bm{\mathcal{X}}_h^n\|^2 + \sum_{n=1}^{m} \omega \tau \mu \|\bm{\mathcal{X}}_h^n\|_\text{dG}^2 \leq 0.
\end{equation*}
Which implies,
\begin{equation*}
\|\bm{\mathcal{X}}_h^n\|^2 + \omega \tau \mu \|\bm{\mathcal{X}}_h^n\|_\text{dG}^2 \leq 0.
\end{equation*}
Thus, $\bm{\mathcal{X}}_h^n = 0$. \\
For the non-symmetric form of $\mathcal{A}_\epsilon$, we have used \eqref{eq4.23} and Young's inequality to have
\begin{equation}
\begin{aligned}
|\mathcal{A}_\epsilon(q_r^n(\bm{\mathcal{X}}_h), \bm{\mathcal{X}}_h^n)| &\leq C \tau \sum_{i=1}^n \beta(t_n-t_i) \|\bm{\mathcal{X}}_h^i\|_\text{dG} \|\bm{\mathcal{X}}_h^n\|_\text{dG} \\
&\leq \frac{\omega \mu}{2}\|\bm{\mathcal{X}}_h^n\|^2_\text{dG}+\tilde{C} \left(\tau \sum_{i=1}^n \beta(t_n-t_i) \|\bm{\mathcal{X}}_h^i\|_\text{dG} \right)^2 . \label{55}
\end{aligned}
\end{equation}
Utilizing \eqref{55} in \eqref{51} and summing from $n=1$ to $m$, we arrive at
\begin{equation}
\sum_{n=1}^{m} \|\bm{\mathcal{X}}_h^n\|^2 + \frac{\omega \tau \mu}{2} \sum_{n=1}^{m}  \|\bm{\mathcal{X}}_h^n\|_\text{dG}^2 \leq \tilde{C} \tau \sum_{n=1}^m \left(\tau \sum_{i=1}^n \beta(t_n-t_i) \|\bm{\mathcal{X}}_h^i\|_\text{dG} \right)^2. \label{56}
\end{equation}
By Holder's inequality, the term on the right hand side (RHS) of \eqref{56} is bounded as
\begin{equation}
    \begin{aligned}
        \tau \sum_{n=1}^m \left( \tau \sum_{i=1}^n \beta(t_n-t_i) \|\bm{\mathcal{X}}_h^i\|_\text{dG} \right)^2 &\leq \gamma^2 \tau \sum_{n=1}^m \left(\tau \sum_{i=1}^n e^{-2\eta(t_n-t_i)}\right) \left(\tau \sum_{i=1}^n \|\bm{\mathcal{X}}_h^i\|^2_\text{dG}\right)\\
        & \leq \frac{\gamma^2 e^{2\eta \tau}}{2\eta} \tau^2 \sum_{n=1}^m \sum_{i=1}^n \|\bm{\mathcal{X}}_h^i\|^2_\text{dG} . \label{eqn58}
    \end{aligned}
\end{equation}
Using \eqref{eqn58} and applying the discrete Gronwall's inequality, \eqref{56} yields:
\begin{equation*}
\sum_{n=1}^{m} \|\bm{\mathcal{X}}_h^n\|^2 + \frac{\omega \tau \mu}{2} \sum_{n=1}^{m}  \|\bm{\mathcal{X}}_h^n\|_\text{dG}^2 \leq 0,
\end{equation*}
which yields $\bm{\mathcal{X}}_h^n = 0$.
Hence, \eqref{eq5.36} has a unique solution.
Following similar arguments, using the coercivity property of $\mathcal{A}_{\text{sip}}(\cdot, \cdot)$, and the fact that $\|\cdot\|_\text{dG}$ is a norm for the space $P_h^0$, one can show the existence of $v_h^n \in P_h^0$. Using \eqref{eq5.40} and Lemma 1, we have the existence of $p_h^n \in P_h^0$. Finally, the existence of $\mathbf{u}_h^n \in \mathbf{M}_h$ follows from \eqref{eq5.41}.
\end{proof}
\begin{lemma}
Let $(\mathbf{u}, p)$ be the solution of \eqref{eq3.1}-\eqref{eq3.4}. Then $(\mathbf{u},p)$ satisfies,
\begin{equation*}
(\partial_t \mathbf{u}, \mathbf{w}) + \mathcal{A}_c(\mathbf{u}; \mathbf{u}, \mathbf{u}, \mathbf{w}) + \mu \mathcal{A}_\epsilon(\mathbf{u}, \mathbf{w}) + \int_0^t \beta(t-s) \mathcal{A}_\epsilon(\mathbf{u}(s), \mathbf{w}) ds = \mathcal{b}(\mathbf{w}, p) + (\mathbf{f}, \mathbf{w}), \quad \forall \mathbf{w} \in \mathbf{M}.
\end{equation*}
\end{lemma}
\begin{proof}
Multiply \eqref{eq3.1} by $\mathbf{w} \in \mathbf{M}$ and integrate over each mesh element $K \in \mathcal{K}_h$. Then, using Green’s formula and summing over all elements $K$, we obtain
\begin{align*}
\sum_{K \in \mathcal{K}_h} \int_K \mathbf{u}_t \cdot \mathbf{w} 
+ \mu \sum_{K \in \mathcal{K}_h} \int_K \nabla \mathbf{u} \cdot \nabla \mathbf{w}
- \mu \sum_{F \in \mathcal{F}_h} \int_F [\nabla \mathbf{u} \mathbf{n}_F \cdot \mathbf{w}] 
+ \sum_{K \in \mathcal{K}_h} \int_K \mathbf{u} \cdot \nabla \mathbf{u} \cdot \mathbf{w}  \\
+ \int_0^t \beta(t-s) \left( \sum_{K \in \mathcal{K}_h} \int_K \nabla \mathbf{u}(s) \cdot \nabla \mathbf{w} 
- \sum_{F \in \mathcal{F}_h} \int_F [\nabla \mathbf{u}(s) \mathbf{n}_F \cdot \mathbf{w}] \right) ds \\
- \sum_{K \in \mathcal{K}_h} \int_K p \nabla \cdot v 
+ \sum_{F \in \mathcal{F}_h} \int_F [p \mathbf{w} \cdot \mathbf{n}_F] = \int_\mathcal{O} \mathbf{f} \cdot \mathbf{w}.
\end{align*}
Note that, $[\nabla \mathbf{u} \mathbf{n}_F \cdot \mathbf{w}] = \{ \nabla \mathbf{u} \} \mathbf{n}_F \cdot [\mathbf{w}] + [\nabla \mathbf{u}] \mathbf{n}_F \cdot \{\mathbf{w}\}$. Since, $[\mathbf{u}]=\mathbf{0}, \forall F \in \mathcal{F}_h$ and $[\nabla \mathbf{u}]\cdot \mathbf{n}_F=\mathbf{0}, \forall F \in \mathcal{F}_h^i$, we have
\begin{align*}
&\mu \sum_{K \in \mathcal{K}_h} \int_K \nabla \mathbf{u} \cdot \nabla \mathbf{w}
- \mu \sum_{F \in \mathcal{F}_h} \int_F [\nabla \mathbf{u} \mathbf{n}_F \cdot \mathbf{w}] \\
&= \mu \sum_{K \in \mathcal{K}_h} \int_K \nabla \mathbf{u} \cdot \nabla \mathbf{w}
- \mu \sum_{F \in \mathcal{F}_h} \int_F \{\nabla \mathbf{u}\} \mathbf{n}_F \cdot [\mathbf{w}] = \mu \mathcal{A}_\epsilon(\mathbf{u}, \mathbf{w}).
\end{align*}
An application of similar arguments as above yields
\begin{align*}
\int_0^t \beta(t-s) \left( \sum_{K \in \mathcal{K}_h} \int_K \nabla \mathbf{u}(s) \cdot \nabla \mathbf{w}
- \sum_{F \in \mathcal{F}_h} \int_F [\nabla \mathbf{u}(s) \mathbf{n}_F \cdot \mathbf{w}] \right) ds
= \int_0^t \beta(t-s) \mathcal{A}_\epsilon(\mathbf{u}(s), \mathbf{w}) ds.
\end{align*}
With the incompressibility condition \eqref{eq3.2}, and using $[\mathbf{u}]=\mathbf{0}, \forall F \in \mathcal{F}_h$, $[\nabla \mathbf{u}]=0, \forall F\in \mathcal{F}_h^i$ and $[p]=0, \forall F\in \mathcal{F}_h^i$, one can find
\[
\sum_{K \in \mathcal{K}_h} \int_K \mathbf{u} \cdot \nabla \mathbf{u} \cdot \mathbf{w} = \mathcal{A}_C(\mathbf{u};\mathbf{u}, \mathbf{u}, \mathbf{w})
\]
and
\[
- \sum_{K \in \mathcal{K}_h} \int_K p \nabla \cdot \mathbf{w}
+ \sum_{F \in \mathcal{F}_h} \int_F [p \mathbf{w} \cdot \mathbf{n}_F] = - \mathcal{b}(\mathbf{w}, p).
\]
This completes the proof of this lemma.
\end{proof}

\section{Stability}
We recall from \cite{masri2022discontinuous}, the following identities which will be used in sequel.
\begin{lemma}
Let $(\mathbf{u}_h^n, \widetilde{\mathbf{u}}_h^n, v_h^n) \in \mathbf{M}_h \times \mathbf{M}_h \times P_h^0$ satisfies \eqref{eq5.36}--\eqref{eq5.39}, then  
for all $g_h \in P_h$ and $n \geq 1$, we have
\begin{equation}
    \mathcal{b}(\mathbf{u}_h^n, g_h) = \mathcal{b}(\widetilde{\mathbf{u}}_h^n, g_h) + \tau \mathcal{A}_{\text{sip}}(v_h^n, g_h)
    - \tau \sum_{F \in \mathcal{F}_h^i}  \frac{\tilde{\sigma}}{h_F} \int_F [v_h^n][g_h] 
    + \tau (\mathbf{G}_h([v_h^n]), \mathbf{G}_h([g_h])), \label{eq6.52}
\end{equation}
\begin{equation}
    \mathcal{b}(\mathbf{u}_h^n, g_h) = -\tau \sum_{F \in \mathcal{F}_h^i}  \frac{\tilde{\sigma}}{h_F} \int_F [v_h^n][g_h] 
    + \tau (\mathbf{G}_h([v_h^n]), \mathbf{G}_h([g_h])). \label{eq6.53}
\end{equation}
\end{lemma}
\begin{lemma} [\cite{girault2005discontinuous,lasis2003poincare}]
There exists a constant $C_k$ that depends on $k$, but independent of $h$ and $\tau$ such that
\begin{equation}
    \|\mathbf{w}\|_{L^k(\mathcal{O})} \leq C_k \|\mathbf{w}\|_\text{dG}, \quad \forall \mathbf{w} \in \mathbf{M}, \label{eq6.54}
\end{equation}
where $2 \leq k < \infty$ for $d = 2$ and $2 \leq k \leq 6$ for $d = 3$.
\end{lemma}
In order to establish the stability result, we carefully define the auxiliary functions $S_h^n \in P_h$ and $\xi_h^n \in P_h$, for $n \geq 0$:
\begin{equation}
S_h^0 = 0, \quad S_h^n = \delta \mu \sum_{i=1}^n (\nabla_h \cdot \widetilde{\mathbf{u}}_h^i - R_h([\widetilde{\mathbf{u}}_h^i])) + \delta \sum_{i=1}^n \sum_{j=1}^i \beta_{i-j} (\nabla_h \cdot \widetilde{\mathbf{u}}_h^j - R_h([\widetilde{\mathbf{u}}_h^j])), \quad n \geq 1, \label{eq6.55}
\end{equation}
\begin{equation}
\xi_h^0 = 0, \quad \xi_h^n = p_h^n + S_h^n, \quad n \geq 1. \label{eq6.56}
\end{equation}
Using the equivalent definition \eqref{eq4.31}, we obtain
\begin{equation*}
    \int_\mathcal{O} S_h^n = \delta \mu \sum_{i=1}^n \mathcal{b}(\widetilde{\mathbf{u}}_h^i, 1) + \delta \sum_{i=1}^n \sum_{j=1}^i \beta_{i-j} \mathcal{b}(\widetilde{\mathbf{u}}_h^j,1) = 0, \quad n \geq 1,
\end{equation*}
and this implies $S_h^n \in P_h^0$.
Hence, $\xi_h^n \in P_h^0$ as $p_h^n \in P_h^0$, by Lemma 1.

We now make the following assumptions, which will be used in sequel.

(\textbf{A1}) The penalty parameters $\sigma$ and $\tilde{\sigma}$ are large enough, specifically $\sigma \geq B_{r_2}^2 / d$ and $\tilde{\sigma} \geq 4 \tilde{B}_{r_1}^2$. \\Furthermore, the positive parameter $\delta \leq \min\left(\frac{\omega}{8d}, \frac{\omega \mu^2}{16 \gamma^2 d}\right)$.

(\textbf{A2}) The exact solution satisfies the following regularity assumptions:
\begin{align*}
&\mathbf{u} \in L^\infty(0, T; (H^{r+1}(\mathcal{O}))^d) \cap L^2(0, T; (H^{r+1}(\mathcal{O}))^d), \quad 
\partial_t \mathbf{u} \in L^2(0, T; (H^{r+1}(\mathcal{O}))^d), \\
&\partial_{tt} \mathbf{u} \in L^2(0, T; (L^2(\mathcal{O}))^d), \quad
 p \in L^\infty(0, T; H^r(\mathcal{O})).
\end{align*}

(\textbf{A3}) The exact solution satisfies the following regularity assumption:
\begin{align*}
\partial_t p \in L^2(0,T; H^1(\mathcal{O})). 
\end{align*}

(\textbf{A4}) Fixing $0 < \upsilon \leq 1$, there exist positive constants $c_1, c_2$ such that $\tau$ satisfies
\begin{equation}
    c_1 h^2 \leq \tau \leq c_2 h^{(1+\upsilon)d/3}. \label{8.223}
\end{equation}
With these preparations, we now state and prove the main result of this section.
\begin{theorem}
Let assumption (\textbf{A1}) be satisfied, and $\mathbf{u}^0 \in (L^2(\mathcal{O}))^d$. If $\tau$ is small enough, the discrete solution to the pressure correction scheme \eqref{eq5.36}--\eqref{eq5.39} satisfies for all $\tau > 0$ and $1 \leq m \leq N$,
\begin{equation*}
    \norm{\mathbf{u}_h^m}^2 + \frac{\omega \mu}{2} \tau \sum_{n=1}^m \norm{\widetilde{\mathbf{u}}_h^n}_\text{dG}^2 
    + \frac{1}{2} \tau^2 \sum_{n=1}^m |\xi_h^m|_\text{dG}^2 
    + \frac{1}{\delta \mu} \tau \norm{S_h^m}^2 \\
    \leq \tilde{C}_T\left(\norm{\mathbf{u}^0}^2 + \frac{4C_2^2}{\omega \mu} \tau \sum_{n=1}^m \norm{\mathbf{f}^n}^2\right).
\end{equation*}
\end{theorem}
\begin{proof}
Set $\mathbf{w}_h = \widetilde{\mathbf{u}}_h^n$ in \eqref{eq5.36}. With the positivity of $\mathcal{A}_c$ \eqref{eq4.12} and the coercivity of $\mathcal{A}_\epsilon$ \eqref{eq3.20}, we get
\begin{equation}
\begin{aligned}
    \frac{1}{2} \left( \norm{\widetilde{\mathbf{u}}_h^n}^2 - \norm{\mathbf{u}_h^{n-1}}^2 + \norm{\widetilde{\mathbf{u}}_h^n - \mathbf{u}_h^{n-1}}^2 \right)
    + \omega \mu \tau \norm{\widetilde{\mathbf{u}}_h^n}_\text{dG}^2 + \tau \mathcal{A}_\epsilon(q_r^n(\widetilde{\mathbf{u}}_h),\widetilde{\mathbf{u}}_h^n) \\
    \leq \tau \mathcal{b}(\widetilde{\mathbf{u}}_h^n, p_h^{n-1}) + \tau (\mathbf{f}^n, \widetilde{\mathbf{u}}_h^n).
    \label{eq6.59}
\end{aligned}
\end{equation}
Next, set $\mathbf{w}_h = \mathbf{u}_h^n$ in \eqref{eq5.39} to have
\begin{equation}
    \frac{1}{2} \left( \norm{\mathbf{u}_h^n}^2 - \norm{\widetilde{\mathbf{u}}_h^n}^2 + \norm{\mathbf{u}_h^n - \widetilde{\mathbf{u}}_h^n}^2 \right)
    = \tau \mathcal{b}(\mathbf{u}_h^n, v_h^n). \label{eq6.60}
\end{equation}
Set $\mathbf{w}_h = \mathbf{u}_h^n - \widetilde{\mathbf{u}}_h^n$ in \eqref{eq5.39} to obtain
\begin{equation}
    \norm{\mathbf{u}_h^n - \widetilde{\mathbf{u}}_h^n}^2 = \tau \mathcal{b}(\mathbf{u}_h^n - \widetilde{\mathbf{u}}_h^n, v_h^n). \label{eqn6.71}
\end{equation}
Using equation \eqref{eq6.52}, we obtain
\begin{equation}
    \norm{\mathbf{u}_h^n - \widetilde{\mathbf{u}}_h^n}^2 
    = \tau^2 \mathcal{A}_{\text{sip}}(v_h^n, v_h^n) 
    - \tau^2 \sum_{F \in \mathcal{F}_h^i}  \frac{\tilde{\sigma}}{h_F} \norm{[v_h^n]}^2_{L^2(F)} 
    + \tau^2 \norm{\mathbf{G}_h([v_h^n])}^2.
\end{equation}
Again using equation \eqref{eq6.53}, we find
\begin{equation}
    \mathcal{b}(\mathbf{u}_h^n, v_h^n) = - \tau \sum_{F \in \mathcal{F}_h^i}  \frac{\tilde{\sigma}}{h_F} \norm{[v_h^n]}^2_{L^2(F)} + \tau \norm{\mathbf{G}_h([v_h^n])}^2. \label{eqn6.73}
\end{equation}
With the help of \eqref{eqn6.71}-\eqref{eqn6.73}, equation \eqref{eq6.60} yields
\begin{equation}
    \frac{1}{2} \left( \norm{\mathbf{u}_h^n}^2 - \norm{\widetilde{\mathbf{u}}_h^n}^2 \right) 
    + \frac{\tau^2}{2} \mathcal{A}_{\text{sip}}(v_h^n, v_h^n) \\
    + \frac{\tau^2}{2} \sum_{F \in \mathcal{F}_h^i}  \frac{\tilde{\sigma}}{h_F} \norm{[v_h^n]}^2_{L^2(F)} 
    = \frac{\tau^2}{2} \norm{\mathbf{G}_h([v_h^n])}^2. \label{eq6.66}
\end{equation}
The last term in \eqref{eq6.66} is estimated, using \eqref{eq4.30}, as
\begin{equation}
    \norm{\mathbf{G}_h([v_h^n])} \leq \tilde{B}_{r_1} \left( \sum_{F \in \mathcal{F}_h^i}  h_F^{-1} \norm{[v_h^n]}_{L^2(F)}^2 \right)^{1/2}. \label{eqn6.75}
\end{equation}
Using \eqref{eqn6.75} in \eqref{eq6.66} to have
\begin{equation*}
    \frac{1}{2} \left( \norm{\mathbf{u}_h^n}^2 - \norm{\widetilde{\mathbf{u}}_h^n}^2 \right) 
    + \frac{\tau^2}{2} \mathcal{A}_{\text{sip}}(v_h^n, v_h^n) \\
    + \frac{\tau^2}{2} \sum_{F \in \mathcal{F}_h^i}  \left( \tilde{\sigma} - \tilde{B}_{r_1}^2 \right) h_F^{-1} \norm{[v_h^n]}_{L^2(F)}^2 \leq 0.
\end{equation*}
Invoking the assumption, $\tilde{\sigma} \geq \tilde{B}_{r_1}^2$, we obtain
\begin{equation}
    \frac{1}{2} \norm{\mathbf{u}_h^n}^2 - \frac{1}{2} \norm{\widetilde{\mathbf{u}}_h^n}^2 
    + \frac{\tau^2}{2} \mathcal{A}_{\text{sip}}(v_h^n, v_h^n) \leq 0. \label{eq6.67}
\end{equation}
Adding \eqref{eq6.67} to \eqref{eq6.59} yields
\begin{equation}
\begin{aligned}
    \frac{1}{2} \left( \norm{\mathbf{u}_h^n}^2 - \norm{\mathbf{u}_h^{n-1}}^2 - \norm{\widetilde{\mathbf{u}}_h^n - \mathbf{u}_h^{n-1}}^2 \right) 
    + \omega \mu \tau \norm{\widetilde{\mathbf{u}}_h^n}_\text{dG}^2 
    + \frac{\tau^2}{2} \mathcal{A}_{\text{sip}}(v_h^n, v_h^n) + \tau \mathcal{A}_\epsilon(q_r^n(\widetilde{\mathbf{u}}_h),\widetilde{\mathbf{u}}_h^n) \\
    \leq \tau \mathcal{b}(\widetilde{\mathbf{u}}_h^n, p_h^{n-1}) + \tau (\mathbf{f}^n, \widetilde{\mathbf{u}}_h^n).
    \label{eq6.68}
\end{aligned}
\end{equation}
Now we utilize the auxiliary function $S_h^n$ defined in \eqref{eq6.55} to tackle the first term on the RHS of \eqref{eq6.68} and obtain
\begin{equation}
    \mathcal{b}(\widetilde{\mathbf{u}}_h^n, p_h^{n-1}) = \mathcal{b}(\widetilde{\mathbf{u}}_h^n, p_h^{n-1} + S_h^n) - \mathcal{b}(\widetilde{\mathbf{u}}_h^n, S_h^n). \label{eq6.69}
\end{equation}
Now we use the equivalent expression \eqref{eq4.31} to rewrite the second term in \eqref{eq6.69} as
\begin{equation}
\begin{aligned}
    \mathcal{b}(\widetilde{\mathbf{u}}_h^n, S_h^{n-1}) 
    &= \left(\nabla_h \cdot \widetilde{\mathbf{u}}_h^n - R_h([\widetilde{\mathbf{u}}_h^n]), S_h^{n-1}\right) \\
    &= \frac{1}{\delta \mu} (S_h^n - S_h^{n-1},S_h^{n-1}) - \frac{1}{ \mu} \left(\sum_{j=1}^n \beta_{i-j} (\nabla_h \cdot \widetilde{\mathbf{u}}_h^j - R_h([\widetilde{\mathbf{u}}_h^j])),S_h^{n-1}\right) \\
    &= \frac{1}{2 \delta \mu} \left( \norm{S_h^n}^2 - \norm{S_h^{n-1}}^2 + \norm{S_h^n - S_h^{n-1}}^2 \right) \\
    &\hspace{12em} - \frac{1}{ \mu} \left(\nabla_h \cdot q_r^n(\widetilde{\mathbf{u}}_h)-R_h([q_r^n(\widetilde{\mathbf{u}}_h)]),S_h^{n-1}\right). \label{eq6.70}
\end{aligned}
\end{equation}
With the definition of $\xi_h^n$ in \eqref{eq6.56}, \eqref{eq5.38} implies
\begin{equation}
    \xi_h^n - \xi_h^{n-1} = v_h^n, \quad \forall n \geq 1. \label{eq6.71}
\end{equation}
Note that, $S_h^n, p_h^n \in P_h^0$. Hence by applying \eqref{eq:5.41} with \eqref{eq6.71}, we obtain
\begin{equation*}
    \mathcal{b}(\widetilde{\mathbf{u}}_h^n, \xi_h^{n-1}) 
    = - \tau \mathcal{A}_{\text{sip}}(v_h^n, \xi_h^{n-1}) 
    = - \mathcal{A}_{\text{sip}}(\xi_h^n - \xi_h^{n-1}, \xi_h^{n-1}).
\end{equation*}
Since $\mathcal{A}_{\text{sip}}(\cdot, \cdot)$ is symmetric, we have
\begin{equation}
\begin{aligned}
    \mathcal{b}(\widetilde{\mathbf{u}}_h^n, \xi_h^{n-1}) &= -\frac{\tau}{2} \left( 
    \mathcal{A}_{\text{sip}}(\xi_h^n, \xi_h^n) 
    - \mathcal{A}_{\text{sip}}(\xi_h^{n-1}, \xi_h^{n-1}) 
    - \mathcal{A}_{\text{sip}}(\xi_h^n - \xi_h^{n-1}, \xi_h^n - \xi_h^{n-1}) \right) \\
    &= -\frac{\tau}{2} ( 
    \mathcal{A}_{\text{sip}}(\xi_h^n, \xi_h^n) 
    - \mathcal{A}_{\text{sip}}(\xi_h^{n-1}, \xi_h^{n-1}) 
    - \mathcal{A}_{\text{sip}}(v_h^n, v_h^n) ). \label{eq6.73}
\end{aligned}
\end{equation}
Substituting \eqref{eq6.70} and \eqref{eq6.73} into \eqref{eq6.69} yields
\begin{equation}
\begin{aligned}
    \mathcal{b}(\widetilde{\mathbf{u}}_h^n, p_h^{n-1}) = -\frac{\tau}{2} &\left(
    \mathcal{A}_{\text{sip}}(\xi_h^n, \xi_h^n) 
    - \mathcal{A}_{\text{sip}}(\xi_h^{n-1}, \xi_h^{n-1}) 
    - \mathcal{A}_{\text{sip}}(v_h^n, v_h^n) \right) \\
    &- \frac{1}{2 \delta \mu} \left(
    \norm{S_h^n}^2 - \norm{S_h^{n-1}}^2 + \norm{S_h^n - S_h^{n-1}}^2 \right) \\
    &+ \frac{1}{ \mu} \left(\nabla_h \cdot q_r^n(\widetilde{\mathbf{u}}_h)-R_h([q_r^n(\widetilde{\mathbf{u}}_h)]),S_h^{n-1}\right). \label{eq:6.75}
\end{aligned}
\end{equation}
With the expression \eqref{eq:6.75} above, \eqref{eq6.68} becomes
\begin{equation}
\begin{aligned}
    &\frac{1}{2} \left( \norm{\mathbf{u}_h^n}^2 - \norm{\mathbf{u}_h^{n-1}}^2 \right) 
    + \norm{\widetilde{\mathbf{u}}_h^n - \mathbf{u}_h^{n-1}}^2 )
    + \omega \mu \tau \norm{\widetilde{\mathbf{u}}_h^n}_\text{dG}^2 \\
    &+ \frac{\tau^2}{2} (\mathcal{A}_{\text{sip}}(\xi_h^n, \xi_h^n) 
    - \mathcal{A}_{\text{sip}}(\xi_h^{n-1}, \xi_h^{n-1})) 
    + \frac{\tau}{2 \delta \mu} \left(
    \norm{S_h^n}^2 - \norm{S_h^{n-1}}^2 \right) + \tau \mathcal{A}_\epsilon(q_r^n(\widetilde{\mathbf{u}}_h),\widetilde{\mathbf{u}}_h^n) \\
    &\leq \tau (\mathbf{f}^n, \widetilde{\mathbf{u}}_h^n) + \frac{\tau}{2 \delta \mu} \norm{S_h^n - S_h^{n-1}}^2 + \frac{\tau}{ \mu} \left(\nabla_h \cdot q_r^n(\widetilde{\mathbf{u}}_h)-R_h([q_r^n(\widetilde{\mathbf{u}}_h)]),S_h^{n-1}\right). \label{eq6.75}
\end{aligned}
\end{equation}
Now, we handle the second term on the RHS of \eqref{eq6.75}. Recall that from the definition of $S_h^n$, we have
\begin{equation}
    S_h^n - S_h^{n-1} = \delta \mu \left( \nabla_h \cdot \widetilde{\mathbf{u}}_h^n - R_h(\widetilde{\mathbf{u}}_h^n) \right) + \delta \left(\nabla_h \cdot q_r^n(\widetilde{\mathbf{u}}_h)-R_h([q_r^n(\widetilde{\mathbf{u}}_h)])\right). \label{6.76}
\end{equation}
From \eqref{6.76}, we write using triangle inequality
\begin{equation}
\|S^n_h - S^{n-1}_h\|^2 \leq 2 \delta^2 \mu^2 \|\nabla_h \cdot \widetilde{\mathbf{u}}^n_h - R_h([\widetilde{\mathbf{u}}^n_h])\|^2 + 2 \delta^2 \norm{\nabla_h \cdot q_r^n(\widetilde{\mathbf{u}}_h)-R_h([q_r^n(\widetilde{\mathbf{u}}_h)])}^2. \label{6.77}
\end{equation}
Using definition of $\beta(t)$, we obtain
\begin{equation}
\norm{\nabla_h \cdot q_r^n(\widetilde{\mathbf{u}}_h)-R_h([q_r^n(\widetilde{\mathbf{u}}_h)])}^2 \leq \tau^2 \gamma^2  \sum_{j=1}^n \norm{\nabla_h \cdot \widetilde{\mathbf{u}}_h^j - R_h(\widetilde{\mathbf{u}}_h^j)}^2 . \label{6.78}
\end{equation}
Using \eqref{eq4.29}, it follows that
\begin{equation}
\begin{aligned}
\norm{\nabla_h \cdot \widetilde{\mathbf{u}}_h^n - R_h(\widetilde{\mathbf{u}}_h^n)}^2
&\leq 2 ( \|\nabla_h \cdot \widetilde{\mathbf{u}}^n_h\|^2 + \norm{R_h(\widetilde{\mathbf{u}}_h^n)}^2) \\
&\leq 2 \left(d\norm{\nabla_h \widetilde{\mathbf{u}}^n_h}^2+ B_{r_2}^2 \sum_{F \in \mathcal{F}_h}h_F^{-1} \|[\widetilde{\mathbf{u}}^n_h]\|^2_{L^2(F)}\right). \label{6.79}
\end{aligned}
\end{equation}
Using \eqref{6.78}, \eqref{6.79}, and the bound $\|\nabla_h \cdot \mathbf{w}\| \leq d^{1/2} \|\nabla_h \mathbf{w}\|, \forall \mathbf{w} \in \mathbf{M}$, inequality \eqref{6.77} yields
\begin{equation}
\begin{aligned}
\|S^n_h - S^{n-1}_h\|^2 &\leq 2 \delta^2 \mu^2 \|\nabla_h \cdot \widetilde{\mathbf{u}}^n_h - R_h([\widetilde{\mathbf{u}}^n_h])\|^2 + 2 \delta^2 \gamma^2 \tau^2 \sum_{j=1}^n \norm{\nabla_h \cdot \widetilde{\mathbf{u}}_h^j - R_h(\widetilde{\mathbf{u}}_h^j)}^2\\
&\leq 4 \delta^2 \mu^2 \left(d\norm{\nabla_h \widetilde{\mathbf{u}}^n_h}^2+ B_{r_2}^2 \sum_{F \in \mathcal{F}_h}h_F^{-1} \|[\widetilde{\mathbf{u}}^n_h]\|^2_{L^2(F)}\right) \\
&\hspace{2em} +4 \delta^2 \gamma^2 \tau^2 \sum_{j=1}^n \left(d\norm{\nabla_h \widetilde{\mathbf{u}}^j_h}^2+ B_{r_2}^2 \sum_{F \in \mathcal{F}_h}h_F^{-1} \|[\widetilde{\mathbf{u}}^j_h]\|^2_{L^2(F)}\right).  \label{6.80}
\end{aligned}
\end{equation}
Using the assumptions that $\delta \leq \min\left(\frac{\omega}{4d}, \frac{\omega \mu^2}{16 \gamma^2 d}\right)$ and $\sigma \geq \frac{B_{r_2}^2}{d}$, \eqref{6.80} leads to
\begin{equation}
\begin{aligned}
\frac{1}{2\delta\mu}\|S^n_h - S^{n-1}_h\|^2 
\leq \frac{\omega \mu}{2}&\left(\norm{\nabla_h \widetilde{\mathbf{u}}^n_h}^2+ \sum_{F \in \mathcal{F}_h}\sigma h_F^{-1} \|[\widetilde{\mathbf{u}}^n_h]\|^2_{L^2(F)}\right) \\
&+\frac{\omega \mu}{8} \tau^2 \sum_{j=1}^n \left(\norm{\nabla_h \widetilde{\mathbf{u}}^j_h}^2+ \sum_{F \in \mathcal{F}_h}\sigma h_F^{-1} \|[\widetilde{\mathbf{u}}^j_h]\|^2_{L^2(F)}\right). \label{6.81}
\end{aligned}
\end{equation}
To handle the third term on the RHS of \eqref{eq6.75}, we use the CS inequality and Young's inequality to obtain
\begin{equation}
\begin{aligned}
&\left |\frac{\tau}{\mu}\left(\nabla_h \cdot q_r^n(\widetilde{\mathbf{u}}_h)-R_h([q_r^n(\widetilde{\mathbf{u}}_h)]),S_h^{n-1}\right) \right | \\
&\leq \frac{\tau}{\mu}\norm{\nabla_h \cdot q_r^n(\widetilde{\mathbf{u}}_h)-R_h([q_r^n(\widetilde{\mathbf{u}}_h)])} \norm{S_h^{n-1}}\\
&\leq \frac{\delta}{2 \mu}\norm{\nabla_h \cdot q_r^n(\widetilde{\mathbf{u}}_h)-R_h([q_r^n(\widetilde{\mathbf{u}}_h)])}^2 + \frac{\tau^2}{2 \delta \mu} \norm{S_h^{n-1}}^2 \\
&\leq \frac{2 \gamma^2 \delta}{ \mu} \tau^2 \sum_{j=1}^n \left(d\norm{\nabla_h \widetilde{\mathbf{u}}^j_h}^2+ B_{r_2}^2 \sum_{F \in \mathcal{F}_h}h_F^{-1} \|[\widetilde{\mathbf{u}}^j_h]\|^2_{L^2(F)}\right) + \frac{\tau^2}{2 \delta \mu} \norm{S_h^{n-1}}^2. \label{6.82}
\end{aligned}
\end{equation}
Using the assumptions that $\delta \leq \frac{\omega \mu^2}{16 \gamma^2 d}$ and $\sigma \geq \frac{B_{r_2}^2}{d}$, \eqref{6.82} yields
\begin{equation}
\begin{aligned}
&\left |\frac{\tau}{\mu}\left(\nabla_h \cdot q_r^n(\widetilde{\mathbf{u}}_h)-R_h([q_r^n(\widetilde{\mathbf{u}}_h)]),S_h^{n-1}\right) \right | \\
&\leq \frac{\omega \mu }{8} \tau^2 \sum_{j=1}^n\left(\norm{\nabla_h \widetilde{\mathbf{u}}^j_h}^2+ \sum_{F \in \mathcal{F}_h}\sigma h_F^{-1} \|[\widetilde{\mathbf{u}}^j_h]\|^2_{L^2(F)}\right) + \frac{\tau^2}{2 \delta \mu} \norm{S_h^{n-1}}^2. \label{6.83}
\end{aligned}
\end{equation}
Substituting the bounds \eqref{6.81} and \eqref{6.83} in \eqref{eq6.75}, we have
\begin{equation}
\begin{aligned}
\frac{1}{2} \left(\|\mathbf{u}^n_h\|^2 - \|\mathbf{u}^{n-1}_h\|^2 + \|\widetilde{\mathbf{u}}^n_h - \mathbf{u}^{n-1}_h\|^2 \right) + \frac{\omega \mu}{2} \tau \|\widetilde{\mathbf{u}}^n_h\|^2_\text{dG} 
+ \frac{\tau^2}{2} \left(\mathcal{A}_{\text{sip}}(\xi^n_h, \xi^n_h) - \mathcal{A}_{\text{sip}}(\xi^{n-1}_h, \xi^{n-1}_h)\right) \\ 
+ \frac{\tau}{2\delta \mu} \left(\|S^n_h\|^2 - \|S^{n-1}_h\|^2\right) + \tau \mathcal{A}_\epsilon(q_r^n(\widetilde{\mathbf{u}}_h),\widetilde{\mathbf{u}}_h^n)
\leq \tau \left(\mathbf{f}^n, \widetilde{\mathbf{u}}^n_h\right) + \frac{\tau^2}{2 \delta \mu} \norm{S_h^{n-1}}^2+\frac{\omega \mu}{4} \tau^2 \sum_{i=1}^n \|\widetilde{\mathbf{u}}^i_h\|^2_\text{dG}. \label{eq6.84}
\end{aligned}
\end{equation}
First consider the symmetric case $\mathcal{A}_d(\cdot,\cdot)$.
Using the CS inequality, bound \eqref{eq6.54}, and then Young's inequality, we obtain
\begin{equation}
\left|(\mathbf{f}^n, \widetilde{\mathbf{u}}^n_h)\right| \leq \frac{C^2_2}{\omega \mu} \|\mathbf{f}^n\|^2 + \frac{\omega \mu}{4} \|\widetilde{\mathbf{u}}^n_h\|^2_\text{dG}. \label{6.85}
\end{equation}
Substituting \eqref{6.85} in \eqref{eq6.84} and multiplying by 2, we sum the inequality over $n=1$ to $m$ to arrive at
\begin{equation}
\begin{aligned}
&\|\mathbf{u}^m_h\|^2 + \frac{\omega \mu}{2} \tau \sum_{n=1}^m \|\widetilde{\mathbf{u}}^n_h\|^2_\text{dG} + \tau^2 \mathcal{A}_{\text{sip}}(\xi^m_h, \xi^m_h) + \frac{1}{\delta \mu} \tau \|S^m_h\|^2 +2 \tau \sum_{n=1}^m \mathcal{A}_d(q_r^n(\widetilde{\mathbf{u}}_h),\widetilde{\mathbf{u}}_h^n)\\
&\leq \|\mathbf{u}^0_h\|^2 + \frac{2C^2_2}{\omega \mu} \tau \sum_{n=1}^m \|\mathbf{f}^n\|^2 + \frac{1}{\delta \mu}\tau^2 \sum_{n=1}^m \norm{S_h^{n-1}}^2 + \tau^2 \mathcal{A}_{\text{sip}}(\xi^0_h, \xi^0_h) + \frac{1}{\delta \mu} \tau \|S^0_h\|^2 \\
&\hspace{1em} +\frac{\omega \mu}{2} \tau^2 
\sum_{n=1}^m \sum_{i=1}^n \|\widetilde{\mathbf{u}}^i_h\|^2_\text{dG}. \label{eq6.86}
\end{aligned}
\end{equation}
Using the positivity property given in  \eqref{5.38} and applying the discrete Gronwall's inequality, \eqref{eq6.86} yields
\begin{equation}
\begin{aligned}
&\|\mathbf{u}^m_h\|^2 + \frac{\omega \mu}{2} \tau \sum_{n=1}^m \|\widetilde{\mathbf{u}}^n_h\|^2_\text{dG} + \tau^2 \mathcal{A}_{\text{sip}}(\xi^m_h, \xi^m_h) + \frac{1}{\delta \mu} \tau \|S^m_h\|^2 \\
&\leq \tilde{C}_T \left(\|\mathbf{u}^0_h\|^2 + \frac{2C^2_2}{\omega \mu} \tau \sum_{n=1}^m \|\mathbf{f}^n\|^2 + \tau^2 \mathcal{A}_{\text{sip}}(\xi^0_h, \xi^0_h) + \frac{1}{\delta \mu} \tau \|S^0_h\|^2\right). \label{94}
\end{aligned}
\end{equation}
For the non-symmetric form of $\mathcal{A}_\epsilon$, another application of the CS inequality, Poincaré's inequality \eqref{eq6.54}, and Young's inequality leads to
\begin{equation}
\left|(\mathbf{f}^n, \widetilde{\mathbf{u}}^n_h)\right| \leq \frac{2 C^2_2}{\omega \mu} \|\mathbf{f}^n\|^2 + \frac{\omega \mu}{8} \|\widetilde{\mathbf{u}}^n_h\|^2_\text{dG}. \label{95}
\end{equation}
The last term in the left hand side (LHS) of \eqref{eq6.84} can be bounded using \eqref{eq4.23} and Young's inequality as
\begin{equation}
\begin{aligned}
|\mathcal{A}_\epsilon(q_r^n(\widetilde{\mathbf{u}}_h),\widetilde{\mathbf{u}}_h^n)| &\leq C \tau \sum_{i=1}^n \beta(t_n-t_i) \|\widetilde{\mathbf{u}}_h^i\|_\text{dG} \|\widetilde{\mathbf{u}}_h^n\|_\text{dG} \\
&\leq \frac{\omega \mu}{8}\|\widetilde{\mathbf{u}}_h^n\|^2_\text{dG}+\tilde{C}\left(\tau \sum_{i=1}^n \beta(t_n-t_i) \|\widetilde{\mathbf{u}}_h^i\|_\text{dG} \right)^2 .
    \label{6.87}
    \end{aligned}
\end{equation}
We utilize \eqref{95}-\eqref{6.87} in \eqref{eq6.84}. Then, multiplying the resulting inequality by 2, and taking the sum over $n=1$ to $m$, we arrive at
\begin{equation}
    \begin{aligned}
       &\|\mathbf{u}^m_h\|^2 + \frac{\omega \mu}{2} \tau \sum_{n=1}^m \|\widetilde{\mathbf{u}}^n_h\|^2_\text{dG} + \tau^2 \mathcal{A}_{\text{sip}}(\xi^m_h, \xi^m_h) + \frac{1}{\delta \mu} \tau \|S^m_h\|^2
\leq \|\mathbf{u}^0_h\|^2 + \frac{4 C^2_2}{\omega \mu} \tau \sum_{n=1}^m \|\mathbf{f}^n\|^2 \\&+ \frac{1}{\delta \mu}\tau^2 \sum_{n=1}^m \norm{S_h^{n-1}}^2 + \tau^2 \mathcal{A}_{\text{sip}}(\xi^0_h, \xi^0_h) + \frac{1}{\delta \mu} \tau \|S^0_h\|^2 +\frac{\omega \mu}{2} \tau^2 
\sum_{n=1}^m \sum_{i=1}^n \|\widetilde{\mathbf{u}}^i_h\|^2_\text{dG}\\
&+\tilde{C}\tau \sum_{n=1}^m \left( \tau \sum_{i=1}^n \beta(t_n-t_i) \|\widetilde{\mathbf{u}}_h^i\|_\text{dG} \right)^2.
\label{6.88}
    \end{aligned}
\end{equation}
Using Holder's inequality the last term of the above inequality is bounded as
\begin{equation}
    \begin{aligned}
        \tau \sum_{n=1}^m \left( \tau \sum_{i=1}^n \beta(t_n-t_i) \|\widetilde{\mathbf{u}}_h^i\|_\text{dG} \right)^2 &\leq \gamma^2 \tau \sum_{n=1}^m \left(\tau \sum_{i=1}^n e^{-2\eta(t_n-t_i)}\right) \left(\tau \sum_{i=1}^n \|\widetilde{\mathbf{u}}_h^i\|^2_\text{dG}\right)\\
        & \leq \frac{\gamma^2 e^{2\eta \tau}}{2\eta} \tau^2 \sum_{n=1}^m \sum_{i=1}^n \|\widetilde{\mathbf{u}}_h^i\|^2_\text{dG} . \label{eq6.89}
    \end{aligned}
\end{equation}
Using \eqref{eq6.89} and the discrete Gronwall's inequality, the inequality \eqref{6.88} becomes
\begin{equation}
\begin{aligned}
&\|\mathbf{u}^m_h\|^2 + \frac{\omega \mu}{2} \tau \sum_{n=1}^m \|\widetilde{\mathbf{u}}^n_h\|^2_\text{dG} + \tau^2 \mathcal{A}_{\text{sip}}(\xi^m_h, \xi^m_h) + \frac{1}{\delta \mu} \tau \|S^m_h\|^2 \\
&\leq \tilde{C}_T \left(\|\mathbf{u}^0_h\|^2 + \frac{4 C^2_2}{\omega \mu} \tau \sum_{n=1}^m \|\mathbf{f}^n\|^2 + \tau^2 \mathcal{A}_{\text{sip}}(\xi^0_h, \xi^0_h) + \frac{1}{\delta \mu} \tau \|S^0_h\|^2\right). \label{99}
\end{aligned}
\end{equation}
Recall that $\xi^0_h = S^0_h = 0$. With the coercivity property of $\mathcal{A}_{\text{sip}}$ \eqref{eq3.21} and the stability of the $L^2$ projection, \eqref{94} and \eqref{99} conclude the result for symmetric and non-symmetric cases, respectively.
\end{proof}

\section{\textit{A priori} bounds for the velocity}
From this point, we consider the case where $r_1=r$ and $r_2 = r - 1$. Then we introduce a new norm, for given $\bm\varphi \in (W^{1,3}(\mathcal{O}) \cap L^{\infty}(\mathcal{O}))^d$, as
\begin{equation*}
    ||| \bm\varphi ||| = \| \bm\varphi \|_{L^{\infty}(\mathcal{O})} + | \bm\varphi |_{W^{1,3}(\mathcal{O})}.
\end{equation*}
We assume that $\mathbf{u}^0 \in (H_0^1(\mathcal{O}))^d$ and that $\nabla \cdot \mathbf{u}^0 = 0$.
We will utilize the following inverse inequalities \cite{brenner2007mathematical,riviere2008discontinuous}:
\begin{align}
    \|\mathbf{w}_h\|_{L^3(\mathcal{O})} &\leq C h^{-d/6} \|\mathbf{w}_h\|, \quad \forall \mathbf{w}_h \in \mathbf{M}_h,
    \label{8.157}
\\
    \|\mathbf{w}_h\|_\text{dG} &\leq C h^{-1} \|\mathbf{w}_h\|, \quad \forall \mathbf{w}_h \in \mathbf{M}_h, \label{8.159}
\\
    \|g_h\|_\text{dG} &\leq C h^{-1} \|g_h\|, \quad \forall g_h \in P_h, \label{8.160}
\\
    \|\mathbf{w}_h\|_{L^\infty(\mathcal{O})} &\leq C h^{-d/2} \|\mathbf{w}_h\|. \quad \forall \mathbf{w}_h \in \mathbf{M}_h.
    \label{8.158}
\end{align}
Additionally, we will employ the following trace inequalities \cite{di2011mathematical,riviere2008discontinuous}. For all $K \in \mathcal{K}_h$ and $F \in \partial K$,
\begin{align}
    \|\bm\varphi\|_{L^2(F)} &\leq C h^{-1/2} \left( \|\bm\varphi\|_{L^2(K)} + h \|\nabla \bm\varphi\|_{L^2(K)} \right), \quad \forall \bm\varphi \in H^1(K), \label{8.161}
\\
    \|\mathbf{w}_h\|_{L^k(F)} &\leq C h^{-1/k} \|\mathbf{w}_h\|_{L^k(K)}, \quad \forall \mathbf{w}_h \in \mathbf{M}_h, \, k \geq 1.
\end{align}
\subsection{Approximation}
\textbullet \quad For $0 \leq t \leq T$, the local $L^2$ projection $\pi_h : L^2(\mathcal{O}) \to P_h$, defined on each $K \in \mathcal{K}_h$ by
\begin{equation*}
    \int_K (\pi_h q(t) - q(t)) g_h = 0, \quad \forall g_h \in \mathbb{P}_{r-1}(K), \quad \forall q(t) \in L^2(\mathcal{O}).
\end{equation*}
$\pi_h$ satisfies for all $0 \leq t \leq T$, and for all $K \in \mathcal{K}_h$,
\begin{equation}
    \|\pi_h q(t) - q(t)\|_{L^2(K)} + h_K \|\nabla_h (\pi_h q(t) - q(t))\|_{L^2(K)} \leq C h_K^{\min(r,s)} |q(t)|_{H^s(K)}, \\ \forall q(t) \in H^s(\mathcal{O}). \label{eq6.101}
\end{equation}
We also recall the stability of this projection in the dG norm \cite{girault2016strong}.
\begin{lemma}
Assume $q(t) \in H^1(\mathcal{O})$. Then, for some $t \in [0,T]$, we have 
\begin{equation}
    |\pi_h q(t)|_{\text{dG}} \leq C |q(t)|_{H^1(\mathcal{O})}. \label{eq6.102}
\end{equation}
\end{lemma}
\textbullet \quad For $0 \leq t \leq T$, we employ the operator $\Pi_h : (H_0^1(\mathcal{O}))^d \to \mathbf{M}_h$, defined on each $K \in \mathcal{K}_h$, which ensures the discrete divergence preservation \cite{chaabane2017convergence, riviere2008discontinuous}, as follows
\begin{equation}
    \mathcal{b}(\Pi_h \bm\varphi(t), g_h) = \mathcal{b}(\bm\varphi(t), g_h), \quad \forall g_h \in P_h. \label{6.89}
\end{equation}
If $\bm\varphi(t) \in \mathbf{M} \cap (W^{s,k}(\mathcal{O}) \cap H_0^1(\mathcal{O}))^d$, then for $0 \leq t \leq T$, we have the following approximation properties \cite{chaabane2017convergence}:
\begin{align}
    \|\Pi_h \bm\varphi(t) - \bm\varphi(t)\|_{L^k(K)} &\leq C h_K^s |\bm\varphi(t)|_{W^{s,k}(\Delta_K)}, \label{6.90}
\\
    \|\nabla(\Pi_h \bm\varphi(t) - \bm\varphi(t))\|_{L^k(K)} &\leq C h_K^{s-1} |\bm\varphi(t)|_{W^{s,k}(\Delta_K)},
    \label{6.91}
\end{align}
where $1 \leq s \leq r + 1$, $1 \leq k \leq \infty$, and $\Delta_K$ represents a macro element containing $K$.
If $\bm\varphi(t) \in (W^{s,k}(\mathcal{O}) \cap H_0^1(\mathcal{O}))^d$, then for $0 \leq t \leq T$, bounds \eqref{6.90} and \eqref{6.91} yield the global estimates:
\begin{align}
    \|\Pi_h \bm\varphi(t) - \bm\varphi(t)\|_{L^k(\mathcal{O})} &\leq C h^s |\bm\varphi(t)|_{W^{s,k}(\mathcal{O})}, \label{eq:6.96}
\\
    \|\Pi_h \bm\varphi(t) - \bm\varphi(t)\|_{\text{dG}} &\leq C h^{s-1} |\bm\varphi(t)|_{H^s(\mathcal{O})}. \label{eq:6.97}
\end{align}
\begin{lemma} [\cite{masri2022discontinuous}]
Assume that $\bm\varphi(t) \in (L^\infty(\mathcal{O}) \cap W^{1,3}(\mathcal{O}) \cap H_0^1(\mathcal{O}))^d$. Then, for some $t \in [0,T]$, we have
\begin{equation}
    \|\Pi_h \bm\varphi(t)\|_{L^\infty(\mathcal{O})} \leq C ||| \bm\varphi(t) |||.
\end{equation}
\end{lemma}
As a corollary, we have
\begin{equation}
    |||\bm\varphi-\Pi_h\bm\varphi||| \leq C |||\bm\varphi|||. \label{eq7.99}
\end{equation}
\textbullet \quad For $0 \leq t \leq T$, we define the Ritz projection $Q_h :(H_0^1(\mathcal{O}))^d \to \mathbf{M}_h$ by
\begin{equation}
\mathcal{A}_d\left(Q_h \bm{\varphi}(t), \mathbf{w}_h\right) = \mathcal{A}_d\left( \bm{\varphi}(t), \mathbf{w}_h\right), \quad \forall \mathbf{w}_h \in \mathbf{M}_h.  
\end{equation}
If the domain is convex, then the projection satisfies the following approximation properties \cite{riviere2008discontinuous},
\begin{equation}
\|\bm{\varphi}(t) - Q_h \bm{\varphi}(t)\| \leq C h^{r+1} |\bm{\varphi}(t)|_{H^{r+1}(\mathcal{O})}. \label{8.185}
\end{equation}
\subsection{Error equations}
In view of Lemma 3, since $\mathbf{M}_h \subset \mathbf{M}$ the true solution $(\mathbf{u},p)$ of \eqref{eq3.1}--\eqref{eq3.4} satisfies $\forall n \geq 1$,
\begin{equation}
\begin{aligned}
((\partial_t \mathbf{u})^n, \mathbf{w}_h) + \mathcal{A}_c(\mathbf{u}^n; \mathbf{u}^n, \mathbf{u}^n, \mathbf{w}_h) + \mu \mathcal{A}_\epsilon(\mathbf{u}^n, \mathbf{w}_h) + \int_0^{t_n} \beta(t_n-s) \mathcal{A}_\epsilon(\mathbf{u}(s), \mathbf{w}_h) ds \\= \mathcal{b}(\mathbf{w}_h, p^n) + (\mathbf{f}^n, \mathbf{w}_h), \quad \forall \mathbf{w}_h \in \mathbf{M}_h. \label{6.98}
\end{aligned}
\end{equation}
For simplification of notation, we introduce the following discretization errors, $\forall n \geq 0$,
\begin{align*}
\tilde{\bm{\mathcal{E}}}_h^n &= \widetilde{\mathbf{u}}_h^n - \Pi_h \mathbf{u}^n, \\
\bm{\mathcal{E}}_h^n &= \mathbf{u}_h^n - \Pi_h \mathbf{u}^n.
\end{align*}
We set $\widetilde{\mathbf{u}}_h^0 = \mathbf{u}_h^0 = \Pi_h\mathbf{u}^0$, therefore $\tilde{\bm{\mathcal{E}}}_h^0 = 0$. \\
Multiply \eqref{6.98} by $\tau$ and then subtract it from \eqref{eq5.36} to obtain  $\forall \mathbf{w}_h \in \mathbf{M}_h$ and $\forall n \geq 1$,
\begin{equation}
\begin{aligned}
&(\tilde{\bm{\mathcal{E}}}_h^n, \mathbf{w}_h) + \tau \mathcal{A}_c(\mathbf{u}_h^{n-1}; \mathbf{u}_h^{n-1}, \tilde{\bm{\mathcal{E}}}_h^n, \mathbf{w}_h) + \tau R_C(\mathbf{w}_h) + \tau \mu \mathcal{A}_\epsilon(\tilde{\bm{\mathcal{E}}}_h^n, \mathbf{w}_h)+ \tau \mathcal{A}_\epsilon(q_r^n(\tilde{\bm{\mathcal{E}}}_h),\mathbf{w}_h) \\
&+ \tau R_S(\mathbf{w}_h) = (\bm{\mathcal{E}}_h^{n-1}, \mathbf{w}_h) - \tau \mu \mathcal{A}_\epsilon(\Pi_h \mathbf{u}^n - \mathbf{u}^n, \mathbf{w}_h) -\tau \mathcal{A}_\epsilon(q_r^n(\Pi_h \mathbf{u}-\mathbf{u}),\mathbf{w}_h) \\
&+ \tau \mathcal{b}(\mathbf{w}_h, p_h^{n-1} - p^n) + R_t(\mathbf{w}_h),
\end{aligned}
\label{7.105}
\end{equation}
where
\begin{align*}
R_C(\mathbf{w}_h) &= \mathcal{A}_c(\mathbf{u}_h^{n-1}; \mathbf{u}_h^{n-1}, \Pi_h \mathbf{u}^n, \mathbf{w}_h) - \mathcal{A}_c(\mathbf{u}^n; \mathbf{u}^n, \mathbf{u}^n, \mathbf{w}_h), \\
R_S(\mathbf{w}_h) &=\mathcal{A}_\epsilon(q_r^n(\mathbf{u}),\mathbf{w}_h) - \int_0^{t_n} \beta(t_n-s) \mathcal{A}_\epsilon(\mathbf{u}(s), \mathbf{w}_h) ds, \\
R_t(\mathbf{w}_h) &= \tau ((\partial_t \mathbf{u})^n, \mathbf{w}_h) + (\Pi_h \mathbf{u}^{n-1} - \Pi_h \mathbf{u}^n, \mathbf{w}_h).
\end{align*}
In view of $\mathcal{b}(\mathbf{u}^n,g_h)=0, \forall g_h \in P_h$, using \eqref{6.89} and \eqref{eq4.31}, we obtain $\forall n \geq 1$,
\begin{equation}
\mathcal{b}(\Pi_h \mathbf{u}^n, g_h) = (\nabla_h \cdot \Pi_h \mathbf{u}^n - R_h([\Pi_h \mathbf{u}^n]), g_h) = 0, \quad \forall g_h \in P_h. \label{6.102}
\end{equation}
Using \eqref{6.102} and \eqref{eq:5.41}, we have $\forall n \geq 1$,
\begin{equation}
\mathcal{A}_\text{sip}(v_h^n, g_h) = -\frac{1}{\tau} \mathcal{b}(\tilde{\bm{\mathcal{E}}}_h^n, g_h), \quad \forall g_h \in P_h^0. \label{eq6.108}
\end{equation}
Adding and subtracting $\Pi_h \mathbf{u}^n$ in \eqref{eq5.41}, we obtain $\forall n \geq 1$,
\begin{equation}
(\bm{\mathcal{E}}_h^n, \mathbf{w}_h) = (\tilde{\bm{\mathcal{E}}}_h^n, \mathbf{w}_h) - \tau (\nabla_h v_h^n - \mathbf{G}_h([v_h^n]), \mathbf{w}_h), \quad \forall \mathbf{w}_h \in \mathbf{M}_h. \label{6.104}
\end{equation}
From \eqref{eq4.32} and \eqref{6.104}, we obtain $\forall n \geq 1$,
\begin{equation}
(\bm{\mathcal{E}}_h^n, \mathbf{w}_h) = (\tilde{\bm{\mathcal{E}}}_h^n, \mathbf{w}_h) + \tau \mathcal{b}(\mathbf{w}_h, v_h^n), \quad \forall \mathbf{w}_h \in \mathbf{M}_h.  \label{eq:6.110}  
\end{equation}
Adding and subtracting $\Pi_h \mathbf{u}^n$ in \eqref{eq5.40} we have $\forall n \geq 1$,
\begin{equation}
(p_h^n, g_h) = (p_h^{n-1}, g_h) + (v_h^n, g_h) - \delta \mu (\nabla_h \cdot \tilde{\bm{\mathcal{E}}}_h^n - R_h([\tilde{\bm{\mathcal{E}}}_h^n]), g_h) - \delta (\nabla_h \cdot q_r^n(\tilde{\bm{\mathcal{E}}}_h) - R_h([q_r^n(\tilde{\bm{\mathcal{E}}}_h)]), g_h).
\label{eq:6.111}
\end{equation}

\subsection{Auxiliary estimates}
\begin{lemma} [\cite{masri2022discontinuous}]
Fix $g_h \in P_h$. Then we have $\forall n \geq 1$,
\begin{equation}
\mathcal{b}(\bm{\mathcal{E}}_h^n, g_h) = \mathcal{b}(\tilde{\bm{\mathcal{E}}}_h^n, g_h) + \tau \mathcal{A}_{\text{sip}}(v_h^n, g_h) - \tau \sum_{F \in \mathcal{F}_h^i}  \frac{\tilde{\sigma}}{h_F} \int_F [v_h^n][g_h] + \tau (\mathbf{G}_h([v_h^n]), \mathbf{G}_h([g_h])), \label{6.112}
\end{equation}
\begin{equation}
\mathcal{b}(\bm{\mathcal{E}}_h^n, g_h) = -\tau \sum_{F \in \mathcal{F}_h^i}  \frac{\tilde{\sigma}}{h_F} \int_F [v_h^n][g_h] + \tau (\mathbf{G}_h([v_h^n]), \mathbf{G}_h([g_h])). \label{6.113}
\end{equation}
In addition, $\forall n \geq 1$,
\begin{equation}
\begin{aligned}
\| \tilde{\bm{\mathcal{E}}}_h^n - \bm{\mathcal{E}}_h^{n-1} \|^2 = \| \bm{\mathcal{E}}_h^n - \bm{\mathcal{E}}_h^{n-1} \|^2 + \tau^2 (\|\nabla_h v_h^n \|^2 + \|\mathbf{G}_h([v_h^n])\|^2) + \tau^2 (H_1^n - H_2^n) \\
 + \tau^2 \bigg(\sum_{F \in \mathcal{F}_h^i} \frac{\tilde{\sigma}}{h_F} \| [v_h^n - v_h^{n-1}] \|^2 - \|\mathbf{G}_h([v_h^n - v_h^{n-1}])\|^2\bigg) \\
 - 2\tau^2 (\nabla_h v_h^n, \mathbf{G}_h([v_h^n])) + 2\delta_{n,1} \tau \mathcal{b}(\bm{\mathcal{E}}_h^0, v_h^1), \label{6.114}
\end{aligned}
\end{equation}
\textit{where,} $\forall n \geq 1$,
\begin{align}
H_1^n &= \sum_{F \in \mathcal{F}_h^i}  \frac{\tilde{\sigma}}{h_F} \|[v_h^n]\|_{L^2(F)}^2 
- \sum_{F \in \mathcal{F}_h^i}  \frac{\tilde{\sigma}}{h_F} \|[v_h^{n-1}]\|_{L^2(F)}^2,
\label{eq:6.115}
\\
H_2^n &= \|\mathbf{G}_h([v_h^n])\|^2 - \|\mathbf{G}_h([v_h^{n-1}])\|^2.
\label{eq:6.116}
\end{align}
\end{lemma}
\begin{lemma} [\cite{masri2022discontinuous}]
We have the following bounds on the discretization forms $ \mathcal{A}_c, \mathcal{b}, \mathscr{U}, $ and $ \mathscr{C} $.
\begin{itemize}
\item[(i)] Assuming $ \bm\varphi \in (W^{1,3}(\mathcal{O}) \cap H_0^1(\mathcal{O}) \cap L^{\infty}(\mathcal{O}))^d $, we have for all $ \bm{\theta} \in \mathbf{M} $ and $ \mathbf{w}_h, \mathbf{z}_h \in \mathbf{M}_h $,
\begin{equation}
|\mathcal{A}_c(\bm{\theta}; \mathbf{z}_h, \bm\varphi, \mathbf{w}_h)| + |\mathcal{b}(\mathbf{z}_h, \bm\varphi \cdot \mathbf{w}_h)| \leq C \|\mathbf{z}_h\| \|\bm\varphi\| \|\mathbf{w}_h\|_\text{dG}. \label{4.117}
\end{equation}
\item[(ii)] Assuming $ \bm\varphi \in (H^{r+1}(\mathcal{O}) \cap H_0^1(\mathcal{O}))^d $, we have for all $ \bm{\theta} \in \mathbf{M} $ and $ \mathbf{w}_h, \mathbf{z}_h \in \mathbf{M}_h $,
\begin{equation}
|\mathscr{C}(\mathbf{z}_h, \bm\varphi - \Pi_h \bm\varphi, \mathbf{w}_h)| + |\mathcal{b}(\mathbf{z}_h, \bm\varphi - \Pi_h \bm\varphi \cdot \mathbf{w}_h)| \leq C \|\mathbf{z}_h\| \|\bm\varphi\|_{H^{r+1}(\mathcal{O})} \|\mathbf{w}_h\|_\text{dG}. \label{4.118}
\end{equation}
\item[(iii)] Assume $ \mathbf{z} \in (H^{r+1}(\mathcal{O}) \cap H_0^1(\mathcal{O}))^d $ satisfies $ \mathcal{b}(\mathbf{z}, g_h) = 0, \forall g_h \in P_h $ and $ \bm\varphi \in (W^{1,3}(\mathcal{O}) \cap H_0^1(\mathcal{O}) \cap L^{\infty}(\mathcal{O}))^d $. Then for all $ \bm{\theta} \in \mathbf{M} $ and $ \mathbf{w}_h \in \mathbf{M}_h $,
\begin{equation}
|\mathcal{A}_c(\bm{\theta}; \Pi_h \mathbf{z} - \mathbf{z}, \bm\varphi, \mathbf{w}_h)| + |\mathcal{b}(\Pi_h \mathbf{z} - \mathbf{z}, \bm\varphi \cdot \mathbf{w}_h)| \leq C h^{r+1} \|\mathbf{z}\|_{H^{r+1}(\mathcal{O})} \|\bm\varphi\| \|\mathbf{w}_h\|_\text{dG}. \label{4.119}
\end{equation}
\item[(iv)] Assume $ \mathbf{z} \in (H^{r+1}(\mathcal{O}) \cap H_0^1(\mathcal{O}))^d $ and $ \bm\varphi \in (W^{1,3}(\mathcal{O}) \cap H_0^1(\mathcal{O}) \cap L^{\infty}(\mathcal{O}))^d $. Then for all $ \bm{\theta} \in \mathbf{M} $ and $ \mathbf{w}_h \in \mathbf{M}_h $,
\begin{equation}
|\mathscr{C}(\Pi_h \bm\varphi, \Pi_h \mathbf{z} - \mathbf{z}, \mathbf{w}_h)| + |\mathscr{U}(\bm{\theta}; \Pi_h \bm\varphi, \Pi_h \mathbf{z} - \mathbf{z}, \mathbf{w}_h)| \leq C h^r \|\mathbf{z}\|_{H^{r+1}(\mathcal{O})} \|\bm\varphi\| \|\mathbf{w}_h\|_\text{dG}. \label{4.120}
\end{equation}
\item[(v)] Finally, if $ \bm\varphi \in (H^{r+1}(\mathcal{O}) \cap H_0^1(\mathcal{O}))^d $, then for all $ \bm{\theta} \in \mathbf{M} $ and $ \mathbf{w}_h \in \mathbf{M}_h $,
\begin{equation}
|\mathscr{U}(\bm{\theta}; \mathbf{z}_h, \bm\varphi - \Pi_h \bm\varphi, \mathbf{w}_h)| \leq C h^{r - d/4} \|\mathbf{z}_h\| \|\bm\varphi\|_{H^{r+1}(\mathcal{O})} \|\mathbf{w}_h\|_\text{dG}.
\label{4.121}
\end{equation}
\end{itemize}
Here, $ C $ is a constant that is independent of $ h, \tau, \mathbf{z}_h, \bm{\theta}, \bm\varphi $ and $ \mathbf{w}_h $.
\end{lemma}
\begin{lemma} [\cite{masri2022discontinuous}] Assuming $\mathbf{u} \in L^{\infty}(0, T; (H^{r+1}(\mathcal{O}))^d)$ and $\partial_t \mathbf{u} \in L^2(0, T; (L^2(\mathcal{O}))^d)$, we have
\begin{equation}
|R_C(\tilde{\bm{\mathcal{E}}}_h^n)| \leq \frac{C}{\omega \mu} \tau \int_{t_{n-1}}^{t_n} \|\partial_t \mathbf{u}\|^2 
+ \frac{C}{\omega \mu} h^{2r} 
+ \frac{C}{\omega \mu} \|\bm{\mathcal{E}}_h^{n-1}\|^2 
+ \frac{\omega \mu}{16} \|\tilde{\bm{\mathcal{E}}}_h^n\|_{\text{dG}}^2, \quad \forall n \geq 1. \label{eq:6.117}
\end{equation}
\end{lemma}
We are now in a state to present and prove the key convergence estimates.
\begin{theorem}
Let the assumptions (\textbf{A1}) and (\textbf{A2}) be satisfied. Then for $\tau$ small enough, we have for $0 \leq m \leq N$, 
\begin{equation}
\|\mathbf{u}_h^m - \mathbf{u}^m\|^2 + \frac{\omega \mu}{4} \tau \sum_{n=1}^m \|\widetilde{\mathbf{u}}_h^n - \mathbf{u}^n\|^2_\text{dG} 
\leq \tilde{C}_T \left(\tau + h^{2r}\right). \label{eq:6.118}
\end{equation}
\end{theorem}
\begin{proof}
Set $\mathbf{w}_h = \tilde{\bm{\mathcal{E}}}_h^n$ in \eqref{7.105}. Using the positivity property of $\mathcal{A}_c$ \eqref{eq4.12} and the coercivity property of $\mathcal{A}_\epsilon$ \eqref{eq3.20}, we have
\begin{align*}
\frac{1}{2} (\|\tilde{\bm{\mathcal{E}}}_h^n\|^2 - \|\bm{\mathcal{E}}_h^{n-1}\|^2 + \|\tilde{\bm{\mathcal{E}}}_h^n - \bm{\mathcal{E}}_h^{n-1}\|^2) + \omega \mu \tau \|\tilde{\bm{\mathcal{E}}}_h^n\|^2_\text{dG} 
+ \tau R_C(\tilde{\bm{\mathcal{E}}}_h^n) + \tau \mathcal{A}_\epsilon(q_r^n(\tilde{\bm{\mathcal{E}}}_h),\tilde{\bm{\mathcal{E}}}_h^n) + \tau R_S(\tilde{\bm{\mathcal{E}}}_h^n)\\
\leq \tau \mathcal{b}(\tilde{\bm{\mathcal{E}}}_h^n, p_h^{n-1} - p^n) - \tau \mu \mathcal{A}_\epsilon(\Pi_h \mathbf{u}^n - \mathbf{u}^n, \tilde{\bm{\mathcal{E}}}_h^n)-\tau \mathcal{A}_\epsilon(q_r^n(\Pi_h \mathbf{u}-\mathbf{u}),\tilde{\bm{\mathcal{E}}}_h^n) + R_t(\tilde{\bm{\mathcal{E}}}_h^n).
\end{align*}
Inserting \eqref{6.114} into the inequality above results in
\begin{equation}
\begin{aligned}
&\frac{1}{2} \|\tilde{\bm{\mathcal{E}}}_h^n\|^2 - \|\bm{\mathcal{E}}_h^{n-1}\|^2 + \|\bm{\mathcal{E}}_h^n - \bm{\mathcal{E}}_h^{n-1}\|^2 
+ \frac{\tau^2}{2} \left( \|\nabla_h v_h^n\|^2 + \|\mathbf{G}_h([v_h^n])\|^2 \right) + \omega \mu \tau \|\tilde{\bm{\mathcal{E}}}_h^n\|_\text{dG}^2\\ 
&+ \frac{\tau^2}{2} \sum_{F \in \mathcal{F}_h^i}\frac{\tilde{\sigma}}{h_F} \|[v_h^n - v_h^{n-1}]\|_{L^2(F)}^2 + \frac{\tau^2}{2} (H_1^n - H_2^n)+ \tau \mathcal{A}_\epsilon(q_r^n(\tilde{\bm{\mathcal{E}}}_h),\tilde{\bm{\mathcal{E}}}_h^n)\\
&\leq -\tau R_C(\tilde{\bm{\mathcal{E}}}_h^n) - \tau R_S(\tilde{\bm{\mathcal{E}}}_h^n) + \tau \mathcal{b}(\tilde{\bm{\mathcal{E}}}_h^n, p_h^{n-1} - p^n) 
- \mu \tau \mathcal{A}_\epsilon (\Pi_h \mathbf{u}^n - \mathbf{u}^n, \tilde{\bm{\mathcal{E}}}_h^n) + R_t(\tilde{\bm{\mathcal{E}}}_h^n) \\
&\hspace{1em} -\tau \mathcal{A}_\epsilon(q_r^n(\Pi_h \mathbf{u}-\mathbf{u}),\tilde{\bm{\mathcal{E}}}_h^n)+ \frac{\tau^2}{2} \|\mathbf{G}_h([v_h^n - v_h^{n-1}])\|^2 
+ \tau^2 (\nabla_h v_h^n, \mathbf{G}_h([v_h^n])) - \delta_{n,1} \tau \mathcal{b}(\bm{\mathcal{E}}_h^0, v_h^1).
\label{6.117}
\end{aligned}
\end{equation}
We start with handling the sixth and seventh terms in the RHS of \eqref{6.117}. Using \eqref{eq4.30} and the assumption on the penalty parameter $\tilde{\sigma} \geq 2 \tilde{B}_r^2$, we have
\begin{equation}
\|\mathbf{G}_h([v_h^n - v_h^{n-1}])\|^2 
\leq \sum_{F \in \mathcal{F}_h^i}  \frac{\tilde{B}_r^2}{h_F} \|[v_h^n - v_h^{n-1}]\|_{L^2(F)}^2 
\leq \frac{1}{2} \sum_{F \in \mathcal{F}_h^i}  \frac{\tilde{\sigma}}{h_F} \|[v_h^n - v_h^{n-1}]\|_{L^2(F)}^2.
\label{6.118}
\end{equation}
With the assumption $\tilde{\sigma} \geq 4 \tilde{B}_r^2$, we use the CS inequality, Young's inequality and \eqref{eq4.30}, to obtain
\begin{equation}
\begin{aligned}
| (\nabla_h v_h^n, \mathbf{G}_h([v_h^n])) | 
&\leq \frac{1}{4} \|\nabla_h v_h^n\|^2 + \|\mathbf{G}_h([v_h^n])\|^2 \\
&\leq \frac{1}{4} \|\nabla_h v_h^n\|^2 
+ \sum_{F \in \mathcal{F}_h^i}  \frac{\tilde{B}_r^2}{h_F} \|[v_h^n]\|_{L^2(F)}^2 
\leq \frac{1}{4} \|\nabla_h v_h^n\|^2 
+ \frac{1}{4} \sum_{F \in \mathcal{F}_h^i}  \frac{\tilde{\sigma}}{h_F} \|[v_h^n]\|_{L^2(F)}^2.
\label{6.119}
\end{aligned}
\end{equation}
Next, setting $\mathbf{w}_h = \bm{\mathcal{E}}_h^n$ in \eqref{6.104} and use \eqref{eq4.32} to have 
\begin{equation*}
\frac{1}{2} (\|\bm{\mathcal{E}}_h^n\|^2 - \|\tilde{\bm{\mathcal{E}}}_h^n\|^2) + \frac{1}{2} \|\bm{\mathcal{E}}_h^n - \tilde{\bm{\mathcal{E}}}_h^n\|^2 
= \tau \mathcal{b}(\bm{\mathcal{E}}_h^n, v_h^n).
\end{equation*}
Using \eqref{6.113}, we find that
\begin{equation}
\frac{1}{2} (\|\bm{\mathcal{E}}_h^n\|^2 - \|\tilde{\bm{\mathcal{E}}}_h^n\|^2) + \frac{1}{2} \|\bm{\mathcal{E}}_h^n - \tilde{\bm{\mathcal{E}}}_h^n\|^2 
+ \tau^2 \sum_{F \in \mathcal{F}_h^i}  \frac{\tilde{\sigma}}{h_F}\|v_h^n\|_{L^2(F)}^2 = \tau^2 \|\mathbf{G}_h([v_h^n])\|^2. \label{6.121}
\end{equation}
Now, we set $\mathbf{w}_h = \bm{\mathcal{E}}_h^n - \tilde{\bm{\mathcal{E}}}_h^n$ in \eqref{6.104} and then apply \eqref{6.112} to obtain
\begin{equation*}
\|\bm{\mathcal{E}}_h^n - \tilde{\bm{\mathcal{E}}}_h^n\|^2 = \tau^2 \mathcal{A}_{\text{sip}}(v_h^n, v_h^n) - \tau^2 \sum_{F \in \mathcal{F}_h^i}  \frac{\tilde{\sigma}}{h_F}\|v_h^n\|_{L^2(F)}^2 + \tau^2 \|\mathbf{G}_h([v_h^n])\|^2.
\end{equation*}
Thus, \eqref{6.121} now reads
\begin{equation}
\begin{aligned}
\frac{1}{2} (\|\bm{\mathcal{E}}_h^n\|^2 - \|\bm{\mathcal{E}}_h^{n-1}\|^2) + \frac{\tau^2}{2} \mathcal{A}_{\text{sip}}(v_h^n, v_h^n) 
+ \frac{\tau^2}{2} \sum_{F \in \mathcal{F}_h^i}  \frac{\tilde{\sigma}}{h_F} \|[v_h^n]\|_{L^2(F)}^2 = \frac{\tau^2}{2} \|\mathbf{G}_h([v_h^n])\|^2. \label{6.123}
\end{aligned}
\end{equation}
We now add \eqref{6.123} to \eqref{6.117}. With bounds \eqref{6.118} and \eqref{6.119}, we obtain
\begin{equation}
\begin{aligned}
&\frac{1}{2} \left( \| \bm{\mathcal{E}}_h^n \|^2 - \| \bm{\mathcal{E}}_h^{n-1} \|^2 + \| \bm{\mathcal{E}}_h^n - \bm{\mathcal{E}}_h^{n-1} \|^2 \right) + \frac{\tau^2}{4} | v_h^n |_{\text{dG}}^2 + \frac{\tau^2}{2} \mathcal{A}_{\text{sip}} (v_h^n, v_h^n) + \omega \mu \tau \| \tilde{\bm{\mathcal{E}}}_h^n \|_{\text{dG}}^2 \\
 &+ \frac{\tau^2}{4} \sum_{F \in \mathcal{F}_h^i}  \frac{\tilde{\sigma}}{h_F} \| [v_h^n - v_h^{n-1}] \|_{L^2(F)}^2 + \frac{\tau^2}{2} \left(H_1^n - H_2^n\right)+ \tau \mathcal{A}_\epsilon(q_r^n(\tilde{\bm{\mathcal{E}}}_h),\tilde{\bm{\mathcal{E}}}_h^n) \leq - \tau R_C(\tilde{\bm{\mathcal{E}}}_h^n) \\ 
&- \tau R_S(\tilde{\bm{\mathcal{E}}}_h^n) + \tau \mathcal{b}(\tilde{\bm{\mathcal{E}}}_h^n, p_h^{n-1} - p^n) - \tau \mu \mathcal{A}_\epsilon(\Pi_h \mathbf{u}^n - \mathbf{u}^n, \tilde{\bm{\mathcal{E}}}_h^n)-\tau \mathcal{A}_\epsilon(q_r^n(\Pi_h \mathbf{u}-\mathbf{u}),\tilde{\bm{\mathcal{E}}}_h^n) \\
&+ R_t(\tilde{\bm{\mathcal{E}}}_h^n) - \delta_{n,1} \tau \mathcal{b}(\bm{\mathcal{E}}_h^0, v_h^1).
\end{aligned}
\label{eq:6.59}
\end{equation}
To bound the third term on the RHS of \eqref{eq:6.59}, we split the term as
\begin{equation*}
\mathcal{b}(\tilde{\bm{\mathcal{E}}}_h^n, p_h^{n-1} - p^n) = \mathcal{b}(\tilde{\bm{\mathcal{E}}}_h^n, p_h^{n-1}) + \mathcal{b}(\tilde{\bm{\mathcal{E}}}_h^n, \pi_h p^n - p^n)- \mathcal{b}(\tilde{\bm{\mathcal{E}}}_h^n, \pi_h p^n) .
\end{equation*}
Since $\nabla_h \cdot \tilde{\bm{\mathcal{E}}}_h^n \in P_h$, using \eqref{eq6.101} and the trace inequalities, we obtain
\begin{equation}
\left|\mathcal{b}(\tilde{\bm{\mathcal{E}}}_h^n, \pi_h p^n - p^n) \right|= \left| \sum_{F \in \mathcal{F}_h} \int_F \{ \pi_h p^n - p^n \} [\tilde{\bm{\mathcal{E}}}_h^n] \cdot \mathbf{n}_F\right|\\
\leq \tilde{C} h^{2r} | p^n |^2_{H^r(\mathcal{O})} + \frac{\omega \mu}{16} \| \tilde{\bm{\mathcal{E}}}_h^n \|^2_{\text{dG}}. \label{eq:6.61}
\end{equation}
Since $p^n$ satisfies \eqref{eq3.4}, we observe that $\pi_h p^n \in P_h^0$. By \eqref{eq6.108}, continuity of $\mathcal{A}_{\text{sip}}(\cdot, \cdot)$ \eqref{eq4.24}, Young's inequality, and \eqref{eq6.102}, we have
\begin{equation}
\begin{aligned}
\big| \mathcal{b}(\tilde{\bm{\mathcal{E}}}_h^n, \pi_h p^n) \big| = \tau | \mathcal{A}_{\text{sip}} (v_h^n, \pi_h p^n) | &\leq C \tau | v_h^n |_{\text{dG}} | \pi_h p^n |_{\text{dG}} \\
&\leq \frac{\tau}{8} | v_h^n |_{\text{dG}}^2 + C \tau | p^n |^2_{H^1(\mathcal{O})}.
\label{eq:6.60}
\end{aligned}
\end{equation}
Thus, with bounds \eqref{eq:6.61} and \eqref{eq:6.60}, \eqref{eq:6.59} becomes
\begin{equation}
\begin{aligned}
\frac{1}{2} \left( \| \bm{\mathcal{E}}_h^n \|^2 - \| \bm{\mathcal{E}}_h^{n-1} \|^2 + \| \bm{\mathcal{E}}_h^n - \bm{\mathcal{E}}_h^{n-1} \|^2 \right) + \frac{\tau^2}{8} | v_h^n |_{\text{dG}}^2 + \frac{\tau^2}{2} \mathcal{A}_{\text{sip}} (v_h^n, v_h^n) + \omega \mu \tau \| \tilde{\bm{\mathcal{E}}}_h^n \|_{\text{dG}}^2 \\
 + \frac{\tau^2}{4} \sum_{F \in \mathcal{F}_h^i}  \frac{\tilde{\sigma}}{h_F} \| [v_h^n - v_h^{n-1}] \|_{L^2(F)}^2 + \frac{\tau^2}{2} \left(H_1^n - H_2^n\right)+ \tau \mathcal{A}_\epsilon(q_r^n(\tilde{\bm{\mathcal{E}}}_h),\tilde{\bm{\mathcal{E}}}_h^n)
\leq - \tau R_C(\tilde{\bm{\mathcal{E}}}_h^n) \\
- \tau R_S(\tilde{\bm{\mathcal{E}}}_h^n) + \tau \mathcal{b}(\tilde{\bm{\mathcal{E}}}_h^n,p_h^{n-1})+ C \tau^2 | p^n |^2_{H^1(\mathcal{O})} + \tilde{C} \tau h^{2r}|p^n|^2_{H^r(\mathcal{O})} - \tau \mu \mathcal{A}_\epsilon(\Pi_h \mathbf{u}^n - \mathbf{u}^n, \tilde{\bm{\mathcal{E}}}_h^n) \\
-\tau \mathcal{A}_\epsilon(q_r^n(\Pi_h \mathbf{u}-\mathbf{u}),\tilde{\bm{\mathcal{E}}}_h^n)
+ \frac{\omega \mu}{16} \tau \norm{\tilde{\bm{\mathcal{E}}}_h^n}_\text{dG}^2+ R_t(\tilde{\bm{\mathcal{E}}}_h^n) - \delta_{n,1} \tau \mathcal{b}(\bm{\mathcal{E}}_h^0, v_h^1). \label{6.129}
\end{aligned}
\end{equation}
Note that by \eqref{6.89}, $\mathcal{b}(\bm{\mathcal{E}}_h^n, p_h^{n-1}) = \mathcal{b}(\widetilde{\mathbf{u}}_h^n, p_h^{n-1})$. Hence, from \eqref{eq:6.75}, we have
\begin{equation}
\begin{aligned}
    \mathcal{b}(\tilde{\bm{\mathcal{E}}}_h^n, p_h^{n-1}) = -\frac{\tau}{2} \left(
    \mathcal{A}_{\text{sip}}(\xi_h^n, \xi_h^n) 
    - \mathcal{A}_{\text{sip}}(\xi_h^{n-1}, \xi_h^{n-1}) 
    - \mathcal{A}_{\text{sip}}(v_h^n, v_h^n) \right) \\
    - \frac{1}{2 \delta \mu} \left(
    \norm{S_h^n}^2 - \norm{S_h^{n-1}}^2 + \norm{S_h^n - S_h^{n-1}}^2 \right) \\
    + \frac{1}{ \mu} \left(\nabla_h \cdot q_r^n(\widetilde{\mathbf{u}}_h)-R_h([q_r^n(\widetilde{\mathbf{u}}_h)]),S_h^{n-1}\right). \label{eqn123}
\end{aligned}
\end{equation}
With the above expression \eqref{eqn123}, inequality \eqref{6.129} takes the form
\begin{equation}
\begin{aligned}
\frac{1}{2} \left( \| \bm{\mathcal{E}}_h^n \|^2 - \| \bm{\mathcal{E}}_h^{n-1} \|^2 + \| \bm{\mathcal{E}}_h^n - \bm{\mathcal{E}}_h^{n-1} \|^2 \right) 
+ \frac{\tau^2}{2} \left( \mathcal{A}_{\text{sip}}(\xi_h^n, \xi_h^n) - \mathcal{A}_{\text{sip}}(\xi_h^{n-1}, \xi_h^{n-1}) \right) + \frac{\tau^2}{8} | v_h^n |_{\text{dG}}^2  \\
+ \omega \mu \tau \| \tilde{\bm{\mathcal{E}}}_h^n \|_{\text{dG}}^2+ \frac{\tau^2}{2} \left( H_1^n - H_2^n \right) + \frac{\tau^2}{4} \sum_{F \in \mathcal{F}_h^i}  \frac{\tilde{\sigma}}{h_F} \| [v_h^n - v_h^{n-1}] \|_{L^2(F)}^2 + \frac{\tau}{2 \delta \mu} \left( \| S_h^n \|^2 - \| S_h^{n-1} \|^2\right)  \\
+ \tau \mathcal{A}_\epsilon(q_r^n(\tilde{\bm{\mathcal{E}}}_h),\tilde{\bm{\mathcal{E}}}_h^n) \leq - \tau R_C(\tilde{\bm{\mathcal{E}}}_h^n) 
- \tau R_S(\tilde{\bm{\mathcal{E}}}_h^n) - \tau \mu \mathcal{A}_\epsilon(\Pi_h \mathbf{u}^n - \mathbf{u}^n, \tilde{\bm{\mathcal{E}}}_h^n) -\tau \mathcal{A}_\epsilon(q_r^n(\Pi_h \mathbf{u}-\mathbf{u}),\tilde{\bm{\mathcal{E}}}_h^n) \\
+ C \tau^2 | p^n |^2_{H^1(\mathcal{O})}
+ \tilde{C}\tau h^{2r}|p^n|^2_{H^r(\mathcal{O})} 
+ \frac{\omega \mu}{16} \tau \norm{\tilde{\bm{\mathcal{E}}}_h^n}_\text{dG}^2
+ \frac{\tau}{ \mu} \left(\nabla_h \cdot q_r^n(\widetilde{\mathbf{u}}_h)-R_h([q_r^n(\widetilde{\mathbf{u}}_h)]),S_h^{n-1}\right) \\
+ \frac{\tau}{2 \delta \mu} \left( \| S_h^n - S_h^{n-1} \|^2 \right) + R_t(\tilde{\bm{\mathcal{E}}}_h^n) - \delta_{n,1} \tau \mathcal{b}(\bm{\mathcal{E}}_h^0, v_h^1).
\end{aligned}
\label{eq:6.64}
\end{equation}
We now consider $R_S(\tilde{\bm{\mathcal{E}}}_h^n)$. Since $\mathbf{u}$ is continuous, the two jump terms containing $\mathbf{u}$ in the bilinear form $\mathcal{A}_\epsilon(\cdot,\cdot)$ will vanish. We use the trace inequalities, and estimates \eqref{eq6.54} and \eqref{41}. Then using Young's inequality, we obtain
\begin{equation}
\begin{aligned}
&\left|\mathcal{A}_\epsilon(q_r^n(\mathbf{u}),\tilde{\bm{\mathcal{E}}}_h^n) - \int_0^{t_n} \beta(t_n-s) \mathcal{A}_\epsilon(\mathbf{u}(s), \tilde{\bm{\mathcal{E}}}_h^n) ds\right| \\ 
&\leq  \Bigg(C \tau \sum_{i=1}^n \int_{t_{i-1}}^{t_i}\bigg(\beta_s(t_n-s)\Big(\|\mathbf{u}(s)\|_{H^1(\mathcal{O})}+h\|\mathbf{u}(s)\|_{H^2(\mathcal{O})}\Big)
+\beta(t_n-s)\Big(\|\mathbf{u}_s(s)\|_{H^1(\mathcal{O})}\\
&\hspace{5em} +h\|\mathbf{u}_s(s)\|_{H^2(\mathcal{O})}\Big)\bigg)ds \Bigg)\|\tilde{\bm{\mathcal{E}}}_h^n\|_\text{dG}\\
&\leq \tilde{C} \tau^2 \int_0^{t_n} e^{-2\eta(t_n-t)}\left(\|\mathbf{u}(t)\|_{H^1(\mathcal{O})}^2+h^2\|\mathbf{u}(t)\|_{H^2(\mathcal{O})}^2+\|\mathbf{u}_t(t)\|_{H^1(\mathcal{O})}^2+h^2\|\mathbf{u}_t(t)\|_{H^2(\mathcal{O})}^2\right)dt\\
&\hspace{3em} + \frac{\omega \mu}{16}\|\tilde{\bm{\mathcal{E}}}_h^n\|_{\text{dG}}^2. \label{eqn125}
\end{aligned}
\end{equation}
Using Young's inequality, the coercivity of $\mathcal{A}_\epsilon(\cdot,\cdot)$ and \eqref{eq:6.97}, we obtain
\begin{equation}
\begin{aligned}
|\mathcal{A}_\epsilon(q_r^n(\Pi_h \mathbf{u} - \mathbf{u}),\tilde{\bm{\mathcal{E}}}_h^n)|  
 &\leq C \norm{q_r^n(\Pi_h \mathbf{u} - \mathbf{u})}_\text{dG} \norm{\tilde{\bm{\mathcal{E}}}_h^n}_\text{dG} \\
 &= C \tau \sum_{i=1}^n \beta(t_n - t_i) \norm{(\Pi_h \mathbf{u} - \mathbf{u})^i}_\text{dG} \norm{\tilde{\bm{\mathcal{E}}}_h^n}_\text{dG} \\
  &\leq \frac{\omega \mu }{16} \norm{\tilde{\bm{\mathcal{E}}}_h^n}_\text{dG}^2 + \tilde{C} \left(\tau \sum_{i=1}^n \beta(t_n - t_i) \norm{(\Pi_h \mathbf{u} - \mathbf{u})^i}_\text{dG}\right)^2 \\
 &\leq \frac{\omega \mu }{16} \norm{\tilde{\bm{\mathcal{E}}}_h^n}_\text{dG}^2 + \tilde{C} h^{2r} \left(\tau \sum_{i=1}^n \beta(t_n - t_i) |\mathbf{u}^i|_{H^{r+1}(\mathcal{O})}\right)^2 . \label{6.134}
\end{aligned}
\end{equation}
Assuming that $\delta \leq \frac{\omega \mu^2}{16 \gamma^2 d}$ and $\sigma \geq M_{r-1}^2 / d$, with \eqref{eq4.29} and \eqref{6.102}, we arrive at
\begin{equation}
\begin{aligned}
&\left |\frac{\tau}{\mu}\left(\nabla_h \cdot q_r^n(\widetilde{\mathbf{u}}_h)-R_h([q_r^n(\widetilde{\mathbf{u}}_h)]),S_h^{n-1}\right) \right | \\
&\leq \frac{\tau}{\mu}\norm{\nabla_h \cdot q_r^n(\tilde{\bm{\mathcal{E}}}_h)-R_h([q_r^n(\tilde{\bm{\mathcal{E}}}_h)])} \norm{S_h^{n-1}}\\
&\leq \frac{\delta}{2 \mu}\norm{\nabla_h \cdot q_r^n(\tilde{\bm{\mathcal{E}}}_h)-R_h([q_r^n(\tilde{\bm{\mathcal{E}}}_h)])}^2 + \frac{\tau^2}{2 \delta \mu} \norm{S_h^{n-1}}^2 \\
&\leq \frac{2 \gamma^2 \delta}{ \mu} \tau^2 \sum_{j=1}^n \left(d\norm{\nabla_h \tilde{\bm{\mathcal{E}}}_h^j}+ B_{r-1}^2 \sum_{F \in \mathcal{F}_h}h_F^{-1} \|[\tilde{\bm{\mathcal{E}}}_h^j]\|^2_{L^2(F)}\right) + \frac{\tau^2}{2 \delta \mu} \norm{S_h^{n-1}}^2 \\
&\leq \frac{\omega \mu }{8}\tau^2 \sum_{j=1}^n\left(\norm{\nabla_h \tilde{\bm{\mathcal{E}}}_h^j}+ \sum_{F \in \mathcal{F}_h}\sigma h_F^{-1} \|[\tilde{\bm{\mathcal{E}}}_h^j]\|^2_{L^2(F)}\right) + \frac{\tau^2}{2 \delta \mu} \norm{S_h^{n-1}}^2.
\label{6.135}
\end{aligned}
\end{equation}
With the definition of $S_h^n$ \eqref{eq6.55} and \eqref{6.102}, we have
\begin{equation*}
S_h^n - S_h^{n-1} = \delta \mu \left( \nabla_h \cdot \tilde{\bm{\mathcal{E}}}_h^n - R_h([\tilde{\bm{\mathcal{E}}}_h^n]) \right) + \delta \left(\nabla_h \cdot q_r^n(\tilde{\bm{\mathcal{E}}}_h)-R_h([q_r^n(\tilde{\bm{\mathcal{E}}}_h)])\right).
\label{6.136}
\end{equation*}
 Using \eqref{eq4.29} in the above, we have
\begin{equation}
\begin{aligned}
\|S^n_h - S^{n-1}_h\|^2 &\leq 2 \delta^2 \mu^2 \|\nabla_h \cdot \tilde{\bm{\mathcal{E}}}_h^n - R_h([\tilde{\bm{\mathcal{E}}}_h^n])\|^2 + 2 \delta^2 \gamma^2 \tau^2 \sum_{j=1}^n \norm{\nabla_h \cdot \tilde{\bm{\mathcal{E}}}_h^j - R_h(\tilde{\bm{\mathcal{E}}}_h^j)}^2\\
&\leq 4 \delta^2 \mu^2 \left(d\norm{\nabla_h \tilde{\bm{\mathcal{E}}}_h^n}+ B_{r-1}^2 \sum_{F \in \mathcal{F}_h}h_F^{-1} \|[\tilde{\bm{\mathcal{E}}}_h^n]\|^2_{L^2(F)}\right) \\
&\hspace{3em} +4 \delta^2 \gamma^2 \tau^2 \sum_{j=1}^n \left(d\norm{\nabla_h \tilde{\bm{\mathcal{E}}}_h^j}+ B_{r-1}^2 \sum_{F \in \mathcal{F}_h}h_F^{-1} \|[\tilde{\bm{\mathcal{E}}}_h^j]\|^2_{L^2(F)}\right). \label{6.137}
\end{aligned}
\end{equation}
With the assumptions that $\delta \leq \min\left(\frac{\omega}{8d}, \frac{\omega \mu^2}{16 \gamma^2 d}\right)$ and $\sigma \geq M_{r-1}^2 / d$, \eqref{6.137} reads
\begin{equation}
\begin{aligned}
\frac{1}{2\delta\mu}\|S^n_h - S^{n-1}_h\|^2 
\leq \frac{\omega \mu}{4}&\left(\norm{\nabla_h \tilde{\bm{\mathcal{E}}}_h^n}+ \sum_{F \in \mathcal{F}_h}\sigma h_F^{-1} \|[\tilde{\bm{\mathcal{E}}}_h^n]\|^2_{L^2(F)}\right) \\
&+\frac{\omega \mu}{8} \tau^2 \sum_{j=1}^n \left(\norm{\nabla_h \tilde{\bm{\mathcal{E}}}_h^j}+ \sum_{F \in \mathcal{F}_h}\sigma h_F^{-1} \|[\tilde{\bm{\mathcal{E}}}_h^j]\|^2_{L^2(F)}\right).
\label{6.138}
\end{aligned}
\end{equation}
To tackle $\mathcal{A}_\epsilon(\Pi_h \mathbf{u}^n - \mathbf{u}^n, \tilde{\bm{\mathcal{E}}}_h^n)$, we split it as
\begin{align*}
&\mathcal{A}_\epsilon(\Pi_h \mathbf{u}^n - \mathbf{u}^n, \tilde{\bm{\mathcal{E}}}_h^n) = \sum_{K \in \mathcal{K}_h} \int_K \nabla(\Pi_h \mathbf{u}^n - \mathbf{u}^n) \nabla \tilde{\bm{\mathcal{E}}}_h^n \\
 &- \sum_{F \in \mathcal{F}_h} \int_F \{ \nabla(\Pi_h \mathbf{u}^n - \mathbf{u}^n) \} \mathbf{n}_F \cdot [\tilde{\bm{\mathcal{E}}}_h^n] + \epsilon \sum_{F \in \mathcal{F}_h} \int_F \{ \nabla \tilde{\bm{\mathcal{E}}}_h^n \} \mathbf{n}_F \cdot [\Pi_h \mathbf{u}^n - \mathbf{u}^n] \\
 &+ \sum_{F \in \mathcal{F}_h} \int_F \frac{\sigma}{h_F} [\Pi_h \mathbf{u}^n - \mathbf{u}^n] \cdot [\tilde{\bm{\mathcal{E}}}_h^n] = T_1 + T_2 + T_3 + T_4.
\end{align*}
By using the trace inequalities, the terms $T_1$, $T_3$, and $T_4$ can be treated via standard arguments. We have
\begin{equation*}
|T_1| + |T_3| + |T_4| \leq \frac{\omega}{8} \| \tilde{\bm{\mathcal{E}}}_h^n \|_{\text{dG}}^2 + \frac{C}{\omega} h^{2r} | \mathbf{u}^n |_{H^{r+1}(\mathcal{O})}^2.
\end{equation*}
To tackle the term $T_2$, we employ a strategy similar to the one used in \cite{riviere2008discontinuous}, resulting in
\begin{equation*}
|T_2| \leq \frac{\omega}{16} \| \tilde{\bm{\mathcal{E}}}_h^n \|_{\text{dG}}^2 + \frac{C}{\omega} h^{2r} | \mathbf{u}^n |_{H^{r+1}(\mathcal{O})}^2.
\end{equation*}
Next, we consider the term $R_t(\tilde{\bm{\mathcal{E}}}_h^n)$. Using the CS inequality, Young's inequality, and \eqref{eq6.102}, we obtain
\begin{align*}
| R_t(\tilde{\bm{\mathcal{E}}}_h^n) | &\leq \tilde{C} \tau^{-1} \| \tau (\partial_t \mathbf{u})^n - (\mathbf{u}^n - \mathbf{u}^{n-1}) \|^2  \\
&\quad + \tilde{C} \tau^{-1} \| (\Pi_h \mathbf{u}^{n-1} - \mathbf{u}^{n-1}) - (\Pi_h \mathbf{u}^n - \mathbf{u}^n) \|^2 + \frac{\omega \mu}{16} \tau \| \tilde{\bm{\mathcal{E}}}_h^n \|_{\text{dG}}^2.
\end{align*}
Using Taylor expansion along with the approximation property \eqref{eq:6.96}, we find that
\begin{equation}
\begin{aligned}
| R_t(\tilde{\bm{\mathcal{E}}}_h^n) | &\leq \tilde{C} \tau^2\int_{t_{n-1}}^{t_n} \| \partial_{tt} \mathbf{u} \|^2 \, dt + \tilde{C} \int_{t_{n-1}}^{t_n} \| \partial_t (\Pi_h \mathbf{u} - \mathbf{u}) \|^2 \, dt + \frac{\omega \mu}{16} \tau \| \tilde{\bm{\mathcal{E}}}_h^n \|_{\text{dG}}^2 \\
&\leq \tilde{C} \tau^2 \int_{t_{n-1}}^{t_n} \| \partial_{tt} \mathbf{u} \|^2 \, dt + \tilde{C} h^{2r} \int_{t_{n-1}}^{t_n} | \partial_t \mathbf{u} |^2_{H^r(\mathcal{O})} \, dt + \frac{\omega \mu}{16} \tau \| \tilde{\bm{\mathcal{E}}}_h^n \|_{\text{dG}}^2. \label{6.133}
\end{aligned}
\end{equation}
We use the aforementioned bounds on $T_1$--$T_4$, the bound on the term $|R_C(\tilde{\bm{\mathcal{E}}}_h^n)|$ \eqref{eq:6.117}, the bound on $|R_S(\tilde{\bm{\mathcal{E}}}_h^n)|$ \eqref{eqn125}, \eqref{6.134}, \eqref{6.135}, \eqref{6.138}, and the bound on $|R_t(\bm{\mathcal{E}}_h^n)|$ \eqref{6.133} in \eqref{eq:6.64}. Using these bounds along with the assumption on regularity (\textbf{A2}), \eqref{eq:6.64} reads
\begin{equation}
\begin{aligned}
\frac{1}{2} \left( \| \bm{\mathcal{E}}_h^n \|^2 - \| \bm{\mathcal{E}}_h^{n-1} \|^2 + \| \bm{\mathcal{E}}_h^n - \bm{\mathcal{E}}_h^{n-1} \|^2 \right) 
+ \frac{\tau^2}{2} \left( \mathcal{A}_{\text{sip}}(\xi_h^n, \xi_h^n) - \mathcal{A}_{\text{sip}}(\xi_h^{n-1}, \xi_h^{n-1}) \right)
+ \frac{\tau^2}{8} | v_h^n |_{\text{dG}}^2 \\
+ \frac{\omega \mu}{4} \tau \| \tilde{\bm{\mathcal{E}}}_h^n \|_{\text{dG}}^2+ \frac{\tau^2}{2} \left( H_1^n - H_2^n \right) + \frac{\tau}{2 \delta \mu} \| S_h^n - S_h^{n-1} \|^2 + \tau \mathcal{A}_\epsilon(q_r^n(\tilde{\bm{\mathcal{E}}}_h),\tilde{\bm{\mathcal{E}}}_h^n)\\
\leq \tilde{C} \tau^2 \int_{t_{n-1}}^{t_n} \left( \| \partial_t \mathbf{u} \|^2+ \| \partial_{tt} \mathbf{u} \|^2 \right) dt
+ \tilde{C} h^{2r} \int_{t_{n-1}}^{t_n} | \partial_t \mathbf{u} |^2_{H^r(\mathcal{O})} dt
+\tilde{C}\tau h^{2r} + \tilde{C} \tau \| \bm{\mathcal{E}}_h^{n-1} \|^2 + C \tau^2 \\
 +\tilde{C} \tau^3 \int_0^{t_n} e^{-2\eta(t_n-t)}\left(\|\mathbf{u}(t)\|_{H^1(\mathcal{O})}^2+h^2\|\mathbf{u}(t)\|_{H^2(\mathcal{O})}^2+\|\mathbf{u}_t(t)\|_{H^1(\mathcal{O})}^2+h^2\|\mathbf{u}_t(t)\|_{H^2(\mathcal{O})}^2\right)dt\\
 - \delta_{n,1} \tau \mathcal{b}(\bm{\mathcal{E}}_h^0, v_h^1) +\frac{\omega \mu}{4} \tau^2 \sum_{i=1}^n \|\tilde{\bm{\mathcal{E}}}^i_h\|^2_\text{dG}
 + \frac{\tau^2}{2 \delta \mu} \norm{S_h^{n-1}}^2 
 + \tilde{C} h^{2r} \tau \left(\tau \sum_{i=1}^n \beta(t_n - t_i) |\mathbf{u}^i|_{H^{r+1}(\mathcal{O})}\right)^2. \label{6.142}
\end{aligned}
\end{equation}
Now, we consider the symmetric form $\mathcal{A}_d(\cdot,\cdot)$ first. 
We multiply \eqref{6.142} by 2, sum over $n=1$ to $m$, and employ the regularity assumption (\textbf{A2}) to have
\begin{equation}
\begin{aligned}
     \|\bm{\mathcal{E}}_h^m\|^2 + \sum_{n=1}^m \|\bm{\mathcal{E}}_h^n - \bm{\mathcal{E}}_h^{n-1}\|^2 + \tau^2 \mathcal{A}_{\text{sip}}(\xi_h^m, \xi_h^m) + \frac{\tau^2}{4} \sum_{n=1}^m |v_h^n|_{\text{dG}}^2 + \frac{\omega \mu}{2} \tau \sum_{n=1}^m \|\tilde{\bm{\mathcal{E}}}_h^n\|_{\text{dG}}^2  \\
     + \tau^2 \sum_{n=1}^m\left( H_1^n - H_2^n \right)
    + \frac{\tau}{\delta \mu} \|S_h^m\|^2 + 2 \tau \sum_{n=1}^m \mathcal{A}_d(q_r^n(\tilde{\bm{\mathcal{E}}}_h),\tilde{\bm{\mathcal{E}}}_h^n) \leq \tilde{C}\left(\tau + h^{2r}\right) 
 + \tilde{C} \tau \sum_{n=0}^{m-1} \|\bm{\mathcal{E}}_h^n\|^2 \\ 
 +\tilde{C} \tau^3 \sum_{n=1}^m \int_0^{t_n}e^{-2\eta(t_n-t)}\left(\|\mathbf{u}(t)\|_{H^1(\mathcal{O})}^2+h^2\|\mathbf{u}(t)\|_{H^2(\mathcal{O})}^2+\|\mathbf{u}_t(t)\|_{H^1(\mathcal{O})}^2+h^2\|\mathbf{u}_t(t)\|_{H^2(\mathcal{O})}^2\right)dt \\
 + \|\bm{\mathcal{E}}_h^0\|^2 + 2\tau |\mathcal{b}(\bm{\mathcal{E}}_h^0, v_h^1)|
 + \frac{\tau^2}{ \delta \mu} \sum_{n=1}^m \norm{S_h^{n-1}}^2 
 +\frac{\omega \mu}{2} \tau^2 \sum_{n=1}^m \sum_{i=1}^n \|\tilde{\bm{\mathcal{E}}}^i_h\|^2_\text{dG}\\
 + \tilde{C} h^{2r} \tau \sum_{n=1}^m \tau^2 \left(\sum_{i=1}^n \beta(t_n - t_i) |\mathbf{u}^i|_{H^{r+1}(\mathcal{O})}\right)^2. \label{6.146}
\end{aligned} 
\end{equation}
We note that $v_h^0 = 0$. Recalling the definitions of $H_1^n$ and $H_2^n$ (see \eqref{eq:6.115}--\eqref{eq:6.116}), we use \eqref{eq4.30} and the assumption $\tilde{\sigma} \geq \tilde{B}_r^2$ to get
\begin{equation}
\begin{aligned}
    \sum_{n=1}^m (H_1^n - H_2^n) &= \sum_{F \in \mathcal{F}_h^i}  \frac{\tilde{\sigma}}{h_F} \| [v_h^m] \|_{L^2(F)}^2 - \|\mathbf{G}_h([v_h^m])\|^2 \\
    &\geq \sum_{F \in \mathcal{F}_h^i}  \frac{\tilde{\sigma} - \tilde{B}_r^2}{h_F} \| [v_h^m] \|_{L^2(F)}^2 \geq 0. \label{6.147}
\end{aligned}
\end{equation}
Use of \eqref{eq4.37} and approximation properties yields
\begin{equation}
|\mathcal{b}(\bm{\mathcal{E}}_h^0, v_h^1)| \leq C \|\bm{\mathcal{E}}_h^0\| |v_h^1|_{\text{dG}} \leq C \tau^{-1} h^{2r+2} |\mathbf{u}^0|^2_{H^{r+1}(\mathcal{O})} + \frac{\tau}{8} |v_h^1|_{\text{dG}}^2. \label{6.148}
\end{equation}
By using the CS inequality and change of order of summation, we arrive at
\begin{equation}
 \tilde{C} h^{2r} \tau \sum_{n=1}^m \tau^2 \left(\sum_{i=1}^n \beta(t_n - t_i) |\mathbf{u}^i|_{H^{r+1}(\mathcal{O})}\right)^2
 \leq \tilde{C} h^{2r} \frac{\gamma^2 e^{2 \eta \tau}}{\eta^2} \tau \sum_{n=1}^m |\mathbf{u}^n|_{H^{r+1}(\mathcal{O})}^2. \label{6.149}
\end{equation} 
Under the assumption (\textbf{A2}), we observe that
\begin{equation}
\begin{aligned}
\tilde{C} \tau^3 \sum_{n=1}^m \int_0^{t_n}e^{-2\eta(t_n-t)}&\left(\|\mathbf{u}(t)\|_{H^1(\mathcal{O})}^2+h^2\|\mathbf{u}(t)\|_{H^2(\mathcal{O})}^2+\|\mathbf{u}_t(t)\|_{H^1(\mathcal{O})}^2+h^2\|\mathbf{u}_t(t)\|_{H^2(\mathcal{O})}^2\right)dt \\
&\leq \tilde{C} \tau^2 e^{2 \eta t_m}. \label{166}
\end{aligned}
\end{equation}
With the bounds \eqref{6.147}-\eqref{166} and approximation properties, we apply the positivity property of $\mathcal{A}_d$ \eqref{5.38} and the coercivity of $\mathcal{A}_{\text{sip}}$ \eqref{eq3.21} in \eqref{6.146}. Then use of the discrete Gronwall's inequality along with the triangle inequality yield bound \eqref{eq:6.118}. \\
For the non-symmetric form of $\mathcal{A}_\epsilon$, the last term in the LHS of \eqref{6.142} can be bounded using \eqref{eq4.23} and Young's inequality as
\begin{equation}
\begin{aligned}
|\mathcal{A}_\epsilon(q_r^n(\tilde{\bm{\mathcal{E}}}_h), \tilde{\bm{\mathcal{E}}}_h^n)| &\leq C \gamma \sum_{i=1}^n \beta(t_n-t_i) \|\tilde{\bm{\mathcal{E}}}_h^i\|_\text{dG} \|\tilde{\bm{\mathcal{E}}}_h^n\|_\text{dG} \\
&\leq \frac{\omega \mu}{8}\|\tilde{\bm{\mathcal{E}}}_h^n\|^2_\text{dG}+\tilde{C}\left(\tau \sum_{i=1}^n \beta(t_n-t_i) \|\tilde{\bm{\mathcal{E}}}_h^i\|_\text{dG} \right)^2 .
    \label{6.150}
    \end{aligned}
\end{equation}
Now we use \eqref{6.150} in inequality \eqref{6.142}. Multiplying the resulting expression by 2, sum over $n=1$ to $m$, and employ the regularity assumption (\textbf{A2}) to have
\begin{equation}
\begin{aligned}
     \|\bm{\mathcal{E}}_h^m\|^2 + \sum_{n=1}^m \|\bm{\mathcal{E}}_h^n - \bm{\mathcal{E}}_h^{n-1}\|^2 + \tau^2 \mathcal{A}_{\text{sip}}(\xi_h^m, \xi_h^m) + \frac{\tau^2}{4} \sum_{n=1}^m |v_h^n|_{\text{dG}}^2 + \frac{\omega \mu}{4} \tau \sum_{n=1}^m \|\tilde{\bm{\mathcal{E}}}_h^n\|_{\text{dG}}^2 \\
    + \tau^2 \sum_{n=1}^m\left( H_1^n - H_2^n \right)  
    + \frac{\tau}{\delta \mu} \|S_h^m\|^2 \leq \tilde{C}\left(\tau + h^{2r}\right) 
 + \tilde{C} \tau \sum_{n=0}^{m-1} \|\bm{\mathcal{E}}_h^n\|^2 + \frac{\tau^2}{ \delta \mu} \sum_{n=1}^m \norm{S_h^{n-1}}^2 \\
 +\tilde{C} \tau^3 \sum_{n=1}^m \int_0^{t_n}e^{-2\eta(t_n-t)}\left(\|\mathbf{u}(t)\|_{H^1(\mathcal{O})}^2+h^2\|\mathbf{u}(t)\|_{H^2(\mathcal{O})}^2+\|\mathbf{u}_t(t)\|_{H^1(\mathcal{O})}^2+h^2\|\mathbf{u}_t(t)\|_{H^2(\mathcal{O})}^2\right)dt \\
 + \|\bm{\mathcal{E}}_h^0\|^2 + 2\tau |\mathcal{b}(\bm{\mathcal{E}}_h^0, v_h^1)|
 +\frac{\omega \mu}{2} \tau^2 \sum_{n=1}^m \sum_{i=1}^n \|\tilde{\bm{\mathcal{E}}}^i_h\|^2_\text{dG} 
 + \tilde{C} h^{2r} \tau \sum_{n=1}^m \tau^2 \left(\sum_{i=1}^n \beta(t_n - t_i) |\mathbf{u}^i|_{k+1}\right)^2\\
 +\tilde{C}\tau \sum_{n=1}^m\left(\tau \sum_{i=1}^n \beta(t_n-t_i) \|\tilde{\bm{\mathcal{E}}}_h^i\|_\text{dG} \right)^2. \label{6.151}
\end{aligned} 
\end{equation}
The last term in the RHS of \eqref{6.151} is bounded using Holder's inequality as follows
\begin{equation}
    \begin{aligned}
        \tau \sum_{n=1}^m \left( \tau \sum_{i=1}^n \beta(t_n-t_i) \|\tilde{\bm{\mathcal{E}}}_h^i\|_\text{dG} \right)^2 &\leq \gamma^2 \tau \sum_{n=1}^m \left(\tau \sum_{i=1}^n e^{-2\eta(t_n-t_i)}\right) \left(\tau \sum_{i=1}^n \|\tilde{\bm{\mathcal{E}}}_h^i\|^2_\text{dG}\right)\\
        & \leq \frac{\gamma^2 e^{2\eta \tau}}{2\eta} \tau^2 \sum_{n=1}^m \sum_{i=1}^n \|\tilde{\bm{\mathcal{E}}}_h^i\|^2_\text{dG} . \label{6.152}
    \end{aligned}
\end{equation} 
Next, we apply the bounds \eqref{6.147}-\eqref{166}, \eqref{6.152} in \eqref{6.151}. Then use of the discrete Gronwall's inequality, approximation properties, the coercivity of $\mathcal{A}_\text{sip}$, and the triangle inequality yield the bound \eqref{eq:6.118}.
\end{proof}
\begin{lemma}
Under the assumptions of Theorem 2, we obtain the following estimate for $1 \leq m \leq N$:
\begin{align*}
    \|\bm{\mathcal{E}}_h^m\|^2+\sum_{n=1}^m \|\bm{\mathcal{E}}_h^n - \bm{\mathcal{E}}_h^{n-1}\|^2 + \frac{1}{4} \tau^2 \sum_{n=1}^m |v_h^n|_{\text{dG}}^2 + \frac{\omega \mu}{4} \tau \sum_{n=1}^m \|\tilde{\bm{\mathcal{E}}}_h^n\|_{\text{dG}}^2
    \leq \tilde{C}_T\left(\tau + h^{2r}\right).
\end{align*}
\end{lemma}
This lemma will be used repeatedly in the subsequent sections.
\begin{remark} As a consequence of the Lemma 11, with \eqref{eq5.41} and \eqref{eq4.30}, we obtain
\begin{align*}
    \|\widetilde{\mathbf{u}}_h^n - \mathbf{u}^n\|^2 &\leq 2 \|\widetilde{\mathbf{u}}_h^n - \mathbf{u}_h^n\|^2 + 2 \|\mathbf{u}_h^n - \mathbf{u}^n\|^2 \\
    &\leq C \tau^2 |v_h^n|_{\text{dG}}^2 + 2 \|\mathbf{u}_h^n - \mathbf{u}^n\|^2 \leq \tilde{C}_T\left(\tau + h^{2r}\right).
\end{align*}
\end{remark}
\begin{remark}
In view of the above bound, we apply \eqref{eq5.41} and the inverse estimate \eqref{8.159} to derive the following bound:
\begin{align*}
    \mu \tau \sum_{n=1}^m \|\mathbf{u}_h^n - \mathbf{u}^n\|_{\text{dG}}^2 &\leq \mu \tau \sum_{n=1}^m \|\widetilde{\mathbf{u}}_h^n - \mathbf{u}^n\|_{\text{dG}}^2 +\mu \tau \sum_{n=1}^m \tau^2 \|\nabla_h v_h^n + \mathbf{G}_h([v_h^n])\|_{\text{dG}}^2 \\
    &\leq \tilde{C}_T\left(\tau + h^{2r}\right) + C \mu\tau h^{-2} \sum_{n=1}^m \tau^2 |v_h^n|_{\text{dG}}^2.
\end{align*}
Thus, provided the CFL condition, $\tau \leq h^2 $, we have
\begin{equation*}
\mu \tau \sum_{n=1}^m \norm{\mathbf{u}_h^n - \mathbf{u}^n}_\text{dG}^2 \leq \tilde{C}_T\left(\tau + h^{2r}\right).
\end{equation*}
\end{remark}

\section{Improved \textit{a priori} estimate for the velocity}
From this point, we will use only the symmetric form $\mathcal{A}_d(\cdot,\cdot)$ for the discretization of the elliptic operator $-\Delta \widetilde{\mathbf{u}}$ defined in \eqref{18}. A crucial step in obtaining refined error estimates is the construction of an appropriate dual problem along with its discretization. Accordingly, we introduce the error functions, $\forall n \geq 1$,
\begin{equation*}
\bm\psi(t) = \mathbf{u}_h^n - \mathbf{u}(t),  \quad \bm\psi(0) = \mathbf{u}_h^0 - \mathbf{u}^0,
\end{equation*}
where $t \in (t_{n-1},t_n]$.
For $t \geq 0$, let $(\mathbf{V}(t), P(t)) \in (H_0^1(\mathcal{O}))^d \times L_0^2(\mathcal{O})$ be the solution of the dual Stokes problem given by:
\begin{align}
    -\Delta \mathbf{V}(t) + \nabla P(t) &= \bm\psi(t) \quad \text{in } \mathcal{O}, \label{8.164}
\\
    \nabla \cdot \mathbf{V}(t) &= 0 \quad \text{in } \mathcal{O}, \label{8.165}
\\
    \mathbf{V}(t) &= 0 \quad \text{on } \partial \mathcal{O}. \label{8.166}
\end{align}
As $\bm\psi(t) \in (L^2(\mathcal{O}))^d$ and the domain is convex, we have
\begin{equation}
\|\mathbf{V}(t)\|_{H^2(\mathcal{O})} + \|P(t)\|_{H^1(\mathcal{O})} \leq C \|\bm\psi(t)\|. \label{8.167}
\end{equation}
Furthermore, for $t \geq 0$, let $(\mathbf{V}_h(t), P_h(t)) \in \mathbf{M}_h \times P_h^0$ be the dG solution corresponding to the dual problem \eqref{8.164}--\eqref{8.166}:
\begin{equation}
    \mathcal{A}_d(\mathbf{V}_h(t), \mathbf{w}_h) - \mathcal{b}(\mathbf{w}_h, P_h(t)) = (\bm\psi(t), \mathbf{w}_h), \quad \forall \mathbf{w}_h \in \mathbf{M}_h, \label{8.168}
\end{equation}
\begin{equation}
    \mathcal{b}(\mathbf{V}_h(t), g_h) = 0, \quad \forall g_h \in P_h^0.  \label{8.169}
\end{equation}
The existence and uniqueness of $(\mathbf{V}_h(t), P_h(t))$ have been rigorously established in \cite{riviere2008discontinuous}. We now outline the stability and error estimates associated with the dG solution of the auxiliary problem.
\begin{lemma} [\cite{masri2023improved}]
For all $t \in [0, T]$, we have
\begin{align}
    |||\mathbf{V}(t)||| &\leq C \|\bm\psi(t)\|, \label{8.170}
\\
    \|\mathbf{V}(t) - \mathbf{V}_h(t)\| + h \|\mathbf{V}(t) - \mathbf{V}_h(t)\|_\text{dG} &+ h \|P(t) - P_h(t)\| \leq C h^2\|\bm\psi(t)\|, \label{8.171}
\\
    \|\mathbf{V}_h(t)\|_{L^\infty(\mathcal{O})} &\leq C \|\bm\psi(t)\|. \label{8.172}
\end{align}
\end{lemma}
\begin{theorem}
Under the same assumptions of Theorem 2, we have $\forall 0 \leq m \leq N$,
\begin{equation*}
\mu \tau \sum_{n=1}^m \| \mathbf{u}_h^n - \mathbf{u}^n \|^2 + \mu \tau \sum_{n=1}^m \| \widetilde{\mathbf{u}}_h^n - \mathbf{u}^n \|^2 \leq \tilde{C}_T \left(\tau^2 + \tau h^2 + h^{2r+2}\right). 
\end{equation*}
\end{theorem}
\begin{proof}
Multiplying \eqref{6.98} by $\tau$, subtracting it from \eqref{eq5.36}, and choosing $\mathbf{w}_h = \mathbf{V}_h^n$, we have
\begin{align}
(\widetilde{\mathbf{u}}_h^n - \mathbf{u}^n - \bm\psi^{n-1}, \mathbf{V}_h^n) + \tau \tilde{R}_C(\mathbf{V}_h^n) + \tau \mu \mathcal{A}_d(\widetilde{\mathbf{u}}_h^n - \mathbf{u}^n, \mathbf{V}_h^n) + \tau \mathcal{A}_d(q_r^n(\widetilde{\mathbf{u}}_h-\mathbf{u}),  \mathbf{V}_h^n) + \tau \tilde{R}_S(\mathbf{V}_h^n)\notag \\
= \tau \mathcal{b}(\mathbf{V}_h^n, p_h^{n-1} - p^n) + (\tau (\partial_t \mathbf{u})^n - (\mathbf{u}^n - \mathbf{u}^{n-1}), \mathbf{V}_h^n), \label{8.175}
\end{align}
where
\begin{align*}
\tilde{R}_C(\mathbf{V}_h^n) &= \mathcal{A}_c(\mathbf{u}_h^{n-1};\mathbf{u}_h^{n-1},\widetilde{\mathbf{u}}_h^n,\mathbf{V}_h^n) - \mathcal{A}_c(\mathbf{u}^n; \mathbf{u}^n, \mathbf{u}^n, \mathbf{V}_h^n),\\
\tilde{R}_S(\mathbf{V}_h^n) &= \mathcal{A}_d(q_r^n(\mathbf{u}), \mathbf{V}_h^n) -\int_0^{t_n} \beta(t_n-s) \mathcal{A}_d(\mathbf{u}(s),\mathbf{V}_h^n)ds.
\end{align*}
For the rest of our analysis, let $\varepsilon$ represent a small positive constant that will be specified later. We start by considering the two terms in the RHS of \eqref{8.175}. Noting that $\mathbf{V}_h^n$ satisfies \eqref{8.169}, we derive
\begin{equation*}
\left|\mathcal{b}\left(\mathbf{V}_h^n, p_h^{n-1} - p^n\right)\right| = \left|\mathcal{b}\left(\mathbf{V}_h^n, p^n - \pi_h p^n\right) \right|= \left|\sum_{F \in \mathcal{F}_h} \int_F \{p^n - \pi_h p^n\} [\mathbf{V}_h^n] \cdot \mathbf{n}_F\right|.
\end{equation*}
By applying a trace inequality, the approximation properties \eqref{eq6.101}, \eqref{8.171}, and observing that $[\mathbf{V}^n] = 0$ almost everywhere $\forall F \in \mathcal{F}_h$ as $\mathbf{V}^n \in (H_0^2(\mathcal{O}))^d$, we have
\begin{equation}
|\mathcal{b}\left(\mathbf{V}_h^n, p_h^{n-1} - p^n\right)| \leq C h^{r+1} \|p^n\|_{H^r(\mathcal{O})} \|\bm\psi^n\| \leq \varepsilon \mu \|\bm{\mathcal{E}}_h^n\|^2 + C h^{2r+2} \left(1 + \frac{1}{\varepsilon \mu}\right). \label{8.179}
\end{equation}
In the above derivation, we have used the following estimate:
\begin{equation*}
\|\bm\psi^n\| \leq C \left(h^{r+1} \|\mathbf{u}^n\|_{H^{r+1}(\mathcal{O})} + \|\bm{\mathcal{E}}_h^n\|\right), \label{8.180}
\end{equation*}
which results from applying the triangle inequality and the approximation property \eqref{eq:6.96}. For the final term in \eqref{8.175}, we have the simple bound
\begin{equation}
\left|\tau(\partial_t \mathbf{u})^n - (\mathbf{u}^n - \mathbf{u}^{n-1}), \mathbf{V}_h^n)\right| \leq C \tau^2 \int_{t_{n-1}}^{t_n} \|\partial_{tt} \mathbf{u}\|^2 dt + \tau \|\mathbf{V}_h^n\|^2_\text{dG}. \label{8.181}
\end{equation}
Next, we examine the terms on the LHS of \eqref{8.175}. Using \eqref{eq5.39} and \eqref{8.169}, we find
\begin{equation*}
(\widetilde{\mathbf{u}}_h^n - \mathbf{u}^n - \bm\psi^n, \mathbf{V}_h^n) = (\bm\psi^n - \bm\psi^{n-1}, \mathbf{V}_h^n) - \tau \mathcal{b}\left(\mathbf{V}_h^n, v_h^n\right) = (\bm\psi^n - \bm\psi^{n-1}, \mathbf{V}_h^n).
\end{equation*}
From \eqref{8.168}, \eqref{8.169}, the symmetry of $\mathcal{A}_d(\cdot, \cdot)$, along with the above equality, we have
\begin{align}
(\widetilde{\mathbf{u}}_h^n - \mathbf{u}^n - \bm\psi^{n-1}, \mathbf{V}_h^n) = \mathcal{A}_d\left(\mathbf{V}_h^n - \mathbf{V}_h^{n-1}, \mathbf{V}_h^n\right) = \frac{1}{2} \left(\mathcal{A}_d(\mathbf{V}_h^n, \mathbf{V}_h^n) - \mathcal{A}_d(\mathbf{V}_h^{n-1}, \mathbf{V}_h^{n-1})\right) \notag \\ 
+ \frac{1}{2} \mathcal{A}_d\left(\mathbf{V}_h^n - \mathbf{V}_h^{n-1}, \mathbf{V}_h^n - \mathbf{V}_h^{n-1}\right). \label{8.183}
\end{align}
Additionally, we write
\begin{equation}
\mathcal{A}_d\left(\widetilde{\mathbf{u}}_h^n - \mathbf{u}^n, \mathbf{V}_h^n\right) = \mathcal{A}_d\left(\tilde{\bm{\mathcal{E}}}_h^n - \bm{\mathcal{E}}_h^n, \mathbf{V}_h^n\right) + \mathcal{A}_d\left(\bm{\mathcal{E}}_h^n, \mathbf{V}_h^n\right) + \mathcal{A}_d\left(\Pi_h \mathbf{u}^n - \mathbf{u}^n, \mathbf{V}_h^n\right). \label{8.184}
\end{equation}
To tackle the final term in \eqref{8.184}, we take $\mathbf{w}_h = \Pi_h \mathbf{u}^n - Q_h \mathbf{u}^n$ in \eqref{8.168}, to obtain
\begin{align*}
\mathcal{A}_d\left(\Pi_h \mathbf{u}^n - \mathbf{u}^n, \mathbf{V}_h^n\right) &= \mathcal{A}_d\left(\Pi_h \mathbf{u}^n -Q_h \mathbf{u}^n, \mathbf{V}_h^n\right) \notag \\
&= (\bm\psi^n, \Pi_h \mathbf{u}^n - Q_h \mathbf{u}^n) + \mathcal{b}\left(\Pi_h \mathbf{u}^n - Q_h \mathbf{u}^n, P_h^n\right).
\end{align*}
Using \eqref{eq4.37}, we derive
\begin{equation}
|\mathcal{b}\left(\Pi_h \mathbf{u}^n - Q_h \mathbf{u}^n, P_h^n\right)| \leq C \|\Pi_h \mathbf{u}^n - Q_h \mathbf{u}^n\| |P_h^n|_\text{dG}. \label{eqn155}
\end{equation}
Applying \eqref{eq6.102}, the inverse estimate \eqref{8.160}, \eqref{eq6.101}, \eqref{8.167}, \eqref{8.171}, and the triangle inequality, we obtain
\begin{equation}
|P_h^n|_\text{dG} \leq |P_h^n - \pi_h P^n|_\text{dG} + |\pi_h P^n|_\text{dG} \leq C h^{-1} \|P_h^n - \pi_h P^n\| + C |P^n|_{H^1(\mathcal{O})}. \label{8.188}
\end{equation}
Hence,
\begin{equation}
|P_h^n|_\text{dG} \leq C \left(\|\bm{\mathcal{E}}_h^n\| + h^{r+1} |\mathbf{u}^n|_{H^{r+1}(\mathcal{O})}\right). \label{8.189}
\end{equation}
Using the bounds \eqref{eqn155}, \eqref{8.189}, the CS inequality, Young's inequality, and the assumption on regularity (\textbf{A2}),  we arrive at
\begin{align}
|\mathcal{A}_d\left(\Pi_h \mathbf{u}^n - \mathbf{u}^n, \mathbf{V}_h^n\right)| &\leq C \|\Pi_h \mathbf{u}^n - Q_h \mathbf{u}^n\| \left(\|\bm{\mathcal{E}}_h^n\| + h^{r+1} \|\mathbf{u}^n\|_{H^{r+1}(\mathcal{O})}\right) \notag \\
&\leq C h^r |\mathbf{u}^n|_{H^{r+1}(\mathcal{O})} \left(\|\bm{\mathcal{E}}_h^n\| + h^{r+1} \|\mathbf{u}^n\|_{H^{r+1}(\mathcal{O})}\right) \notag \\
&\leq \varepsilon \|\bm{\mathcal{E}}_h^n\|^2 + C \left(\frac{1}{\varepsilon} + 1\right) h^{2r+2}. \label{8.190}
\end{align}
Next, consider the second term in \eqref{8.184}. Set $\mathbf{w}_h = \bm{\mathcal{E}}_h^n$ in \eqref{8.168}, to have
\begin{equation}
\mathcal{A}_d\left(\bm{\mathcal{E}}_h^n, \mathbf{V}_h^n\right) = \|\bm{\mathcal{E}}_h^n\|^2 
+ \mathcal{b}\left(\bm{\mathcal{E}}_h^n, P_h^n\right)+ (\bm\psi^n - \bm{\mathcal{E}}_h^n, \bm{\mathcal{E}}_h^n) . \label{8.191}
\end{equation}
Combining \eqref{8.179}, \eqref{8.181}, \eqref{8.183}, \eqref{8.184}, \eqref{8.190}, and \eqref{8.191}, the equality \eqref{8.175} reads
\begin{equation}
\begin{aligned}
&\frac{1}{2} \left(\mathcal{A}_d(\mathbf{V}_h^n, \mathbf{V}_h^n) - \mathcal{A}_d(\mathbf{V}_h^{n-1}, \mathbf{V}_h^{n-1}) + \mathcal{A}_d(\mathbf{V}_h^n - \mathbf{V}_h^{n-1}, \mathbf{V}_h^n - \mathbf{V}_h^{n-1})\right) 
+ \tau \|\bm{\mathcal{E}}_h^n\|^2 \\
&\leq 2 \varepsilon \tau \mu \|\bm{\mathcal{E}}_h^n\|^2 + C\tau \|\mathbf{V}_h^n\|_\text{dG}^2+ C \tau^2 \int_{t_{n-1}}^{t_n} \|\partial_{tt} \mathbf{u}\|^2 dt + C \left(\mu \left(1 + \frac{1}{\varepsilon}\right) + 1 + \frac{1}{\varepsilon \mu}\right) \tau h^{2r+2}\\
&\hspace{1em} -\tau \tilde{R}_C(\mathbf{V}_h^n)-\tau \tilde{R}_S(\mathbf{V}_h^n)- \tau \mathcal{A}_d(q_r^n(\widetilde{\mathbf{u}}_h-\mathbf{u}),  \mathbf{V}_h^n)-\tau \mu \mathcal{A}_d(\tilde{\bm{\mathcal{E}}}_h^n-\bm{\mathcal{E}}_h^n, \mathbf{V}_h^n)\\
&\hspace{1em} -\tau \mu (\Pi_h \mathbf{u}^n-\mathbf{u}^n, \bm{\mathcal{E}}_h^n) -\tau \mu \mathcal{b}(\bm{\mathcal{E}}_h^n,P_h^n). \label{8.192}
\end{aligned}
\end{equation}
In view of \eqref{eq:6.96}, we obtain
\begin{equation}
| \left( \Pi_h \mathbf{u}^n - \mathbf{u}^n, \bm{\mathcal{E}}_h^n \right) | \leq \varepsilon \| \bm{\mathcal{E}}_h^n \|^2 + \frac{C}{\varepsilon} h^{2r+2} | \mathbf{u}^n |_{H^{r+1}(\mathcal{O})}^2. \label{eqn160}
\end{equation}
Next, setting $\mathbf{w}_h = \tilde{\bm{\mathcal{E}}}_h^n - \bm{\mathcal{E}}_h^n$ in \eqref{8.168} to have
\begin{equation}
\mathcal{A}_d\left(\tilde{\bm{\mathcal{E}}}_h^n - \bm{\mathcal{E}}_h^n, \mathbf{V}_h^n\right) = \left(\bm{\mathcal{E}}_h^n, \tilde{\bm{\mathcal{E}}}_h^n - \bm{\mathcal{E}}_h^n\right) + \mathcal{b}\left(\tilde{\bm{\mathcal{E}}}_h^n - \bm{\mathcal{E}}_h^n, P_h^n\right)+ \left(\Pi_h \mathbf{u}^n - \mathbf{u}^n, \tilde{\bm{\mathcal{E}}}_h^n - \bm{\mathcal{E}}_h^n\right) .
\end{equation}
Use of \eqref{eq:6.110}, \eqref{6.113}, \eqref{eq4.30}, and assuming that $\tilde{\sigma} \geq 4 \tilde{B}_r^2$, it follows that
\begin{equation}
\left(\bm{\mathcal{E}}_h^n, \tilde{\bm{\mathcal{E}}}_h^n - \bm{\mathcal{E}}_h^n\right) = -\tau \mathcal{b}\left(\bm{\mathcal{E}}_h^n, v_h^n\right) = \tau^2 \sum_{F \in \mathcal{F}_h^i}  \frac{\tilde{\sigma}}{h_F} \|v_h^n\|_{L^2(F)}^2 - \tau^2 \|\mathbf{G}_h(v_h^n)\|^2 \geq 0. \label{eqn:162}
\end{equation}
Utilizing the CS inequality, \eqref{eq4.37}, and \eqref{eq:6.96}, we obtain
\begin{equation}
 \left| \mathcal{b} ( \tilde{\bm{\mathcal{E}}}_h^n - \bm{\mathcal{E}}_h^n, P_h^n ) \right| +\left| ( \Pi_h \mathbf{u}^n - \mathbf{u}^n, \tilde{\bm{\mathcal{E}}}_h^n - \bm{\mathcal{E}}_h^n ) \right|  
\leq C \| \tilde{\bm{\mathcal{E}}}_h^n - \bm{\mathcal{E}}_h^n \| \left(| P_h^n |_\text{dG} + h^{r+1} | \mathbf{u}^n |_{H^{r+1}(\mathcal{O})} \right). \label{8.195}
\end{equation}
With an aid of \eqref{eq4.37} and \eqref{eq:6.110}, we find that
\begin{equation}
\|\tilde{\bm{\mathcal{E}}}_h^n - \bm{\mathcal{E}}_h^n \| \leq C \tau \| v_h^n \|_\text{dG}.  \label{8.196}
\end{equation}
Thus, using \eqref{8.189}, \eqref{8.196}, the assumption on regularity (\textbf{A2}), and Young's inequality, the bound \eqref{8.195} yields
\begin{equation}
\begin{aligned}
\left| \mathcal{b} ( \tilde{\bm{\mathcal{E}}}_h^n - \bm{\mathcal{E}}_h^n, P_h^n ) \right| +\left| ( \Pi_h \mathbf{u}^n - \mathbf{u}^n, \tilde{\bm{\mathcal{E}}}_h^n - \bm{\mathcal{E}}_h^n ) \right| 
&\leq C \tau | v_h^n |_\text{dG} \left( \| \bm{\mathcal{E}}_h^n \| + h^{r+1} | \mathbf{u}^n |_{H^{r+1}(\mathcal{O})}  \right) \\
&\leq \varepsilon \| \bm{\mathcal{E}}_h^n \|^2 + C h^{2r+2} + C \left( 1+\frac{1}{\varepsilon} \right) \tau^2 | v_h^n |_\text{dG}^2 . \label{eqn:165}
\end{aligned}
\end{equation}
Next, we bound the last term in \eqref{8.192}. Using \eqref{6.113}, \eqref{eq4.30}, and \eqref{8.189}, we have
\begin{equation}
\begin{aligned}
| \mathcal{b} \left( \bm{\mathcal{E}}_h^n, P_h^n \right) | &= \left|\tau \left( \mathbf{G}_h \left( \left[ v_h^n \right] \right), \mathbf{G}_h \left( \left[ P_h^n \right] \right) \right)- \tau \sum_{F \in \mathcal{F}_h^i}  \frac{\tilde{\sigma}}{h_F} \int_F [ v_h^n ] [ P_h^n ]  \right| \\
&\leq C \tau | P_h^n |_\text{dG} | v_h^n |_\text{dG}  \leq \varepsilon \| \bm{\mathcal{E}}_h^n \|^2 + C h^{2r+2} + C \left( 1 + \frac{1}{\varepsilon} \right) \tau^2 | v_h^n |_\text{dG}^2 . \label{eqn:166}
\end{aligned}
\end{equation}
Combining the above bounds \eqref{eqn160}-\eqref{eqn:162}, \eqref{eqn:165}-\eqref{eqn:166}, inequality \eqref{8.192} becomes
\begin{equation}
\begin{aligned}
\frac{1}{2} \left( \mathcal{A}_d \left( \mathbf{V}_h^n, \mathbf{V}_h^n \right) - \mathcal{A}_d \left( \mathbf{V}_h^{n-1}, \mathbf{V}_h^{n-1} \right) + \mathcal{A}_d \left( \mathbf{V}_h^n - \mathbf{V}_h^{n-1}, \mathbf{V}_h^n - \mathbf{V}_h^{n-1} \right) \right)+ \tau \mu \| \bm{\mathcal{E}}_h^n \|^2 \\
\leq 5 \varepsilon \tau \mu \| \bm{\mathcal{E}}_h^n \|^2 + C \mu \left( 1 + \frac{1}{\varepsilon} \right) \tau^3 | v_h^n |_\text{dG}^2 - \tau \tilde{R}_C \left( \mathbf{V}_h^n \right)- \tau \tilde{R}_S \left( \mathbf{V}_h^n \right)- \tau \mathcal{A}_d(q_r^n(\widetilde{\mathbf{u}}_h-\mathbf{u}),  \mathbf{V}_h^n) \\
+ C \tau \| \mathbf{V}_h^n \|_\text{dG}^2 +C \tau^2 \int_{t_{n-1}}^{t_n} \|\partial_{tt}\mathbf{u}^n\|^2 dt + C \left( \mu \left( 1 + \frac{1}{\varepsilon} \right) + 1 + \frac{1}{\varepsilon \mu} \right) \tau h^{2r+2}. \label{8.200}
\end{aligned}
\end{equation}
The bound for the nonlinear term $\tilde{R}_C \left( \mathbf{V}_h^n \right)$ is given in Lemma 13 below. \\
Next, we bound the term $ \tilde{R}_S(\mathbf{V}_h^n)$ as in Theorem 2. 
\begin{equation}
\begin{aligned}
\left|\tilde{R}_S(\mathbf{V}_h^n)\right| &= \left| \mathcal{A}_d(q_r^n(\mathbf{u}), \mathbf{V}_h^n) -\int_0^{t_n} \beta(t_n-s) \mathcal{A}_d(\mathbf{u}(s),\mathbf{V}_h^n)ds \right| \\
&\leq \tilde{C} \tau^2 \int_0^{t_n} e^{-2\eta(t_n-t)}\left(\|\mathbf{u}(t)\|_{H^1(\mathcal{O})}^2+h^2\|\mathbf{u}(t)\|_{H^2(\mathcal{O})}^2+\|\mathbf{u}_t(t)\|_{H^1(\mathcal{O})}^2+h^2\|\mathbf{u}_t(t)\|_{H^2(\mathcal{O})}^2\right)dt\\
&\hspace{2em} + \tilde{C}\|\mathbf{V}_h^n\|_{\text{dG}}^2. \label{eqn162}
\end{aligned}
\end{equation}
To bound the fifth term in the RHS of \eqref{8.200}, we now write
\begin{equation}
\begin{aligned}
&\tau \mathcal{A}_d(q_r^n(\widetilde{\mathbf{u}}_h-\mathbf{u}),\mathbf{V}_h^n) \\
&= \tau \mathcal{A}_d(q_r^n(\tilde{\bm{\mathcal{E}}}_h - \bm{\mathcal{E}}_h),\mathbf{V}_h^n) + \tau \mathcal{A}_d(q_r^n(\bm{\mathcal{E}}_h),\mathbf{V}_h^n) + \tau \mathcal{A}_d(q_r^n(\Pi_h \mathbf{u} - \mathbf{u}),\mathbf{V}_h^n) \\
&=P_1 + P_2 + P_3. \label{8.201}
\end{aligned}
\end{equation}
To bound the term $P_1$, letting $\mathbf{w}_h = q_r^n(\tilde{\bm{\mathcal{E}}}_h - \bm{\mathcal{E}}_h)$ in \eqref{8.168} and multiplying the resulting expression by $\tau$, we obtain
\begin{equation}
\begin{aligned}
\tau \mathcal{A}_d(q_r^n(\tilde{\bm{\mathcal{E}}}_h& - \bm{\mathcal{E}}_h),\mathbf{V}_h^n) = \tau \mathcal{b}(q_r^n(\tilde{\bm{\mathcal{E}}}_h - \bm{\mathcal{E}}_h),P_h^n) + \tau (\bm\psi^n, q_r^n(\tilde{\bm{\mathcal{E}}}_h - \bm{\mathcal{E}}_h)) \\
&= \tau (\bm{\mathcal{E}}_h^n, q_r^n(\tilde{\bm{\mathcal{E}}}_h - \bm{\mathcal{E}}_h)) + \tau (\Pi_h \mathbf{u}^n - \mathbf{u}^n, q_r^n(\tilde{\bm{\mathcal{E}}}_h - \bm{\mathcal{E}}_h)) + \tau \mathcal{b}(q_r^n(\tilde{\bm{\mathcal{E}}}_h - \bm{\mathcal{E}}_h),P_h^n). \label{8.202}
\end{aligned}
\end{equation}
To bound the first term in \eqref{8.202}, using the CS inequality, Young's inequality and \eqref{8.196}, we obtain
\begin{equation}
\begin{aligned}
\tau |(\bm{\mathcal{E}}_h^n, q_r^n(\tilde{\bm{\mathcal{E}}}_h - \bm{\mathcal{E}}_h))| \leq \tau \norm{\bm{\mathcal{E}}_h^n} \norm{q_r^n(\tilde{\bm{\mathcal{E}}}_h - \bm{\mathcal{E}}_h)} &= \tau^2 \sum_{i=1}^n \beta(t_n - t_i) \norm{(\tilde{\bm{\mathcal{E}}}_h - \bm{\mathcal{E}}_h)^i} \norm{\bm{\mathcal{E}}_h^n} \\
& \leq \varepsilon \mu \norm{\bm{\mathcal{E}}_h^n}^2 + \frac{1}{\varepsilon \mu}\left(\tau^2 \sum_{i=1}^n \beta(t_n - t_i) \norm{(\tilde{\bm{\mathcal{E}}}_h - \bm{\mathcal{E}}_h)^i}\right)^2 \\
& \leq \varepsilon \mu \norm{\bm{\mathcal{E}}_h^n}^2 + \frac{1}{\varepsilon \mu}\left(\tau^3 \sum_{i=1}^n \beta(t_n - t_i) |v_h^i|_\text{dG}\right)^2. \label{eqn165}
\end{aligned}
\end{equation}
Now consider the second and third terms in \eqref{8.202}. With \eqref{eq4.37}, \eqref{8.189} and \eqref{8.196}, we have
\begin{equation}
\begin{aligned}
& \quad \tau |(\Pi_h \mathbf{u}^n - \mathbf{u}^n, q_r^n(\tilde{\bm{\mathcal{E}}}_h - \bm{\mathcal{E}}_h))| + \tau |\mathcal{b}(q_r^n(\tilde{\bm{\mathcal{E}}}_h - \bm{\mathcal{E}}_h),P_h^n)| \\
& \leq \tau \norm{\Pi_h \mathbf{u}^n - \mathbf{u}^n} \norm{q_r^n(\tilde{\bm{\mathcal{E}}}_h - \bm{\mathcal{E}}_h)} + C \tau \norm{q_r^n(\tilde{\bm{\mathcal{E}}}_h - \bm{\mathcal{E}}_h)} |P_h^n|_\text{dG} \\
& \leq C \tau \norm{q_r^n(\tilde{\bm{\mathcal{E}}}_h - \bm{\mathcal{E}}_h)} \left(h^{r+1} |\mathbf{u}^n|_{H^{r+1}(\mathcal{O})} + |P_h^n|_\text{dG}\right) \\
& \leq C \tau^2 \sum_{i=1}^n \beta(t_n - t_i) \norm{(\tilde{\bm{\mathcal{E}}}_h - \bm{\mathcal{E}}_h)^i} \left(h^{r+1} |\mathbf{u}^n|_{H^{r+1}(\mathcal{O})} + |P_h^n|_\text{dG}\right) \\
& \leq \frac{C}{\varepsilon \mu} \left(\tau^3 \sum_{i=1}^n \beta(t_n - t_i) |v_h^i|_\text{dG}\right)^2 + C h^{2r+2} + \varepsilon \mu \norm{\bm{\mathcal{E}}_h^n}^2. \label{eqn166}
\end{aligned}
\end{equation}
We now consider the term $P_2$ in \eqref{8.201}. Letting $\mathbf{w}_h = q_r^n(\bm{\mathcal{E}}_h)$ in \eqref{8.168}, and multiplying by $\tau$, we obtain
\begin{equation}
\begin{aligned}
\tau \mathcal{A}_d(q_r^n(\bm{\mathcal{E}}_h),\mathbf{V}_h^n) &= \tau \mathcal{b}(q_r^n(\bm{\mathcal{E}}_h), P_h^n) + \tau (\bm\psi^n, q_r^n(\bm{\mathcal{E}}_h)) \\
&= \tau \mathcal{b}(q_r^n(\bm{\mathcal{E}}_h), P_h^n) + \tau (\bm{\mathcal{E}}_h^n, q_r^n(\bm{\mathcal{E}}_h)) + \tau (\bm\psi^n - \bm{\mathcal{E}}_h^n, q_r^n(\bm{\mathcal{E}}_h)). \label{8.205}
\end{aligned}
\end{equation}
The first term in \eqref{8.205} is bounded using \eqref{eq4.37} and \eqref{8.189} as
\begin{equation}
\begin{aligned}
    \tau |\mathcal{b}(q_r^n(\bm{\mathcal{E}}_h), P_h^n)| & \leq C \tau^2 \sum_{i=1}^n \beta(t_n - t_i) \norm{\bm{\mathcal{E}}_h^i} ( \norm{\bm{\mathcal{E}}_h^n} + h^{r+1} |\mathbf{u}^n|_{H^{k+1}(\mathcal{O})}) \\
& \leq \frac{C}{\varepsilon \mu} \left(\tau^3 \sum_{i=1}^n \beta(t_n - t_i) |v_h^i|_\text{dG}\right)^2 + C h^{2r+2} + \varepsilon \mu\norm{\bm{\mathcal{E}}_h^n}^2. \label{eqn168}
\end{aligned}
\end{equation}
Using the CS inequality we bound the last two terms of \eqref{8.205} as
\begin{equation}
\begin{aligned}
&\tau |(\bm{\mathcal{E}}_h^n, q_r^n(\bm{\mathcal{E}}_h))|+\tau |(\bm\psi^n - \bm{\mathcal{E}}_h^n, q_r^n(\bm{\mathcal{E}}_h))| \\
&\leq \tau \|\bm{\mathcal{E}}_h^n\| \|q_r^n(\bm{\mathcal{E}}_h\|+\tau \norm{\Pi_h \mathbf{u}^n - \mathbf{u}^n} \norm{q_r^n(\bm{\mathcal{E}}_h)} \\
&\leq \varepsilon \mu \|\bm{\mathcal{E}}_h^n\|^2+\frac{C}{\varepsilon \mu}\left(\tau^2 \sum_{i=1}^n \beta(t_n-t_i) \|\bm{\mathcal{E}}_h^i\|\right)^2 + C h^{2r+2}. \label{eqn169}
\end{aligned}
\end{equation}
Next, to bound the term $P_3$, we choose $\mathbf{w}_h = q_r^n(\Pi_h \mathbf{u} - Q_h \mathbf{u})$ in \eqref{8.168}. Multiplying the resulting equality by $\tau$, we obtain
\begin{align*}
\tau \mathcal{A}_d(q_r^n(\Pi_h \mathbf{u} - \mathbf{u}), \mathbf{V}_h^n) &= \tau \mathcal{A}_d(q_r^n(\Pi_h \mathbf{u} - Q_h \mathbf{u}), \mathbf{V}_h^n) \\
&= \tau \mathcal{b}(q_r^n(\Pi_h \mathbf{u} - Q_h \mathbf{u}), P_h^n) + \tau (\bm\psi^n, q_r^n(\Pi_h \mathbf{u} - Q_h \mathbf{u})).
\end{align*}
Use of \eqref{eq4.37} to have
\begin{equation*}
|\mathcal{b}(q_r^n(\Pi_h \mathbf{u} - Q_h \mathbf{u}), P_h^n)| \leq C \norm{q_r^n(\Pi_h \mathbf{u} - Q_h \mathbf{u})} |P_h^n|_\text{dG}.
\end{equation*}
Hence, using the above bound, the CS inequality, Young's inequality, and the assumption on regularity (\textbf{A2}), we have
\begin{equation}
\begin{aligned}
\tau |\mathcal{A}_d(q_r^n(\Pi_h \mathbf{u} - \mathbf{u}), \mathbf{V}_h^n)| & \leq C \tau\norm{q_r^n(\Pi_h \mathbf{u} - Q_h\mathbf{u})}\left(|P_h^n|_\text{dG} + \norm{\bm\psi^n}\right) \\
& \leq C\left(1+\frac{1}{\varepsilon \mu}\right) \left(\tau^2 \sum_{i=1}^n \beta(t_n - t_i) |\mathbf{u}^i|_{H^{r+1}(\mathcal{O})}\right)^2 + C h^{2r+2} + \varepsilon \mu \norm{\bm{\mathcal{E}}_h^n}^2. \label{eqn170}
\end{aligned}
\end{equation}
We use the bounds \eqref{eqn162}-\eqref{eqn170} in \eqref{8.200} to obtain
\begin{equation}
\begin{aligned}
\frac{1}{2} \left( \mathcal{A}_d \left( \mathbf{V}_h^n, \mathbf{V}_h^n \right) - \mathcal{A}_d \left( \mathbf{V}_h^{n-1}, \mathbf{V}_h^{n-1} \right) + \mathcal{A}_d \left( \mathbf{V}_h^n - \mathbf{V}_h^{n-1}, \mathbf{V}_h^n - \mathbf{V}_h^{n-1} \right) \right)+ \tau \mu \| \bm{\mathcal{E}}_h^n \|^2 \\
\leq 10 \varepsilon \tau \mu \| \bm{\mathcal{E}}_h^n \|^2 + C \mu \left( 1 + \frac{1}{\varepsilon} \right) \tau^3 | v_h^n |_\text{dG}^2 - \tau \tilde{R}_C \left( \mathbf{V}_h^n \right) + C \left( \mu \left( 1 + \frac{1}{\varepsilon} \right) + 1 + \frac{1}{\varepsilon \mu} \right) \tau h^{2r+2} \\
+\tilde{C} \tau^3 \int_0^{t_n} e^{-2\eta(t_n-t)}\left(\|\mathbf{u}(t)\|_{H^1(\mathcal{O})}^2+h^2\|\mathbf{u}(t)\|_{H^2(\mathcal{O})}^2+\|\mathbf{u}_t(t)\|_{H^1(\mathcal{O})}^2+h^2\|\mathbf{u}_t(t)\|_{H^2(\mathcal{O})}^2\right)dt\\
+ C\tau \left(1+\frac{1}{\varepsilon \mu}\right) \left(\tau^2 \sum_{i=1}^n \beta(t_n - t_i) |\mathbf{u}^i|_{H^{r+1}(\mathcal{O})}\right)^2
+\frac{C}{\varepsilon \mu}\tau\left(\tau^2 \sum_{i=1}^n \beta(t_n-t_i) \|\bm{\mathcal{E}}_h^i\|\right)^2 \\+ C \tau\left(\tau^3 \sum_{i=1}^n \beta(t_n - t_i) |v_h^i|_\text{dG}\right)^2
+C \tau^2 \int_{t_{n-1}}^{t_n} \|\partial_{tt}\mathbf{u}\|^2 dt + C \tau \| \mathbf{V}_h^n \|_\text{dG}^2. \label{eqn171}
\end{aligned}
\end{equation}
We use \eqref{8.216}, bound \eqref{166}, apply the coercivity property \eqref{eq3.20}, and choose $\varepsilon = 1/34$ in \eqref{eqn171}. Then summing over $n=1$ to $m$, and employing the assumption on regularity ($\mathbf{A2}$), we have
\begin{equation}
\begin{aligned}
&\frac{1}{2} \mathcal{A}_d \left( \mathbf{V}_h^m, \mathbf{V}_h^m \right) - \frac{1}{2} \mathcal{A}_d \left( \mathbf{V}_h^0, \mathbf{V}_h^0 \right) + \frac{1}{4} \sum_{n=1}^{m} \| \mathbf{V}_h^n - \mathbf{V}_h^{n-1} \|_\text{dG}^2 + \frac{\tau \mu}{2} \sum_{n=1}^{m} \| \bm{\mathcal{E}}_h^n \|^2 \\
&\leq \tilde{C} h^{2r+2} + \tilde{C} \tau^2 + \tilde{C} \tau^3 \sum_{n=1}^{m}  | v_h^n |_\text{dG}^2 +\tilde{C}\tau \sum_{n=1}^m\left(\tau^2 \sum_{i=1}^n \beta(t_n-t_i) \|\bm{\mathcal{E}}_h^i\|\right)^2 \\
&+ C \tau \sum_{n=1}^m\left(\tau^3 \sum_{i=1}^n \beta(t_n - t_i) |v_h^i|_\text{dG}\right)^2 
+ \tilde{C} \tau \sum_{n=1}^m\left(\tau^2 \sum_{i=1}^n \beta(t_n - t_i) |\mathbf{u}^i|_{H^{r+1}(\mathcal{O})}\right)^2 \\
&+ \frac{\tau}{17} \sum_{n=1}^{m} \| \bm{\mathcal{E}}_h^n - \bm{\mathcal{E}}_h^{n-1} \|^2
+ \tilde{C} \tau \sum_{n=1}^{m} \left( h^2 \left( \| \bm{\mathcal{E}}_h^n \|^2 + \| \bm{\mathcal{E}}_h^{n-1} \|^2 \right) + \left( \|\bm{\mathcal{E}}_h^{n-1} \|^2 + h^2 \right) \| \tilde{\bm{\mathcal{E}}}_h^n \|_\text{dG}^2 \right)\\
&+\tilde{C} \tau \sum_{n=1}^{m} \| \mathbf{V}_h^n \|_\text{dG}^2 .
\label{8.213}
\end{aligned}
\end{equation} 
Using Holder's inequality, we find that
\begin{equation}
\begin{aligned}
\tau \sum_{n=1}^m \left(\tau^2 \sum_{i=1}^n \beta(t_n - t_i)\|\bm{\mathcal{E}}_h^i\|\right)^2
& \leq \gamma^2 \tau \sum_{n=1}^m\left(\tau^3 \sum_{i=1}^n e^{-2\eta(t_n - t_i)}\right) \left(\tau \sum_{i=1}^n\|\bm{\mathcal{E}}_h^i\|^2\right) \\
& \leq \frac{\gamma^2 e^{2 \eta \tau}}{\eta^3} \tau^2 \sum_{n=1}^m \sum_{i=1}^n \|\bm{\mathcal{E}}_h^i\|^2. \label{eqn173}
\end{aligned}
\end{equation}
Using the CS inequality and change of order of summation, we obtain
\begin{align*}
\tau \sum_{n=1}^m \left(\tau^2 \sum_{i=1}^n \beta(t_n - t_i)|\mathbf{u}^i|_{H^{r+1}(\mathcal{O})}\right)^2
\leq \frac{\gamma^2 e^{2 \eta \tau}}{\eta^4} \tau^3 \sum_{n=1}^m |\mathbf{u}^n|_{H^{r+1}(\mathcal{O})}^2,
\end{align*}
and analogously as above,
\begin{equation}
\begin{aligned}
\tau \sum_{n=1}^m \left(\tau^3 \sum_{i=1}^n \beta(t_n - t_i)|v_h^i|_\text{dG}\right)^2
\leq \frac{\gamma^2 e^{2 \eta \tau}}{\eta^4} \tau^3 \sum_{n=1}^m|v_h^n|_\text{dG}^2. \label{eqn174}
\end{aligned}
\end{equation}
Altogether bounds \eqref{eqn173}, \eqref{eqn174} and the regularity assumption (\textbf{A2}), leads \eqref{8.213} to
\begin{align*}
&\frac{1}{2} \mathcal{A}_d \left( \mathbf{V}_h^m, \mathbf{V}_h^m \right) - \frac{1}{2} \mathcal{A}_d \left( \mathbf{V}_h^0, \mathbf{V}_h^0 \right) + \frac{1}{4} \sum_{n=1}^{m} \| \mathbf{V}_h^n - \mathbf{V}_h^{n-1} \|_\text{dG}^2 + \frac{\tau \mu}{2} \sum_{n=1}^{m} \| \bm{\mathcal{E}}_h^n \|^2 \\
&\leq \tilde{C} h^{2r+2} + \tilde{C} \tau^2 + \tilde{C} \tau^3 \sum_{n=1}^{m}  | v_h^n |_\text{dG}^2 +\tilde{C}\tau^2 \sum_{n=1}^m \sum_{i=1}^n \|\bm{\mathcal{E}}_h^i\|^2 
+ \frac{\tau}{17} \sum_{n=1}^{m} \| \bm{\mathcal{E}}_h^n- \bm{\mathcal{E}}_h^{n-1} \|^2 \\
&\hspace{2em} + \tilde{C} \tau \sum_{n=1}^{m} \left( h^2 \left( \| \bm{\mathcal{E}}_h^n \|^2 + \| \bm{\mathcal{E}}_h^{n-1} \|^2 \right) + \left( \|\bm{\mathcal{E}}_h^{n-1} \|^2 + h^2 \right) \| \tilde{\bm{\mathcal{E}}}_h^n \|_\text{dG}^2 \right)\\
&\hspace{2em} + \tilde{C} \tau  \sum_{n=1}^{m} \| \mathbf{V}_h^n \|_\text{dG}^2 .
\end{align*}
By \eqref{8.168}-\eqref{8.169} and the definition of the $L^2$ projection, we have
\begin{equation*}
\mathcal{A}_d \left( \mathbf{V}_h^0, \mathbf{V}_h^0 \right) = \left( \bm\psi^0, \mathbf{V}_h^0 \right) = 0,
\end{equation*}
hence $\mathbf{V}_h^0 = 0$ as $\mathcal{A}_d$ is coercive.
Using Lemma 11, coercivity of $\mathcal{A}_d$ \eqref{eq3.20}, we use the discrete Gronwall's inequality to obtain
\begin{align*}
\| \mathbf{V}_h^m \|_\text{dG}^2 + \sum_{n=1}^m \| \mathbf{V}_h^n - \mathbf{V}_h^{n-1} \|_\text{dG}^2 + 2 \tau \mu \sum_{n=1}^m \| \bm{\mathcal{E}}_h^n \|^2 
\leq \tilde{C}_T h^{2r+2} + \tilde{C}_T \tau^2 \\
+ \tilde{C}_T \left( \tau + h^{2r} \right) \left( \tau + h^2 \right) 
+ \tilde{C}_T \tau  \sum_{n=1}^m \| \mathbf{V}_h^n \|_\text{dG}^2.
\end{align*}
Application of the discrete Gronwall's inequality yields the bound
\begin{equation*}
\tau \mu \sum_{n=1}^m \| \bm{\mathcal{E}}_h^n \|^2 
\leq \tilde{C}_T \left( \tau^2 + \tau h^2 + h^{2r+2} \right) .
\end{equation*}
Next, to establish a bound on $\tilde{\bm{\mathcal{E}}}_h^n$, we utilize \eqref{eq:6.110}, \eqref{eq4.37}, the above bound and Lemma 12
\begin{align*}
\tau \mu \sum_{n=1}^m \| \tilde{\bm{\mathcal{E}}}_h^n \|^2 
&\leq 2 \tau \mu \sum_{n=1}^m \| \tilde{\bm{\mathcal{E}}}_h^n - \bm{\mathcal{E}}_h^n \|^2
+ 2 \tau \mu \sum_{n=1}^m \| \bm{\mathcal{E}}_h^n \|^2 \\
&\leq C \mu \tau^3 \sum_{n=1}^m \| v_h^n \|_\text{dG}^2 
+ 2 \tau \mu \sum_{n=1}^m \| \bm{\mathcal{E}}_h^n \|^2 \\
&\leq \tilde{C}_T \left( \tau^2 + \tau h^2 + h^{2r+2} \right) .
\end{align*}
By applying the triangle inequality, we obtain the desired result.
\end{proof}
\begin{remark}
Observe that, for $k=1$, the result is optimal as $\tau h^2 \leq (\tau^2+h^4)/2$. If $k \geq 2$, invoking the reverse CFL condition $h^2 \leq \tau$ yields the optimal estimate.
\end{remark}
\begin{lemma} [\cite{masri2023improved}]
For any $\varepsilon > 0$, we have
\begin{equation}
\begin{aligned}
|\tilde{R}_C(\mathbf{V}_h^n)| \leq 7 \varepsilon \mu \|\bm{\mathcal{E}}_h^n\|^2 + 2 \varepsilon \|\bm{\mathcal{E}}_h^{n-1} - \bm{\mathcal{E}}_h^n\|^2 
+ C \left( \frac{h^2}{\varepsilon \mu} + 1 \right) \left( \tau^2 |v_h^n|_{\text{dG}}^2 + h^{2r+2} \right) \\
+ C \left( \frac{1}{\varepsilon \mu} + 1 \right) \left( h^2 (\|\bm{\mathcal{E}}_h^n\|^2 + \|\bm{\mathcal{E}}_h^{n-1}\|^2) + (\|\bm{\mathcal{E}}_h^{n-1}\|^2 + h^2) \|\tilde{\bm{\mathcal{E}}}_h^n\|_{\text{dG}}^2 \right) \\
+ C \tau \int_{t_{n-1}}^{t_n} \|\partial_t \mathbf{u}\|^2 \, dt + C \left( \frac{1}{\varepsilon \mu} + 1 \right) \|\mathbf{V}_h^n\|_{\text{dG}}^2. \label{8.216}
\end{aligned}
\end{equation}
\end{lemma}

\section{Estimates of the discrete time derivative for the velocity}
For $\bm{\theta} \in \mathbf{M}$, we define the discrete time derivative, $\forall n \geq 0$ as follows:
\begin{equation*}
    \delta_\tau \bm{\theta}^{n+1} := \frac{\bm{\theta}^{n+1} - \bm{\theta}^n}{\tau}.
\end{equation*}
For the discrete time derivative of a scalar function in $\mathcal{C}([0,T]; M)$, we use the same notation as well.
\begin{lemma}
Let the assumptions (\textbf{A3}) and (\textbf{A4}) hold. Then, under the same assumptions of Theorem 3, we have the following stability bound. For $1 \leq m \leq N - 1$,
\begin{equation*}
    \|\delta_\tau \bm{\mathcal{E}}_h^{m+1}\|^2 
    + \sum_{n=1}^m \|\delta_\tau \bm{\mathcal{E}}_h^{n+1} - \delta_\tau \bm{\mathcal{E}}_h^n\|^2 
    + \tau^2 \sum_{n=1}^m |\delta_\tau v_h^{n+1}|^2_{\text{dG}}
    + \mu \tau \sum_{n=1}^m \|\delta_\tau \tilde{\bm{\mathcal{E}}}_h^{n+1}\|^2_{\text{dG}} 
    \leq \widehat{C}_T,
\end{equation*}
where $\widehat{C}_T = \tilde{C}_T + (\tilde{C}_T)^{\varrho+4} c_2^\varrho$ is independent of $h$ and $\tau$, where $\varrho$ is an integer such that $\varrho \geq 1/\upsilon$.
\end{lemma}
\begin{proof}
Subtracting \eqref{7.105} at time $t_n$ from \eqref{7.105} at time $t_{n+1}$, we have $ \forall \mathbf{w}_h \in \mathbf{M}_h$,
\begin{equation}
\begin{aligned}
    &(\delta_\tau \tilde{\bm{\mathcal{E}}}_h^{n+1} - \delta_\tau \bm{\mathcal{E}}_h^n, \mathbf{w}_h) 
    + \tau \mu \mathcal{A}_d(\delta_\tau \tilde{\bm{\mathcal{E}}}_h^{n+1}, \mathbf{w}_h)
    + \tau \mathcal{A}_d\left( q_r^{n+1}(\delta_\tau \tilde{\bm{\mathcal{E}}}_h), \mathbf{w}_h\right)
    + \tau \hat{R}_S(\mathbf{w}_h)\\
    &= \tau \mathcal{b}(\mathbf{w}_h, \delta_\tau p_h^n - \delta_\tau p^{n+1}) 
    - \tau \mu \mathcal{A}_d(\delta_\tau \Pi_h \mathbf{u}^{n+1} - \delta_\tau \mathbf{u}^{n+1}, \mathbf{w}_h)
 + \mathcal{L}_1(\mathbf{w}_h) + \mathcal{L}_2(\mathbf{w}_h) + \hat{R}_t(\mathbf{w}_h). \label{8.226}
    \end{aligned}
\end{equation}
Here, $\hat{R}_S(\mathbf{w}_h)$, $\mathcal{L}_1(\mathbf{w}_h)$, $\mathcal{L}_2(\mathbf{w}_h)$, and $\hat{R}_t(\mathbf{w}_h)$ are given by, for $\mathbf{w}_h \in \mathbf{M}_h$,
\begin{align*}
    \hat{R}_S(\mathbf{w}_h) & = \mathcal{A}_d\left(q_r^{n+1}(\mathbf{u}), \mathbf{w}_h\right) -\int_0^{t_{n+1}} \beta(t_{n+1}-s) \mathcal{A}_d(\mathbf{u},\mathbf{w}_h) ds\\
    &\hspace{4em} -\mathcal{A}_d\left(q_r^n(\mathbf{u}), \mathbf{w}_h\right) +\int_0^{t_n} \beta(t_n-s) \mathcal{A}_d(\mathbf{u},\mathbf{w}_h) ds, \\
    \mathcal{L}_1(\mathbf{w}_h) &= \mathcal{A}_d(\mathbf{u}^{n+1}; \mathbf{u}^{n+1}, \mathbf{u}^{n+1}, \mathbf{w}_h) 
    - \mathcal{A}_d(\mathbf{u}_h^n; \mathbf{u}_h^n, \widetilde{\mathbf{u}}_h^{n+1}, \mathbf{w}_h), \\
    \mathcal{L}_2(\mathbf{w}_h) &= \mathcal{A}_d(\mathbf{u}_h^{n-1}; \mathbf{u}_h^{n-1}, \widetilde{\mathbf{u}}_h^n, \mathbf{w}_h) 
    - \mathcal{A}_d(\mathbf{u}^n; \mathbf{u}^n, \mathbf{u}^n, \mathbf{w}_h), \\
    \hat{R}_t(\mathbf{w}_h) &= ((\partial_t \mathbf{u})^{n+1} - (\partial_t \mathbf{u})^n, \mathbf{w}_h) 
    - \frac{1}{\tau} (\Pi_h \mathbf{u}^{n+1} - 2 \Pi_h \mathbf{u}^n + \Pi_h \mathbf{u}^{n-1}, \mathbf{w}_h).
\end{align*}
Choosing $\mathbf{w}_h = \delta_\tau \tilde{\bm{\mathcal{E}}}_h^{n+1}$ in \eqref{8.226} and with an aid of the coercivity of $\mathcal{A}_d$ \eqref{eq3.20}, we obtain
\begin{equation}
\begin{aligned}
    &\frac{1}{2} \left(\|\delta_\tau \tilde{\bm{\mathcal{E}}}_h^{n+1}\|^2 - \|\delta_\tau \bm{\mathcal{E}}_h^n\|^2 + \|\delta_\tau \tilde{\bm{\mathcal{E}}}_h^{n+1} - \delta_\tau \bm{\mathcal{E}}_h^n\|^2\right) 
    + \frac{1}{2} \tau \mu \|\delta_\tau \tilde{\bm{\mathcal{E}}}_h^{n+1}\|^2_{\text{dG}} \\&+ \tau \mathcal{A}_d\left( q_r^{n+1}(\delta_\tau\tilde{\bm{\mathcal{E}}}_h), \delta_\tau \tilde{\bm{\mathcal{E}}}_h^{n+1}\right) 
    + \tau \hat{R}_S(\delta_\tau \tilde{\bm{\mathcal{E}}}_h^{n+1})
    \leq \tau \mathcal{b}(\delta_\tau \tilde{\bm{\mathcal{E}}}_h^{n+1}, \delta_\tau p_h^n - \delta_\tau p^{n+1}) \\
    &- \tau \mu \mathcal{A}_d(\delta_\tau \Pi_h \mathbf{u}^{n+1} - \delta_\tau \mathbf{u}^{n+1}, \delta_\tau \tilde{\bm{\mathcal{E}}}_h^{n+1}) 
    + \mathcal{L}_1(\delta_\tau \tilde{\bm{\mathcal{E}}}_h^{n+1}) 
    + \mathcal{L}_2(\delta_\tau \tilde{\bm{\mathcal{E}}}_h^{n+1}) 
    + \hat{R}_t(\delta_\tau \tilde{\bm{\mathcal{E}}}_h^{n+1}). \label{8.231}
\end{aligned}
\end{equation}
From \eqref{eq:6.110}, we obtain $\forall n \geq 1$,
\begin{equation}
    (\delta_\tau \bm{\mathcal{E}}_h^{n+1} - \delta_\tau \tilde{\bm{\mathcal{E}}}_h^{n+1}, \mathbf{w}_h) 
    = \tau \mathcal{b}(\mathbf{w}_h, \delta_\tau v_h^{n+1}), \quad \forall \mathbf{w}_h \in \mathbf{M}_h.
    \label{8.232}
\end{equation}
Additionally, from \eqref{eq:6.111} and \eqref{6.112}, we have $\forall g_h \in P_h$ and $\forall n \geq 1$,
\begin{equation}
\begin{aligned}
\mathcal{b}(\delta_\tau \bm{\mathcal{E}}_h^{n+1}, g_h) &= \mathcal{b}(\delta_\tau \tilde{\bm{\mathcal{E}}}_h^{n+1}, g_h) + \tau \mathcal{A}_{\text{sip}}(\delta_\tau v_h^{n+1}, g_h) - \tau \sum_{F \in \mathcal{F}_h^i}  \frac{\tilde{\sigma}}{h_F} \int_F [\delta_\tau v_h^{n+1}][g_h] \\
&\hspace{3em} + \tau (\mathbf{G}_h([\delta_\tau v_h^{n+1}]), \mathbf{G}_h([g_h])) \\
&= - \tau \sum_{F \in \mathcal{F}_h^i}  \frac{\tilde{\sigma}}{h_F} \int_F [\delta_\tau v_h^{n+1}][g_h] + \tau (\mathbf{G}_h([\delta_\tau v_h^{n+1}]), \mathbf{G}_h([g_h])). \label{8.233}
\end{aligned}   
\end{equation}
From \eqref{8.232} and \eqref{8.233}, we observe that
\begin{align*}
\|\delta_\tau \bm{\mathcal{E}}_h^{n+1} - \delta_\tau \tilde{\bm{\mathcal{E}}}_h^{n+1}\|^2 &= \tau \mathcal{b}(\delta_\tau \bm{\mathcal{E}}_h^{n+1} - \delta_\tau \tilde{\bm{\mathcal{E}}}_h^{n+1}, \delta_\tau v_h^{n+1}) \\
&= \tau^2 \mathcal{A}_{\text{sip}}(\delta_\tau v_h^{n+1}, \delta_\tau v_h^{n+1}) 
- \tau^2 \sum_{F \in \mathcal{F}_h^i}  \frac{\tilde{\sigma}}{h_F} \|[\delta_\tau v_h^{n+1}]\|_{L^2(F)}^2 + \tau^2 \|\mathbf{G}_h([\delta_\tau v_h^{n+1}])\|^2.
\end{align*}
Thus, by letting $\mathbf{w}_h = \delta_\tau \bm{\mathcal{E}}_h^{n+1}$ in \eqref{8.232} and using \eqref{8.233}, we obtain
\begin{equation}
\begin{aligned}
\frac{1}{2} \left( \|\delta_\tau \bm{\mathcal{E}}_h^{n+1}\|^2 - \|\delta_\tau \tilde{\bm{\mathcal{E}}}_h^{n+1}\|^2\right) + \tau^2 \mathcal{A}_{\text{sip}}(\delta_\tau v_h^{n+1}, \delta_\tau v_h^{n+1}) \\
+ \frac{\tau^2}{2} \sum_{F \in \mathcal{F}_h^i}  \frac{\tilde{\sigma}}{h_F} \|[\delta_\tau v_h^{n+1}]\|_{L^2(F)}^2 = \frac{\tau^2}{2} \|\mathbf{G}_h([\delta_\tau v_h^{n+1}])\|^2. \label{8.235}
\end{aligned}
\end{equation}
Additionally, from \eqref{8.232} and \eqref{eq4.32}, we note that
\[
\|\delta_\tau \tilde{\bm{\mathcal{E}}}_h^{n+1} - \delta_\tau \bm{\mathcal{E}}_h^n\|^2 = \| \delta_\tau \bm{\mathcal{E}}_h^{n+1} - \delta_\tau \bm{\mathcal{E}}_h^n + \tau \nabla_h \delta_\tau v_h^{n+1} - \tau \mathbf{G}_h([\delta_\tau v_h^{n+1}]) \|^2.
\]
Using \eqref{eq4.32} with $\mathbf{w}_h = \delta_\tau \bm{\mathcal{E}}_h^{n+1} - \delta_\tau \bm{\mathcal{E}}_h^n$ and $g_h = \delta_\tau v_h^{n+1}$ yields
\begin{equation}
\begin{aligned}
\|\delta_\tau \tilde{\bm{\mathcal{E}}}_h^{n+1} - \delta_\tau \bm{\mathcal{E}}_h^n\|^2 =  \tau^2 &\|\nabla_h \delta_\tau v_h^{n+1}\|^2 + \tau^2 \|\mathbf{G}_h([\delta_\tau v_h^{n+1}])\|^2 + \|\delta_\tau \bm{\mathcal{E}}_h^{n+1} - \delta_\tau \bm{\mathcal{E}}_h^n\|^2 \\
&- 2\tau^2 (\nabla_h \delta_\tau v_h^{n+1}, \mathbf{G}_h([\delta_\tau v_h^{n+1}])) - 2\tau \mathcal{b}(\delta_\tau \bm{\mathcal{E}}_h^{n+1} - \delta_\tau \bm{\mathcal{E}}_h^n, \delta_\tau v_h^{n+1}). \label{8.236}
\end{aligned}
\end{equation}
We utilize \eqref{8.233} and \eqref{6.112} to obtain
\begin{equation}
\begin{aligned}
-2\mathcal{b}(\delta_\tau \bm{\mathcal{E}}_h^{n+1} - \delta_\tau \bm{\mathcal{E}}_h^n, \delta_\tau v_h^{n+1}) = \tau (\tilde{H}_1^n - \tilde{H}_2^n) - \tau \|\mathbf{G}_h([\delta_\tau v_h^{n+1} - \delta_\tau v_h^n])\|^2 \\
+ \tau \sum_{F \in \mathcal{F}_h^i}  \frac{\tilde{\sigma}}{h_F} \|[\delta_\tau v_h^{n+1} - \delta_\tau v_h^n]\|_{L^2(F)}^2 - \frac{2}{\tau} \delta_{n,1} \mathcal{b}(\bm{\mathcal{E}}_h^0, \delta_\tau v_h^2), \label{8.237}
\end{aligned}
\end{equation}
where $\forall n \geq 1$,
\begin{align*}
\tilde{H}_1^n &= \sum_{F \in \mathcal{F}_h^i}  \frac{\tilde{\sigma}}{h_F} \left( \|[\delta_\tau v_h^{n+1}]\|_{L^2(F)}^2 - \|[\delta_\tau v_h^n]\|_{L^2(F)}^2 \right), \\
\tilde{H}_2^n &= \|\mathbf{G}_h([\delta_\tau v_h^{n+1}])\|^2 - \|\mathbf{G}_h([\delta_\tau v_h^n])\|^2.
\end{align*}
Using \eqref{eq4.30} and assuming $\tilde{\sigma} \geq 4\tilde{B}_r^2$, we have
\begin{equation}
\begin{aligned}
\sum_{F \in \mathcal{F}_h^i}  \frac{\tilde{\sigma}}{h_F} \| [\delta_\tau v_h^{n+1} - \delta_\tau v_h^n] \|_{L^2(F)}^2 &- \| \mathbf{G}_h([\delta_\tau v_h^{n+1} - \delta_\tau v_h^n]) \|^2 \\
&\geq \frac{1}{2} \sum_{F \in \mathcal{F}_h^i}  \frac{\tilde{\sigma}}{h_F} \| [\delta_\tau v_h^{n+1} - \delta_\tau v_h^n] \|_{L^2(F)}^2,
\end{aligned}
\end{equation}
\begin{align}
|(\nabla_h \delta_\tau v_h^{n+1}, \mathbf{G}_h([\delta_\tau v_h^{n+1}])) | \leq \frac{1}{4} \| \nabla_h \delta_\tau v_h^{n+1} \|^2 + \frac{1}{4} \sum_{F \in \mathcal{F}_h^i}  \frac{\tilde{\sigma}}{h_F} \| [\delta_\tau v_h^{n+1}] \|_{L^2(F)}^2. \label{eqn:190}
\end{align}
With the expressions\eqref{8.235} -\eqref{eqn:190}, \eqref{8.231} becomes
\begin{equation}
\begin{aligned}
&\frac{1}{2} (\| \delta_\tau \bm{\mathcal{E}}_h^{n+1} \|^2 + \| \delta_\tau \bm{\mathcal{E}}_h^n \|^2 + \| \delta_\tau \bm{\mathcal{E}}_h^{n+1} - \delta_\tau \bm{\mathcal{E}}_h^n \|^2) + \frac{\tau^2}{4} \| \delta_\tau v_h^{n+1} \|_\text{dG}^2 + \frac{1}{2}\tau \mu \norm{\delta_\tau \tilde{\bm{\mathcal{E}}}_h^{n+1}}_\text{dG}^2\\
&+ \tau^2 \mathcal{A}_{\text{sip}}(\delta_\tau v_h^{n+1}, \delta_\tau v_h^{n+1}) + \frac{\tau^2}{4} \sum_{F \in \mathcal{F}_h^i}  \frac{\tilde{\sigma}}{h_F} \| [\delta_\tau v_h^{n+1} - \delta_\tau v_h^n] \|_{L^2(F)}^2 + \frac{\tau^2}{2} (\tilde{H}_1^n - \tilde{H}_2^n) + \tau \hat{R}_S(\delta_\tau \tilde{\bm{\mathcal{E}}}_h^{n+1})\\
&+ \tau \mathcal{A}_d\left( q_r^{n+1}(\delta_\tau\tilde{\bm{\mathcal{E}}}_h), \delta_\tau \tilde{\bm{\mathcal{E}}}_h^{n+1}\right) \leq \tau \mathcal{b}(\delta_\tau \tilde{\bm{\mathcal{E}}}_h^{n+1}, \delta_\tau p_h^n - \delta_\tau p^{n+1}) - \tau \mu \mathcal{A}_d(\delta_\tau \Pi_h \mathbf{u}^{n+1} - \delta_\tau \mathbf{u}^{n+1}, \delta_\tau \tilde{\bm{\mathcal{E}}}_h^{n+1}) \\
&+ \mathcal{L}_1(\delta_\tau \tilde{\bm{\mathcal{E}}}_h^{n+1}) + \mathcal{L}_2(\delta_\tau \tilde{\bm{\mathcal{E}}}_h^{n+1}) + \hat{R}_t(\delta_\tau \tilde{\bm{\mathcal{E}}}_h^{n+1}) + \delta_{n,1} \mathcal{b}(\bm{\mathcal{E}}_h^0, \delta_\tau v_h^2). \label{8.238}
\end{aligned}
\end{equation}
Following a technique analogous to the one used in \cite{girault2005splitting,masri2022discontinuous}, we get
\begin{equation}
\begin{aligned}
&|\mathcal{A}_d(\delta_\tau \Pi_h \mathbf{u}^{n+1} - \delta_\tau \mathbf{u}^{n+1}, \delta_\tau \tilde{\bm{\mathcal{E}}}_h^{n+1})| = |\mathcal{A}_d(\Pi_h \delta_\tau \mathbf{u}^{n+1} - \delta_\tau \mathbf{u}^{n+1}, \delta_\tau \tilde{\bm{\mathcal{E}}}_h^{n+1})| \\
&\leq C h |\delta_\tau \mathbf{u}^{n+1}|_{H^2(\mathcal{O})} \| \delta_\tau \tilde{\bm{\mathcal{E}}}_h^{n+1} \|_\text{dG} \leq \varepsilon \| \delta_\tau \tilde{\bm{\mathcal{E}}}_h^{n+1} \|_\text{dG}^2 + \frac{C}{\varepsilon} h^2 \tau^{-1} \int_{t_n}^{t_{n+1}} |\partial_t \mathbf{u}|_{H^2(\mathcal{O})}^2 dt. \label{eqn183}
\end{aligned}
\end{equation}
Let us consider the first two terms on the RHS of \eqref{8.238}. We split them as follows:
\begin{align*}
\mathcal{b}(\delta_\tau \tilde{\bm{\mathcal{E}}}_h^{n+1}, \delta_\tau p_h^n - \delta_\tau p^{n+1}) = \mathcal{b}&(\delta_\tau \tilde{\bm{\mathcal{E}}}_h^{n+1}, \delta_\tau p_h^n)+ \mathcal{b}(\delta_\tau \tilde{\bm{\mathcal{E}}}_h^{n+1}, \pi_h(\delta_\tau p^{n+1}) - \delta_\tau p^{n+1}) \notag \\ 
& - \mathcal{b}(\delta_\tau \tilde{\bm{\mathcal{E}}}_h^{n+1}, \pi_h(\delta_\tau p^{n+1})).
\end{align*}
Since, $\pi_h(\delta_\tau p^{n+1})$ belongs to $P_h^0$, hence by \eqref{eq:5.41}, \eqref{eq6.102} and \eqref{eq:6.110}, we have
\begin{equation}
\begin{aligned}
|\mathcal{b}(\delta_\tau \tilde{\bm{\mathcal{E}}}_h^{n+1}, \pi_h(\delta_\tau p^{n+1}))| &= \tau |\mathcal{A}_{\text{sip}}(\delta_\tau v_h^{n+1}, \pi_h(\delta_\tau p^{n+1}))| \\
&\leq C \tau | \delta_\tau p^{n+1} |_{H^1(\mathcal{O})}|\delta_\tau v_h^{n+1}|_\text{dG} .
\end{aligned}
\end{equation}
Also, by \eqref{eq4.37} and \eqref{eq6.101}, we obtain
\begin{equation}
\begin{aligned}
|\mathcal{b}(\delta_\tau \tilde{\bm{\mathcal{E}}}_h^{n+1}, \pi_h(\delta_\tau p^{n+1}) - \delta_\tau p^{n+1}| 
&\leq C h  \|\delta_\tau \tilde{\bm{\mathcal{E}}}_h^{n+1}\|_\text{dG} |\delta_\tau p^{n+1}|_{H^1(\mathcal{O})} \\
&\leq \varepsilon \mu |\delta_\tau \tilde{\bm{\mathcal{E}}}_h^{n+1}|_\text{dG}^2 + \frac{C}{\varepsilon \mu} h^2 \tau^{-1} \int_{t_n}^{t_{n+1}} |\partial_t p|^2_{H^1(\mathcal{O})} dt.
\end{aligned}
\end{equation}
To tackle the term $\mathcal{b}(\delta_\tau \tilde{\bm{\mathcal{E}}}_h^{n+1}, \delta_\tau p_h^n)$, we modify the auxiliary functions, and define $\forall n \geq 1$,
\begin{align*}
\hat{S}_h^n &= \delta \mu \sum_{i=1}^n (\nabla_h \cdot \delta_\tau \tilde{\bm{\mathcal{E}}}_h^i - R_h([\delta_\tau \tilde{\bm{\mathcal{E}}}_h^i])) + \delta \sum_{i=1}^n \sum_{j=1}^i \beta_{i-j} (\nabla_h \cdot \delta_\tau \tilde{\bm{\mathcal{E}}}_h^j - R_h([\delta_\tau \tilde{\bm{\mathcal{E}}}_h^j])), \\
\quad \hat{\xi}_h^n &= \delta_\tau p_h^n + \hat{S}_h^n.
\end{align*}
From \eqref{eq:6.111} and the above definitions, we note that $\hat{\xi}_h^{n+1 } - \hat{\xi}_h^n = \delta_\tau v_h^{n+1}$. Based on this expression, we have
\[
\mathcal{b}(\delta_\tau \tilde{\bm{\mathcal{E}}}_h^{n+1}, \delta_\tau p_h^n) = \mathcal{b}(\delta_\tau \tilde{\bm{\mathcal{E}}}_h^{n+1}, \hat{\xi}_h^n) - \mathcal{b}(\delta_\tau \tilde{\bm{\mathcal{E}}}_h^{n+1}, \hat{S}_h^n).
\]
Since $\hat{S}_h^n \in P_h^0$, it can be noted that $\hat{\xi}_h^n \in P_h^0$ $\forall n \geq 1$. Therefore, we apply \eqref{eq6.108} to conclude that
\begin{equation}
\begin{aligned}
\mathcal{b}(\delta_\tau \tilde{\bm{\mathcal{E}}}_h^{n+1}, \hat{\xi}_h^n) &= -\tau \mathcal{A}_{\text{sip}}(\delta_\tau v_h^{n+1}, \hat{\xi}_h^n) = -\tau \mathcal{A}_{\text{sip}}(\hat{\xi}_h^{n+1} - \hat{\xi}_h^n, \hat{\xi}_h^n) \\
&= -\frac{\tau}{2} \left( \mathcal{A}_{\text{sip}}(\hat{\xi}_h^{n+1}, \hat{\xi}_h^{n+1}) - \mathcal{A}_{\text{sip}}(\hat{\xi}_h^n, \hat{\xi}_h^n) - \mathcal{A}_{\text{sip}}(\delta_\tau v_h^{n+1}, \delta_\tau v_h^{n+1}) \right).
\end{aligned}
\end{equation}
With \eqref{eq4.31}, we have
\begin{equation}
\begin{aligned}
\mathcal{b}(\delta_\tau \tilde{\bm{\mathcal{E}}}_h^{n+1}, \hat{S}_h^n) = \frac{1}{2 \delta \mu} &\left( \|\hat{S}_h^{n+1}\|^2 - \|\hat{S}_h^n\|^2 - \|\hat{S}_h^{n+1} - \hat{S}_h^n\|^2 \right) \\
&- \frac{1}{ \mu} \left(\nabla_h \cdot q_r^{n+1}(\tilde{\bm{\mathcal{E}}}_h)-R_h([q_r^{n+1}(\tilde{\bm{\mathcal{E}}}_h)]),\hat{S}_h^n\right).
\end{aligned}
\end{equation}
With the assumptions that $\delta \leq \min\left(\frac{1}{16d}, \frac{ \mu^2}{32 \gamma^2 d}\right)$ and $\sigma \geq \frac{M_{r-1}^2}{d}$, we use \eqref{eq4.29} to obtain
\begin{equation}
\begin{aligned}
\frac{1}{2 \delta \mu} \|\hat{S}_h^{n+1} - \hat{S}_h^n\|^2 \leq \frac{\mu}{8}&\left( \|\nabla_h \delta_\tau \tilde{\bm{\mathcal{E}}}_h^{n+1}\|^2 + \sum_{F \in \mathcal{F}_h}\sigma h_F^{-1} \|[\delta_\tau \tilde{\bm{\mathcal{E}}}_h^{n+1}]\|_{L^2(F)}^2\right)\\
&+\frac{\mu}{16}\tau^2 \sum_{j=1}^n \left( \|\nabla_h \delta_\tau \tilde{\bm{\mathcal{E}}}_h^{j+1}\|^2 + \sum_{F \in \mathcal{F}_h}\sigma h_F^{-1} \|[\delta_\tau \tilde{\bm{\mathcal{E}}}_h^{j+1}]\|_{L^2(F)}^2\right).
\end{aligned}
\end{equation}
Also, using the assumptions that $\delta \leq \frac{ \mu^2}{32 \gamma^2 d}$ and $\sigma \geq \frac{M_{r-1}^2}{d}$, we obtain
\begin{equation}
\begin{aligned}
&\left |\frac{\tau}{\mu}\left(\nabla_h \cdot q_r^{n+1}(\tilde{\bm{\mathcal{E}}}_h)-R_h([q_r^{n+1}(\tilde{\bm{\mathcal{E}}}_h)]),\hat{S}_h^n\right) \right | \\
&\leq \frac{ \mu }{16} \tau^2 \sum_{j=1}^n\left(\norm{\nabla_h \tilde{\bm{\mathcal{E}}}^{j+1}_h}^2+ \sum_{F \in \mathcal{F}_h}\sigma h_F^{-1} \|[\tilde{\bm{\mathcal{E}}}^{j+1}_h]\|^2_{L^2(F)}\right) + \frac{\tau^2}{2 \delta \mu} \norm{\hat{S}_h^n}^2. \label{eqn189}
\end{aligned}
\end{equation}
Using the above estimates \eqref{eqn183}-\eqref{eqn189}, inequality \eqref{8.238} yields
\begin{equation}
\begin{aligned}
&\frac{1}{2} \left( \|\delta_\tau \bm{\mathcal{E}}_h^{n+1}\|^2 - \|\delta_\tau \bm{\mathcal{E}}_h^n\|^2 + \|\delta_\tau \bm{\mathcal{E}}_h^{n+1} - \delta_\tau \bm{\mathcal{E}}_h^n\|^2 \right) + \frac{\tau \mu}{4} \|\delta_\tau \tilde{\bm{\mathcal{E}}}_h^{n+1}\|_\text{dG}^2 + \frac{3\tau^2}{16} |\delta_\tau v_h^{n+1}|_\text{dG}^2\\
&+ \frac{\tau^2}{2} \left(\mathcal{A}_{\text{sip}}(\hat{\xi}_h^{n+1}, \hat{\xi}_h^{n+1}) - \mathcal{A}_{\text{sip}}(\hat{\xi}_h^n, \hat{\xi}_h^n)\right)  + \frac{\tau^2}{4} \sum_{F \in \mathcal{F}_h^i}  \frac{\tilde{\sigma}}{h_F} \| [\delta_\tau v_h^{n+1} - \delta_\tau v_h^n] \|_{L^2(F)}^2 \\
 &+ \frac{\tau^2}{2} (\tilde{H}_1^n - \tilde{H}_2^n)+ \frac{\tau}{2 \delta \mu} (\|\hat{S}_h^{n+1}\|^2 - \|\hat{S}_h^n\|^2) + \tau \hat{R}_S(\delta_\tau \tilde{\bm{\mathcal{E}}}_h^{n+1}) + \tau \mathcal{A}_d\left( q_r^{n+1}(\delta_\tau\tilde{\bm{\mathcal{E}}}_h), \delta_\tau \tilde{\bm{\mathcal{E}}}_h^{n+1}\right)\\
&\leq \left(2 \varepsilon \tau \mu +\frac{\tau \mu}{8}\right)\|\delta_\tau \tilde{\bm{\mathcal{E}}}_h^{n+1}\|_\text{dG}^2 + C \frac{\mu}{\varepsilon} h^2 \int_{t_n}^{t_{n+1}} |\partial_t \mathbf{u}|_{H^2(\mathcal{O})}^2 dt + C \left(\tau + \frac{h^2}{\varepsilon \mu}\right) \int_{t_n}^{t_{n+1}} |\partial_t p|_{H^1(\mathcal{O})}^2 dt\\
&\hspace{2em}+ \frac{\tau^2}{2 \delta \mu} \norm{\hat{S}_h^n}^2 
+\frac{\mu}{8}\tau^2 \sum_{i=1}^n \|\delta_\tau \tilde{\bm{\mathcal{E}}}_h^{i+1}\|_\text{dG}^2 + \mathcal{L}_1(\delta_\tau \tilde{\bm{\mathcal{E}}}_h^{n+1}) + \mathcal{L}_2(\delta_\tau \tilde{\bm{\mathcal{E}}}_h^{n+1}) + \hat{R}_t(\delta_\tau \tilde{\bm{\mathcal{E}}}_h^{n+1}) \\&\hspace{2em}+ \delta_{n,1} \mathcal{b}(\bm{\mathcal{E}}_h^0, \delta_\tau v_h^2). \label{8.243}
\end{aligned}
\end{equation}
Following the similar approach used in \cite{masri2023improved}, we can bound the nonlinear terms as
\begin{equation}
\begin{aligned}
|\mathcal{L}_1(\delta_\tau \tilde{\bm{\mathcal{E}}}_h^{n+1})| + |\mathcal{L}_2(\delta_\tau \tilde{\bm{\mathcal{E}}}_h^{n+1})| \leq 4 \varepsilon \tau &\mu \|\delta_\tau \tilde{\bm{\mathcal{E}}}_h^{n+1}\|_{\text{dG}}^2 + \frac{C}{\varepsilon \mu} \tau \|\delta_\tau \bm{\mathcal{E}}_h^n\|_{L^3(\mathcal{O})}^2 \|\tilde{\bm{\mathcal{E}}}_h^{n+1}\|_{\text{dG}}^2 \\
& + \frac{C}{\varepsilon \mu} \tau^{-1}\|\bm{\mathcal{E}}_h^n - \bm{\mathcal{E}}_h^{n-1}\|^2 + \frac{C}{\varepsilon \mu} \|\tilde{\bm{\mathcal{E}}}_h^{n+1}\|_{\text{dG}}^2\int_{t_{n-1}}^{t_n} |\partial_t \mathbf{u}|_{H^1(\mathcal{O})}^2 dt \\
& + \frac{C}{\varepsilon \mu} \left(1 + \mu^{-1} + \tau^2 + h^{2r+2} + \|\bm{\mathcal{E}}_h^n\|^2\right) \int_{t_{n-1}}^{t_{n+1}} |\partial_t \mathbf{u}|_{H^2(\mathcal{O})}^2 dt. \label{eqn191}
\end{aligned}
\end{equation}
We now estimate $\hat{R}_t(\delta_\tau \tilde{\bm{\mathcal{E}}}_h^{n+1})$. With the approximation properties \eqref{eq:6.96}–\eqref{eq:6.97}, and \eqref{8.223}, we use a Taylor expansions to obtain $\forall n \geq 1$,
\begin{align}
\hat{R}_t(\delta_\tau \tilde{\bm{\mathcal{E}}}_h^{n+1}) &\leq \varepsilon \tau \mu \|\delta_\tau \tilde{\bm{\mathcal{E}}}_h^{n+1}\|_{\text{dG}}^2 + \frac{C}{\varepsilon \mu} \int_{t_{n-1}}^{t_{n+1}} |\partial_t \mathbf{u}|_{H^2(\mathcal{O})}^2 dt + \frac{C}{\varepsilon \mu} \int_{t_{n-1}}^{t_{n+1}} \|\partial_{tt} \mathbf{u}\|^2 dt . 
\end{align}
Next, using the continuity of $\mathcal{A}_d$ \eqref{eq4.23}, approximation property \eqref{eq:6.97} and Young's inequality we bound the term $\hat{R}_S(\delta_\tau \tilde{\bm{\mathcal{E}}}_h^{n+1})$ as
\begin{equation}
\begin{aligned}
&\left|\hat{R}_S(\delta_\tau \tilde{\bm{\mathcal{E}}}_h^{n+1})\right|  = \left|\mathcal{A}_d\left(q_r^{n+1}(\mathbf{u}), \delta_\tau \tilde{\bm{\mathcal{E}}}_h^{n+1}\right) -\int_0^{t_{n+1}} \beta(t_{n+1}-s) \mathcal{A}_d(\mathbf{u},\delta_\tau \tilde{\bm{\mathcal{E}}}_h^{n+1})ds\right| \\
    &\hspace{2em} +\left|\mathcal{A}_d\left(q_r^n(\mathbf{u}), \delta_\tau \tilde{\bm{\mathcal{E}}}_h^{n+1}\right) -\int_0^{t_n} \beta(t_n-s) \mathcal{A}_d(\mathbf{u},\delta_\tau \tilde{\bm{\mathcal{E}}}_h^{n+1})ds\right| \\
    &\leq \tilde{C} \tau^2 \int_0^{t_{n+1}} e^{-2\eta(t_{n+1-}t)}\left(\|\mathbf{u}(t)\|_{H^1(\mathcal{O})}^2+h^2\|\mathbf{u}(t)\|_{H^2(\mathcal{O})}^2+\|\mathbf{u}_t(t)\|_{H^1(\mathcal{O})}^2+h^2\|\mathbf{u}_t(t)\|_{H^2(\mathcal{O})}^2\right)dt\\
    &\hspace{2em} +\tilde{C} \tau^2 \int_0^{t_n} e^{-2\eta(t_n-t)}\left(\|\mathbf{u}(t)\|_{H^1(\mathcal{O})}^2+h^2\|\mathbf{u}(t)\|_{H^2(\mathcal{O})}^2+\|\mathbf{u}_t(t)\|_{H^1(\mathcal{O})}^2+h^2\|\mathbf{u}_t(t)\|_{H^2(\mathcal{O})}^2\right)dt\\
    &\hspace{2em} + \frac{ \mu}{32}\|\tilde{\bm{\mathcal{E}}}_h^{n+1}\|_{\text{dG}}^2. \label{eqn193}
\end{aligned}
\end{equation}
Using the above bounds \eqref{eqn191}-\eqref{eqn193} and choosing $\varepsilon = 1/224$, inequality \eqref{8.243} becomes
\begin{equation}
\begin{aligned}
&\frac{1}{2} \left( \|\delta_\tau \bm{\mathcal{E}}_h^{n+1}\|^2 - \|\delta_\tau \bm{\mathcal{E}}_h^n\|^2 + \|\delta_\tau \bm{\mathcal{E}}_h^{n+1} - \delta_\tau \bm{\mathcal{E}}_h^n\|^2 \right) + \frac{\tau \mu}{16} \|\delta_\tau \tilde{\bm{\mathcal{E}}}_h^{n+1}\|_\text{dG}^2 + \frac{\tau^2}{16} |\delta_\tau v_h^{n+1}|_\text{dG}^2 \\
&+ \frac{\tau^2}{2} \left(\mathcal{A}_{\text{sip}}(\hat{\xi}_h^{n+1}, \hat{\xi}_h^{n+1}) - \mathcal{A}_{\text{sip}}(\hat{\xi}_h^n, \hat{\xi}_h^n)\right) + \frac{\tau^2}{2} (\tilde{H}_1^n - \tilde{H}_2^n) + \frac{\tau}{2 \delta \mu} (\|\hat{S}_h^n\|^2 - \|\hat{S}_h^{n+1}\|^2) \\
& + \tau \mathcal{A}_d\left( q_r^{n+1}(\delta_\tau\tilde{\bm{\mathcal{E}}}_h), \delta_\tau \tilde{\bm{\mathcal{E}}}_h^{n+1}\right) \leq \frac{C}{\mu} \int_{t_{n-1}}^{t_{n+1}} \|\partial_{tt} \mathbf{u}\|^2 dt+ C(\tau + h^2 \mu^{-1}) \int_{t_n}^{t_{n+1}} |\partial_t p|_{H^1(\mathcal{O})}^2 dt \\
&+ \frac{C}{\mu} \tau \|\delta_\tau \bm{\mathcal{E}}_h^n\|_{L^3(\mathcal{O})}^2 \norm{\tilde{\bm{\mathcal{E}}}_h^{n+1}}_\text{dG}^2 + \frac{C}{\mu}\left(1+\mu^2+\mu^{-1}+\norm{\bm{\mathcal{E}}_h^n}^2 + \norm{\tilde{\bm{\mathcal{E}}}_h^{n+1}}_\text{dG}^2 \right) \int_{t_{n-1}}^{t_{n+1}} \|\partial_t \mathbf{u}\|_{H^2(\mathcal{O})}^2 dt \\
&+ \tilde{C} \tau^2 \int_0^{t_{n+1}} e^{-2\eta(t_{n+1-}t)}\left(\|\mathbf{u}(t)\|_{H^1(\mathcal{O})}^2+h^2\|\mathbf{u}(t)\|_{H^2(\mathcal{O})}^2+\|\mathbf{u}_t(t)\|_{H^1(\mathcal{O})}^2+h^2\|\mathbf{u}_t(t)\|_{H^2(\mathcal{O})}^2\right)dt\\
    &+\tilde{C} \tau^2 \int_0^{t_n} e^{-2\eta(t_n-t)}\left(\|\mathbf{u}(t)\|_{H^1(\mathcal{O})}^2+h^2\|\mathbf{u}(t)\|_{H^2(\mathcal{O})}^2+\|\mathbf{u}_t(t)\|_{H^1(\mathcal{O})}^2+h^2\|\mathbf{u}_t(t)\|_{H^2(\mathcal{O})}^2\right)dt\\
&+ \frac{\tau^2}{2 \delta \mu} \norm{\hat{S}_h^n}^2 +\frac{\mu}{8}\tau^2 \sum_{i=1}^n \|\delta_\tau \tilde{\bm{\mathcal{E}}}_h^{i+1}\|_\text{dG}^2+ \frac{C}{\mu} \tau^{-1} \|\bm{\mathcal{E}}_h^n - \bm{\mathcal{E}}_h^{n-1}\|^2 + \delta_{n,1} |\mathcal{b}(\bm{\mathcal{E}}_h^0, \delta_\tau v_h^2)|. \label{8.248}
\end{aligned}
\end{equation}
Using \eqref{eq4.30}, we note the following
\begin{equation}
\begin{aligned}
\frac{\tau^2}{2} \sum_{n=1}^m \left( \tilde{H}_1^n - \tilde{H}_2^n \right) 
&= \frac{\tau^2}{2} \left( \sum_{F \in \mathcal{F}_h^i}  \frac{\tilde{\sigma}}{h_F} \| [\delta_\tau v_h^{m+1}] \|^2 
- \| \mathbf{G}_h \left( [\delta_\tau v_h^{m+1}] \right) \|^2 \right) \\
&\hspace{2em} + \frac{\tau^2}{2} \left( \| \mathbf{G}_h \left( [\delta_\tau v_h^1] \right) \|^2 
- \sum_{F \in \mathcal{F}_h^i}  \frac{\tilde{\sigma}}{h_F} \| [\delta_\tau v_h^1] \|^2 \right)\\
&\leq C\tau^2 | \delta_\tau v_h^1 |^2_\text{dG}. \label{eqn195}
\end{aligned}
\end{equation}
We then sum \eqref{8.248} over $n=1$ to $m$. Under the assumptions on regularity (\textbf{A2}), (\textbf{A3}), we utilize the continuity of $\mathcal{A}_{\text{sip}}$, Lemma 11, the inverse estimate \eqref{8.159}, \eqref{eqn195} and the lower bound on $\tau$ in \eqref{8.223} to obtain
\begin{equation}
\begin{aligned}
&\| \delta_\tau \bm{\mathcal{E}}_h^{m+1} \|^2 + \sum_{n=1}^m \| \delta_\tau \bm{\mathcal{E}}_h^{n+1} - \delta_\tau \bm{\mathcal{E}}_h^n \|^2 + \frac{1}{8} \sum_{n=1}^m \tau^2 | \delta_\tau v_h^{n+1} |^2_\text{dG} + \frac{\mu}{8} \sum_{n=1}^m \tau \| \delta_\tau \tilde{\bm{\mathcal{E}}}_h^{n+1} \|^2_\text{dG} \\
&+ \tau \sum_{n=1}^m \mathcal{A}_d\left( q_r^{n+1}(\delta_\tau\tilde{\bm{\mathcal{E}}}_h), \delta_\tau \tilde{\bm{\mathcal{E}}}_h^{n+1}\right)\leq \tilde{C}_T + \| \delta_\tau \bm{\mathcal{E}}_h^1 \|^2 + \frac{\tau}{2 \delta \mu} \|\hat{S}_h^1 \|^2 + C \tau^2 \left( | \hat{\xi}_h^1 |^2_\text{dG} + | \delta_\tau v_h^1 |^2_\text{dG} \right) \\
&+\tilde{C} \tau^3 \sum_{n=1}^m \int_0^{t_n}e^{-2\eta(t_n-t)}\left(\|\mathbf{u}(t)\|_{H^1(\mathcal{O})}^2+h^2\|\mathbf{u}(t)\|_{H^2(\mathcal{O})}^2+\|\mathbf{u}_t(t)\|_{H^1(\mathcal{O})}^2+h^2\|\mathbf{u}_t(t)\|_{H^2(\mathcal{O})}^2\right)dt\\
&+\tilde{C} \tau^3 \sum_{n=1}^m \int_0^{t_{n+1}}e^{-2\eta(t_{n+1}-t)}\left(\|\mathbf{u}(t)\|_{H^1(\mathcal{O})}^2+h^2\|\mathbf{u}(t)\|_{H^2(\mathcal{O})}^2+\|\mathbf{u}_t(t)\|_{H^1(\mathcal{O})}^2+h^2\|\mathbf{u}_t(t)\|_{H^2(\mathcal{O})}^2\right)dt \\
&+ \tilde{C} \tau \sum_{n=1}^m h^{-d/3} \| \delta_\tau \bm{\mathcal{E}}_h^n \|^2 \norm{\tilde{\bm{\mathcal{E}}}_h^{n+1}}^2_\text{dG} +\frac{\mu}{4}\tau^2 \sum_{n=1}^m\sum_{i=1}^n \|\delta_\tau \tilde{\bm{\mathcal{E}}}_h^{i+1}\|_\text{dG}^2+ |\mathcal{b}(\bm{\mathcal{E}}_h^0, \delta_\tau v_h^2)|. \label{eq:8.255}
\end{aligned}
\end{equation}
Under the assumption ($\mathbf{A2}$), we have
\begin{equation}
\begin{aligned}
\tilde{C} \tau^3 \sum_{n=1}^m \int_0^{t_n}e^{-2\eta(t_n-t)}&\left(\|\mathbf{u}(t)\|_{H^1(\mathcal{O})}^2+h^2\|\mathbf{u}(t)\|_{H^2(\mathcal{O})}^2+\|\mathbf{u}_t(t)\|_{H^1(\mathcal{O})}^2+h^2\|\mathbf{u}_t(t)\|_{H^2(\mathcal{O})}^2\right)dt \\
&\leq \tilde{C} \tau^2 e^{2 \delta t_m}, \label{269}
\end{aligned}
\end{equation}
and similarly,
\begin{equation}
\begin{aligned}
\tilde{C} \tau^3 \sum_{n=1}^m \int_0^{t_{n+1}}e^{-2\eta(t_{n+1}-t)}&\left(\|\mathbf{u}(t)\|_{H^1(\mathcal{O})}^2+h^2\|\mathbf{u}(t)\|_{H^2(\mathcal{O})}^2+\|\mathbf{u}_t(t)\|_{H^1(\mathcal{O})}^2+h^2\|\mathbf{u}_t(t)\|_{H^2(\mathcal{O})}^2\right)dt \\
&\leq \tilde{C} \tau^2 e^{2 \delta t_{m+1}}. \label{270}
\end{aligned}
\end{equation}
We use the positivity property \eqref{5.38} to the last term in the LHS of \eqref{eq:8.255}. Then applying the bounds \eqref{269} and \eqref{270} and using the discrete Gronwall's lemma, \eqref{eq:8.255} yields
\begin{equation}
\begin{aligned}
&\| \delta_\tau \bm{\mathcal{E}}_h^{m+1} \|^2 + \sum_{n=1}^m \| \delta_\tau \bm{\mathcal{E}}_h^{n+1} - \delta_\tau \bm{\mathcal{E}}_h^n \|^2 + \frac{1}{8} \sum_{n=1}^m \tau^2 | \delta_\tau v_h^{n+1} |^2_\text{dG} + \frac{\mu}{8} \sum_{n=1}^m \tau \| \delta_\tau \tilde{\bm{\mathcal{E}}}_h^{n+1} \|^2_\text{dG} \\
&\leq \tilde{C}_T +\tilde{C}_T \| \delta_\tau \bm{\mathcal{E}}_h^1 \|^2 + \tilde{C}_T\frac{\tau}{2 \delta \mu} \|\hat{S}_h^1 \|^2 + \tilde{C}_T \tau^2 \left( | \hat{\xi}_h^1 |^2_\text{dG} + | \delta_\tau v_h^1 |^2_\text{dG} \right) \\
&\hspace{2em} + \tilde{C}_T \tau \sum_{n=1}^m h^{-d/3} \| \delta_\tau \bm{\mathcal{E}}_h^n \|^2 \norm{\tilde{\bm{\mathcal{E}}}_h^{n+1}}^2_\text{dG} 
+ \tilde{C}_T|\mathcal{b}(\bm{\mathcal{E}}_h^0, \delta_\tau v_h^2)|. \label{eqn199}
\end{aligned}
\end{equation}
To bound the last term, we use \eqref{eq4.37}, the approximation properties, and \eqref{8.223} to obtain 
\begin{equation}
|\mathcal{b}(\bm{\mathcal{E}}_h^0, \delta_\tau v_h^2)| \leq C \| \bm{\mathcal{E}}_h^0 \| | \delta_\tau v_h^2 |_\text{dG} \leq C \tau | \mathbf{u}^0 |_{H^2(\mathcal{O})} | \delta_\tau v_h^2 |_\text{dG} \leq C + \frac{\tau^2}{32} | \delta_\tau v_h^2 |^2_\text{dG}.
\label{8.256}
\end{equation}
We now follow similar argument as in \cite{masri2023improved} to derive a bound for the term $\| \bm{\mathcal{E}}_h^1 - \bm{\mathcal{E}}_h^0 \|$. We choose $n=1$ in the equation \eqref{eq:6.64}. All the terms in the RHS of \eqref{eq:6.64} are bounded exactly as discussed earlier, except for the term $R_t(\tilde{\bm{\mathcal{E}}}_h^n)$. Under the assumption (\textbf{A2}), we obtain
\begin{align*}
|R_t(\tilde{\bm{\mathcal{E}}}_h^n)| \leq \frac{C}{\mu} \tau^2 \int_{t_{n-1}}^{t_n} \| \partial_{tt} \mathbf{u} \|^2 dt+ \frac{C}{\mu} \tau^4 \int_{t_{n-1}}^{t_n} | \partial_t \mathbf{u} |^2_{H^2(\mathcal{O})} dt+ \frac{\tau \mu}{32} \| \tilde{\bm{\mathcal{E}}}_h^n \|^2_\text{dG}.
\end{align*}
This leads to the following bound for the initial error
\begin{equation}
\begin{aligned}
\frac{1}{2} \| \bm{\mathcal{E}}_h^1 - \bm{\mathcal{E}}_h^0 \|^2 + \frac{\tau}{2 \delta \mu} \| S_h^1 \|^2+ \frac{\tau^2}{16} \left(| v_h^1 |^2_\text{dG} + | p_h^1 + S_h^1 |^2_\text{dG}\right) + \frac{\tau \mu}{8} \norm{\tilde{\bm{\mathcal{E}}}_h^1}^2_\text{dG}
\\
\leq \tilde{C} \tau^2 + \tilde{C} \tau^2 h^{2r}+ \tau |\mathcal{b}(\bm{\mathcal{E}}_h^0, v_h^1)| + \frac{1}{2} \tau |\mathcal{b}(\bm{\mathcal{E}}_h^0, \Pi_h \mathbf{u}^1 \cdot \tilde{\bm{\mathcal{E}}}_h^1)|,
\label{eqn272}
\end{aligned}
\end{equation}
where $S_h^1 = \delta \mu (\nabla_h \cdot \widetilde{\mathbf{u}}_h^1 - R_h([\widetilde{\mathbf{u}}_h^1]))+ \delta \beta_0 (\nabla_h \cdot \widetilde{\mathbf{u}}_h^1 - R_h([\widetilde{\mathbf{u}}_h^1])) $. Observe that by \eqref{6.102} and \eqref{eq4.31}, we have
\begin{equation*}
\hat{S}_h^1 = \delta \mu (\nabla_h \cdot \delta_\tau \widetilde{\mathbf{u}}_h^1 - R_h([\delta_\tau \widetilde{\mathbf{u}}_h^1])) + \delta \beta_0 (\nabla_h \cdot \widetilde{\mathbf{u}}_h^1 - R_h([\widetilde{\mathbf{u}}_h^1]))
\end{equation*}
Since, $\widetilde{\mathbf{u}}_h^0 = \Pi_h \mathbf{u}^0$, thus with the above equality, we obtain
\[
\hat{S}_h^1 = \frac{1}{\tau} S_h^1, \quad \hat{\xi}_h^1 = \frac{1}{\tau} (p_h^1 + S_h^1).
\]
Using \eqref{eq4.37} and approximation property \eqref{eq:6.96}, we arrive at
\begin{equation}
|\mathcal{b}(\bm{\mathcal{E}}_h^0, v_h^1)| \leq C \| \bm{\mathcal{E}}_h^0 \| | v_h^1 |_\text{dG} \leq \frac{\tau}{32} | v_h^1 |^2_\text{dG} + C \tau^{-1} h^{2r+2} | \mathbf{u}^0 |^2_{H^{r+1}(\mathcal{O})}. \label{eqn211}
\end{equation}
We divide the final term in \eqref{eqn272} and apply \eqref{4.117}, \eqref{4.118}, \eqref{eq:6.96}, along with the assumption on regularity (\textbf{A2}) to obtain
\begin{equation}
\begin{aligned}
\left|\mathcal{b}(\bm{\mathcal{E}}_h^0, \Pi_h \mathbf{u}^1 \cdot \tilde{\bm{\mathcal{E}}}_h^1)\right| \leq \left|\mathcal{b}(\bm{\mathcal{E}}_h^0, \mathbf{u}^1 \cdot \tilde{\bm{\mathcal{E}}}_h^1)\right| + \left|\mathcal{b}(\bm{\mathcal{E}}_h^0, (\Pi_h \mathbf{u}^1 - \mathbf{u}^1) \cdot \tilde{\bm{\mathcal{E}}}_h^1\right| \\
\leq C \left(|\mathbf{u}^1|_{H^{r+1}(\mathcal{O})} + |||\mathbf{u}^1|||\right)\| \bm{\mathcal{E}}_h^0 \| \|\tilde{\bm{\mathcal{E}}}_h^1\|_\text{dG} 
\leq \tilde{C} h^{2r+2} + \frac{\mu}{16} \|\tilde{\bm{\mathcal{E}}}_h^1\|^2_\text{dG} . \label{eqn212}
\end{aligned}
\end{equation}
With bounds \eqref{eqn211}-\eqref{eqn212}, together with \eqref{8.223} and recalling that $v_h^0 = p_h^0 = 0$, \eqref{eqn272} yields
\begin{equation}
\| \delta_\tau \bm{\mathcal{E}}_h^1 \|^2 + \frac{\tau}{\delta \mu} \| \hat{S}_h^1 \|^2 + \frac{\tau^2}{16} \left( | \hat{\xi}_h^1 |^2_\text{dG} + | \delta_\tau v_h^1 |^2_\text{dG} \right) + \frac{\tau \mu}{16} \| \tilde{\bm{\mathcal{E}}}_h^1 \|^2_\text{dG}
\leq \tilde{C}. \label{8.255}
\end{equation}
Therefore, utilizing \eqref{8.256} and \eqref{8.255}, the bound \eqref{eqn199} reads
\begin{equation}
\begin{aligned}
\| \delta_\tau \bm{\mathcal{E}}_h^{m+1} \|^2 + \sum_{n=1}^m \| \delta_\tau \bm{\mathcal{E}}_h^{n+1} - \delta_\tau \bm{\mathcal{E}}_h^n \|^2 + \frac{\tau^2}{32} \sum_{n=1}^m | \delta_\tau v_h^{n+1} |^2_\text{dG} + \frac{\tau \mu}{4} \sum_{n=1}^m \| \delta_\tau \tilde{\bm{\mathcal{E}}}_h^{n+1} \|^2_\text{dG} \\
\leq \tilde{C}_T + \tilde{C}_T \tau \sum_{n=1}^m h^{-d/3} \| \delta_\tau \bm{\mathcal{E}}_h^n \|^2 \norm{\tilde{\bm{\mathcal{E}}}_h^{n+1}}^2_\text{dG}. \label{8.261}
\end{aligned}
\end{equation}
We still need to address the final term. We shall employ a method similar to the one outlined in \cite{nochetto2005gauge}. Using Lemma 11 and condition \eqref{8.223}, we obtain
\[
\| \delta_\tau \bm{\mathcal{E}}_h^n \|^2 = \frac{1}{\tau^2} \| \bm{\mathcal{E}}_h^n - \bm{\mathcal{E}}_h^{n-1} \|^2 \leq \tilde{C}_T \tau^{-1}, \quad \sum_{n=1}^m \| \tilde{\bm{\mathcal{E}}}_h^{n+1} \|^2_\text{dG} \leq \tilde{C}_T.
\]
By applying the above bounds in \eqref{8.261} and for very small $h$, we derive
\[
\| \delta_\tau \bm{\mathcal{E}}_h^{m+1} \|^2 \leq \tilde{C}_T + \tilde{C}_T^3 h^{-d/3} \leq \tilde{C}_T^3 h^{-d/3}, \quad m \geq 0.
\]
After applying this bound in \eqref{8.261}, and utilizing Lemma 11, we get
\[
\| \delta_\tau \bm{\mathcal{E}}_h^{m+1} \|^2 \leq \tilde{C}_T + \tilde{C}_T^4 \tau h^{-2d/3} \sum_{n=1}^m \| \tilde{\bm{\mathcal{E}}}_h^{n+1} \|^2_\text{dG} \leq \tilde{C}_T + \tilde{C}_T^5 \tau h^{-2d/3}.
\]
In view of \eqref{8.223}, we have $\tau h^{-2d/3} \leq c_2 h^{(\upsilon-1)d/3}$. Since $\upsilon < 1$, for sufficiently small $h$, we obtain
\[
\| \delta_\tau \bm{\mathcal{E}}_h^{m+1} \|^2 \leq \tilde{C}_T+ \tilde{C}_T^5 c_2 h^{(\upsilon-1)d/3} \leq \tilde{C}_T^5 c_2 h^{(\upsilon-1)d/3}.
\]
Choose the smallest integer $\varrho$ such that $\varrho \upsilon \geq 1$. Now iteratively applying the bound in \eqref{8.261} $\varrho$ times and by using \eqref{8.223}, we finally obtain
\[
\| \delta_\tau \bm{\mathcal{E}}_h^{m+1} \|^2 \leq \tilde{C}_T + (\tilde{C}_T)^{\varrho+4} c_2^\varrho h^{(\varrho \upsilon-1)d/3} \leq \tilde{C}_T + (\tilde{C}_T)^{\varrho+4} c_2^\varrho,
\]
and this completes the proof.
\end{proof}
\begin{lemma}
Assume that the hypotheses of Lemma 14 hold. Then, for $1 \leq m \leq N - 1$, the following error bound 
\begin{equation*}
\frac{\mu}{2} \tau \sum_{n=1}^m \| \delta_\tau \bm{\mathcal{E}}_h^{n+1} \|^2 \leq \bar{C}_T (\tau + h^{2r}),
\end{equation*}
holds, where $\bar{C}_T = \tilde{C}_T + \widehat{C}_T (1 + \tilde{C})$ is independent of $h$ and $\tau$.
\end{lemma}
\begin{proof}
By \eqref{8.168}, we derive $\forall n \geq 0$,
\begin{equation}
\mathcal{A}_d (\delta_\tau \mathbf{V}_h^{n+1}, \mathbf{w}_h) - \mathcal{b} (\mathbf{w}_h, \delta_\tau P_h^{n+1}) = (\delta_\tau \bm\psi^{n+1}, \mathbf{w}_h), \quad \forall \mathbf{w}_h \in \mathbf{M}_h. \label{8.257}
\end{equation}
Choose $\mathbf{w}_h = \delta_\tau \mathbf{V}_h^{n+1}$ in \eqref{8.226}. Then, use \eqref{8.169}, with a bit of rearranging, to have 
\begin{equation}
\begin{aligned}
(\delta_\tau (\widetilde{\mathbf{u}}_h^{n+1} - \mathbf{u}^{n+1}) - \delta_\tau \bm\psi^n, \delta_\tau \mathbf{V}_h^{n+1}) + \tau \mu \mathcal{A}_d (\delta_\tau \tilde{\bm{\mathcal{E}}}_h^{n+1}, \delta_\tau \mathbf{V}_h^{n+1})+ \tau \mathcal{A}_d\left( q_r^{n+1}(\delta_\tau \tilde{\bm{\mathcal{E}}}_h), \delta_\tau \mathbf{V}_h^{n+1}\right) \\
+ \tau \hat{R}_S(\delta_\tau \mathbf{V}_h^{n+1})
= -\tau \mathcal{b} (\delta_\tau \mathbf{V}_h^{n+1}, \delta_\tau p^{n+1})- \tau \mu \mathcal{A}_d (\delta_\tau \Pi_h \mathbf{u}^{n+1} - \delta_\tau \mathbf{u}^{n+1}, \delta_\tau \mathbf{V}_h^{n+1}) \\
+ \mathcal{L}_1 (\delta_\tau \mathbf{V}_h^{n+1}) + \mathcal{L}_2 (\delta_\tau \mathbf{V}_h^{n+1}) + \bar{R}_t (\delta_\tau \mathbf{V}_h^{n+1}). \label{8.258}
\end{aligned}
\end{equation}
Where,
\begin{align*}
\hat{R}_S(\delta_\tau \mathbf{V}_h^{n+1}) & = \mathcal{A}_d\left(q_r^{n+1}(\mathbf{u}), \delta_\tau \mathbf{V}_h^{n+1}\right) -\int_0^{t_{n+1}} \beta(t_{n+1}-s) \mathcal{A}_d(\mathbf{u},\delta_\tau \mathbf{V}_h^{n+1}) ds \notag \\
    &\hspace{2em} -\mathcal{A}_d\left(q_r^n(\mathbf{u}), \delta_\tau \mathbf{V}_h^{n+1}\right) +\int_0^{t_n} \beta(t_n-s) \mathcal{A}_d(\mathbf{u},\delta_\tau \mathbf{V}_h^{n+1}) ds, \\
\bar{R}_t (\delta_\tau \mathbf{V}_h^{n+1}) &= ((\partial_t \mathbf{u})^{n+1} - (\partial_t \mathbf{u})^n - (\delta_\tau \mathbf{u}^{n+1} - \delta_\tau \mathbf{u}^n), \delta_\tau \mathbf{V}_h^{n+1}).
\end{align*}
To begin with, we observe from \eqref{8.232}, \eqref{8.169}, and \eqref{8.257} that
\begin{equation}
\begin{aligned}
(\delta_\tau (\widetilde{\mathbf{u}}_h^{n+1} - \mathbf{u}^{n+1}) - \delta_\tau \bm\psi^n, \delta_\tau \mathbf{V}_h^{n+1}) &= \mathcal{A}_d (\delta_\tau \mathbf{V}_h^{n+1} - \delta_\tau \mathbf{V}_h^n, \delta_\tau \mathbf{V}_h^{n+1}) \\
&= \frac{1}{2} (\mathcal{A}_d (\delta_\tau \mathbf{V}_h^{n+1}, \delta_\tau \mathbf{V}_h^{n+1}) - \mathcal{A}_d (\delta_\tau \mathbf{V}_h^n, \delta_\tau \mathbf{V}_h^n)) \\
&\quad + \frac{1}{2} \mathcal{A}_d (\delta_\tau \mathbf{V}_h^{n+1} - \delta_\tau \mathbf{V}_h^n, \delta_\tau \mathbf{V}_h^{n+1} - \delta_\tau \mathbf{V}_h^n). \label{eqn206}
\end{aligned}
\end{equation}
Furthermore, by using \eqref{8.169}, and a bound similar to \eqref{8.171}, we deduce that
\begin{equation}
\begin{aligned}
&\mathcal{b}(\delta_\tau \mathbf{V}_h^{n+1}, \delta_\tau p^{n+1}) \\
&\leq \sum_{F \in \mathcal{F}_h} \left|\int_F \{ \delta_\tau p^{n+1} - \pi_h (\delta_\tau p^{n+1}) \} [\delta_\tau \mathbf{V}_h^{n+1}] \cdot \mathbf{n}_F \right|
\leq C h^2 \| \delta_\tau \bm\psi^{n+1} \| \| \delta_\tau p^{n+1} \|_{H^1(\mathcal{O})}  \\
&\leq \varepsilon \mu \| \delta_\tau \bm{\mathcal{E}}_h^{n+1} \|^2 + C \left( \frac{1}{\varepsilon \mu} + 1 \right) \tau^{-1} h^4 \int_{t_n}^{t_{n+1}} | \partial_t p |_{H^1(\mathcal{O})}^2 dt
+ C \tau^{-1} h^{2r+2} \int_{t_n}^{t_{n+1}} | \partial_t \mathbf{u} |_{H^{r+1}(\mathcal{O})}^2 dt.
\end{aligned}
\end{equation}
From \eqref{8.164}, it follows that
\begin{equation*}
-\Delta (\delta_\tau \mathbf{V}_h^{n+1}) + \nabla (\delta_\tau P^{n+1}) = \delta_\tau \bm\psi^{n+1}, \forall n \geq 1.
\end{equation*}
Since we have assumed the domain to be convex, similar to \eqref{8.167} and \eqref{8.188}, we obtain
\begin{equation}
\begin{aligned}
&\| \delta_\tau \mathbf{V}_h^{n+1} \|_{\text{dG}} + |\delta_\tau P_h^{n+1} |_{\text{dG}} \leq C (\| \delta_\tau \mathbf{V}_h^{n+1} \|_{H^2(\mathcal{O})} + | \delta_\tau P^{n+1} |_{H^1(\mathcal{O})}) \\
&\leq C \| \delta_\tau \bm\psi^{n+1} \| \leq C \| \delta_\tau \bm{\mathcal{E}}_h^{n+1} \| + C \tau^{-1/2} h^{r+1} \left( \int_{t_n}^{t_{n+1}} | \partial_t \mathbf{u} |_{H^{r+1}(\mathcal{O})}^2 dt \right)^{1/2}. \label{8.260}
\end{aligned}
\end{equation}
By setting $\mathbf{w}_h = \delta_\tau \Pi_h \mathbf{u}^{n+1} - \delta_\tau Q_h \mathbf{u}^{n+1}$ in \eqref{8.257}, using \eqref{eq4.37} and \eqref{8.260}, we obtain
\begin{equation}
\begin{aligned}
&|\mathcal{A}_d(\delta_\tau \Pi_h \mathbf{u}^{n+1} - \delta_\tau \mathbf{u}^{n+1}, \delta_\tau \mathbf{V}_h^{n+1})| \\
&\leq |\mathcal{b}(\delta_\tau \Pi_h \mathbf{u}^{n+1} - \delta_\tau Q_h \mathbf{u}^{n+1}, \delta_\tau P_h^{n+1})| + |(\delta_\tau \bm\psi^{n+1}, \delta_\tau \Pi_h \mathbf{u}^{n+1} - \delta_\tau Q_h u^{n+1})|  \\
&\leq C \| \delta_\tau \bm\psi^{n+1} \| \| \delta_\tau \Pi_h \mathbf{u}^{n+1} - \delta_\tau Q_h \mathbf{u}^{n+1} \|  \\
&\leq \varepsilon \| \delta_\tau \bm{\mathcal{E}}_h^{n+1} \|^2 + C \left( \frac{1}{\varepsilon} + 1 \right) \tau^{-1} h^{2r+2} \int_{t_n}^{t_{n+1}} | \partial_t \mathbf{u} |_{H^{r+1}(\mathcal{O})}^2 dt.
\end{aligned}
\end{equation}
Next, choosing $\mathbf{w}_h = \delta_\tau \tilde{\bm{\mathcal{E}}}_h^{n+1}$ in \eqref{8.257}, we get
\begin{equation}
\begin{aligned}
\mathcal{A}_d (\delta_\tau \tilde{\bm{\mathcal{E}}}_h^{n+1}, \delta_\tau \mathbf{V}_h^{n+1}) &=  (\delta_\tau \bm\psi^{n+1},\delta_\tau \tilde{\bm{\mathcal{E}}}_h^{n+1}) 
+ \mathcal{b} (\delta_\tau \tilde{\bm{\mathcal{E}}}_h^{n+1}, \delta_\tau P_h^{n+1}) \\ 
& = \| \delta_\tau \bm{\mathcal{E}}_h^{n+1} \|^2 +(\delta_\tau \bm\psi^{n+1},\delta_\tau \tilde{\bm{\mathcal{E}}}_h^{n+1} - \delta_\tau \bm{\mathcal{E}}_h^{n+1}) + \mathcal{b} (\delta_\tau \tilde{\bm{\mathcal{E}}}_h^{n+1} - \delta_\tau \bm{\mathcal{E}}_h^{n+1}, \delta_\tau P_h^{n+1})  \\
& \hspace{2em} + \mathcal{b} (\delta_\tau \bm{\mathcal{E}}_h^{n+1}, \delta_\tau P_h^{n+1})+ (\delta_\tau (\Pi_h \mathbf{u}^{n+1}) - \delta_\tau \mathbf{u}^{n+1}, \delta_\tau \bm{\mathcal{E}}_h^{n+1}).
\end{aligned}
\end{equation}
Use of \eqref{8.237}, \eqref{eq4.37}, and \eqref{8.260} now leads to
\begin{equation}
\begin{aligned}
|(\delta_\tau \bm\psi^{n+1}, \delta_\tau \tilde{\bm{\mathcal{E}}}_h^{n+1} - \delta_\tau \bm{\mathcal{E}}_h^{n+1})| \leq \varepsilon \| \delta_\tau \bm{\mathcal{E}}_h^{n+1} \|^2 + C \left( \frac{1}{\varepsilon} + 1 \right) \tau^2 | \delta_\tau v_h^{n+1} |_{\text{dG}}^2 \\
+ C \tau^{-1} h^{2r+2} \int_{t_n}^{t_{n+1}} | \partial_t \mathbf{u} |_{H^{r+1}(\mathcal{O})}^2 dt.
\end{aligned}
\end{equation}
Similarly, \eqref{eq4.37} together with \eqref{8.260} now yields
\begin{equation}
\begin{aligned}
|\mathcal{b}(\delta_\tau \tilde{\bm{\mathcal{E}}}_h^{n+1} - \delta_\tau \bm{\mathcal{E}}_h^{n+1}, \delta_\tau P_h^{n+1})| \leq \varepsilon \| \delta_\tau \bm{\mathcal{E}}_h^{n+1} \|^2 + C \left( \frac{1}{\varepsilon} + 1 \right) \tau^2 | \delta_\tau v_h^{n+1} |_{\text{dG}}^2 \\
+ C \tau^{-1} h^{2r+2} \int_{t_n}^{t_{n+1}} | \partial_t \mathbf{u} |_{H^{r+1}(\mathcal{O})}^2 dt.
\end{aligned}
\end{equation}
Additionally, from \eqref{8.233}, the CS inequality, and \eqref{eq4.30}, we have
\begin{equation}
\begin{aligned}
& |\mathcal{b}(\delta_\tau \bm{\mathcal{E}}_h^{n+1}, \delta_\tau P_h^{n+1})|+ |(\delta_\tau (\Pi_h \mathbf{u}^{n+1}) - \delta_\tau \mathbf{u}^{n+1}, \delta_\tau \bm{\mathcal{E}}_h^{n+1})|  \\
&\leq  C \tau | \delta_\tau v_h^{n+1} |_{\text{dG}} | \delta_\tau P_h^{n+1} |_{\text{dG}} + C h^{r+1} \| \delta_\tau \bm{\mathcal{E}}_h^{n+1} \| | \delta_\tau \mathbf{u}^{n+1} |_{H^{r+1}(\mathcal{O})}   \\
&\leq \varepsilon \| \delta_\tau \bm{\mathcal{E}}_h^{n+1} \|^2 + C \left( \frac{1}{\varepsilon} + 1 \right) \tau^2 | \delta_\tau v_h^{n+1} |_{\text{dG}}^2
+ C \left( \frac{1}{\varepsilon} + 1 \right) \tau^{-1} h^{2r+2} \int_{t_n}^{t_{n+1}} | \partial_t \mathbf{u} |_{H^{r+1}(\mathcal{O})}^2 dt.
\end{aligned}
\end{equation}
Next, choose $\mathbf{w}_h = q_r^{n+1}(\delta_\tau \tilde{\bm{\mathcal{E}}}_h)$ in \eqref{8.257} to have
\begin{equation}
\begin{aligned}
\mathcal{A}_d (q_r^{n+1}(\delta_\tau \tilde{\bm{\mathcal{E}}}_h), \delta_\tau \mathbf{V}_h^{n+1}) &= (\delta_\tau \bm\psi^{n+1},q_r^{n+1}(\delta_\tau \tilde{\bm{\mathcal{E}}}_h)) + \mathcal{b} (q_r^{n+1}(\delta_\tau \tilde{\bm{\mathcal{E}}}_h), \delta_\tau P_h^{n+1}) \\  &= (\delta_\tau \bm\psi^{n+1},q_r^{n+1}(\delta_\tau \tilde{\bm{\mathcal{E}}}_h-\delta_\tau \bm{\mathcal{E}}_h)) + (\delta_\tau (\Pi_h \mathbf{u}^{n+1}) - \delta_\tau \mathbf{u}^{n+1}, q_r^{n+1}(\delta_\tau \bm{\mathcal{E}}_h)) \\&\hspace{1em} + (\delta_\tau \bm{\mathcal{E}}_h^{n+1}, q_r^{n+1}(\delta_\tau \bm{\mathcal{E}}_h))
+ \mathcal{b} (q_r^{n+1}(\delta_\tau \tilde{\bm{\mathcal{E}}}_h - \delta_\tau \bm{\mathcal{E}}_h), \delta_\tau P_h^{n+1})\\&\hspace{1em} + \mathcal{b} (q_r^{n+1}(\delta_\tau \bm{\mathcal{E}}_h), \delta_\tau P_h^{n+1}).
\end{aligned}
\end{equation}
Combining \eqref{8.237}, \eqref{eq4.37}, and \eqref{8.260}, we obtain
\begin{equation}
\begin{aligned}
|(\delta_\tau \bm\psi^{n+1}, q_r^{n+1}(\delta_\tau \tilde{\bm{\mathcal{E}}}_h - \delta_\tau \bm{\mathcal{E}}_h)| \leq \varepsilon \mu \| \delta_\tau \bm{\mathcal{E}}_h^{n+1} \|^2 + \tilde{C} \left( \frac{1}{\varepsilon \mu} + 1 \right) \tau^3 \sum_{i=1}^{n+1} | \delta_\tau v_h^i |_{\text{dG}}^2 \\
+ \tilde{C} \tau^{-1} h^{2r+2} \int_{t_n}^{t_{n+1}} | \partial_t \mathbf{u} |_{H^{r+1}(\mathcal{O})}^2 dt.
\end{aligned}
\end{equation}
Analogously, \eqref{eq4.37} and \eqref{8.260} leads to
\begin{equation}
\begin{aligned}
|\mathcal{b}(q_r^{n+1}(\delta_\tau \tilde{\bm{\mathcal{E}}}_h - \delta_\tau \bm{\mathcal{E}}_h), \delta_\tau P_h^{n+1})| \leq \varepsilon \mu \| \delta_\tau \bm{\mathcal{E}}_h^{n+1} \|^2 + \tilde{C} \left( \frac{1}{\varepsilon \mu} + 1 \right) \tau^3 \sum_{i=1}^n | \delta_\tau v_h^i |_{\text{dG}}^2 \\
+ \tilde{C} \tau^{-1} h^{2r+2} \int_{t_n}^{t_{n+1}} | \partial_t \mathbf{u} |_{H^{r+1}(\mathcal{O})}^2 dt.
\end{aligned}
\end{equation}
By the CS inequality and Young's inequality, we obtain
\begin{equation}
\begin{aligned}
    \left|(\delta_\tau \bm{\mathcal{E}}_h^{n+1}, q_r^{n+1}(\delta_\tau \bm{\mathcal{E}}_h))\right| &\leq \tilde{C} \tau \|\delta_\tau \bm{\mathcal{E}}_h^{n+1}\| \sum_{i=1}^{n+1}\|\delta_\tau \bm{\mathcal{E}}_h^i\| \\
    &\leq \varepsilon \mu \|\delta_\tau \bm{\mathcal{E}}_h^{n+1}\|^2 + \frac{\tilde{C}}{\varepsilon \mu} \tau \sum_{i=1}^{n+1}\|\delta_\tau \bm{\mathcal{E}}_h^i\|^2. 
\end{aligned}
\end{equation}
Further, use of \eqref{8.233}, the CS inequality, and \eqref{eq4.30}, yields
\begin{equation}
\begin{aligned}
&|(\delta_\tau (\Pi_h \mathbf{u}^{n+1}) - \delta_\tau \mathbf{u}^{n+1}, q_r^{n+1}(\delta_\tau \bm{\mathcal{E}}_h)| + |\mathcal{b}(q_r^{n+1}(\delta_\tau \bm{\mathcal{E}}_h), \delta_\tau P_h^{n+1})| \\
&\leq \tilde{C} h^{r+1} | \delta_\tau \mathbf{u}^{n+1} |_{H^{r+1}(\mathcal{O})} \sum_{i=1}^{n+1}\| \delta_\tau \bm{\mathcal{E}}_h^i \| + \tilde{C} \tau\sum_{i=1}^{n+1} | \delta_\tau v_h^i |_{\text{dG}} | \delta_\tau P_h^{n+1} |_{\text{dG}} \\
&\leq \varepsilon \mu \| \delta_\tau \bm{\mathcal{E}}_h^{n+1} \|^2 + \tau \sum_{i=1}^{n+1}\| \delta_\tau \bm{\mathcal{E}}_h^i \|^2 + \tilde{C} \left( \frac{1}{\varepsilon \mu} + 1 \right) \tau^3\sum_{i=1}^{n+1} | \delta_\tau v_h^i |_{\text{dG}}^2 \\
&\hspace{2em} + \tilde{C} \tau^{-1} h^{2r+2} \int_{t_n}^{t_{n+1}} | \partial_t \mathbf{u} |_{H^{r+1}(\mathcal{O})}^2 dt.
\end{aligned}
\end{equation}
Next, to handle the term $\hat{R}_S(\delta_\tau \mathbf{V}_h^{n+1})$, we proceed similarly as previous lemma to arrive at
\begin{equation}
\begin{aligned}
&\left|\hat{R}_S(\delta_\tau \mathbf{V}_h^{n+1})\right|  = \left|\mathcal{A}_d\left(q_r^{n+1}(\mathbf{u}), \delta_\tau \mathbf{V}_h^{n+1}\right) -\int_0^{t_{n+1}} \beta(t_{n+1}-s) \mathcal{A}_d(\mathbf{u},\delta_\tau \mathbf{V}_h^{n+1}) ds\right| \\
    &\hspace{2em} +\left|\mathcal{A}_d\left(q_r^n(\mathbf{u}), \delta_\tau \mathbf{V}_h^{n+1}\right) -\int_0^{t_n} \beta(t_n-s) \mathcal{A}_d(\mathbf{u},\delta_\tau \mathbf{V}_h^{n+1}) ds\right| \\
    &\leq \tilde{C} \tau^2 \int_0^{t_{n+1}} e^{-2\eta(t_{n+1-}t)}\left(\|\mathbf{u}(t)\|_{H^1(\mathcal{O})}^2+h^2\|\mathbf{u}(t)\|_{H^2(\mathcal{O})}^2+\|\mathbf{u}_t(t)\|_{H^1(\mathcal{O})}^2+h^2\|\mathbf{u}_t(t)\|_{H^2(\mathcal{O})}^2\right)dt\\
    &\hspace{2em}+\tilde{C} \tau^2 \int_0^{t_n} e^{-2\eta(t_n-t)}\left(\|\mathbf{u}(t)\|_{H^1(\mathcal{O})}^2+h^2\|\mathbf{u}(t)\|_{H^2(\mathcal{O})}^2+\|\mathbf{u}_t(t)\|_{H^1(\mathcal{O})}^2+h^2\|\mathbf{u}_t(t)\|_{H^2(\mathcal{O})}^2\right)dt\\
    &\hspace{2em} + \tilde{C}\|\delta_\tau \mathbf{V}_h^{n+1}\|_{\text{dG}}^2.
\end{aligned}
\end{equation}
To handle the term $\bar{R}_t(\delta_\tau \mathbf{V}_h^{n+1})$, we define the function $\mathbf{J}^n = \tau (\partial_t \mathbf{u})^n - (\mathbf{u}^n - \mathbf{u}^{n-1})$, $\forall n \geq 1$. Note that, $\mathbf{J}^n \in \mathbf{M}$. Thus, we can write
\begin{equation}
\bar{R}_t(\delta_\tau \mathbf{V}_h^{n+1}) = \frac{1}{\tau} (\mathbf{J}^{n+1}, \delta_\tau \mathbf{V}_h^{n+1}) - \frac{1}{\tau} (\mathbf{J}^n, \delta_\tau \mathbf{V}_h^n) - \frac{1}{\tau} (\mathbf{J}^n, \delta_\tau \mathbf{V}_h^{n+1} - \delta_\tau \mathbf{V}_h^{n}).
\end{equation}
Using the CS inequality, \eqref{eq6.102}, and Taylor expansions, we obtain
\begin{equation}
\begin{aligned}
|(\mathbf{J}^n, \delta_\tau \mathbf{V}_h^{n+1} - \delta_\tau \mathbf{V}_h^{n})| &\leq C \| \mathbf{J}^n \| \| \delta_\tau \mathbf{V}_h^{n+1} - \delta_\tau \mathbf{V}_h^n \|_{\text{dG}} \\
&\leq C \tau^2 \int_{t_n}^{t_{n+1}} \| \partial_{tt} \mathbf{u} \|^2 dt+ \frac{1}{8} \tau \| \delta_\tau \mathbf{V}_h^{n+1} - \delta_\tau \mathbf{V}_h^{n} \|_{\text{dG}}^2. \label{eqn221}
\end{aligned}   
\end{equation}
Using the above bounds \eqref{eqn206}-\eqref{eqn221} and the coercivity of $\mathcal{A}_d$ \eqref{eq3.20}, inequality \eqref{8.258} becomes
\begin{equation}
\begin{aligned}
&\frac{1}{2} \left( \mathcal{A}_d (\delta_\tau \mathbf{V}_h^{n+1}, \delta_\tau \mathbf{V}_h^{n+1}) - \mathcal{A}_d (\delta_\tau \mathbf{V}_h^n, \delta_\tau \mathbf{V}_h^n) \right) + \frac{1}{8} \| \delta_\tau \mathbf{V}_h^{n+1} - \delta_\tau \mathbf{V}_h^n \|_{\text{dG}}^2 
+ (1 -9  \varepsilon) \tau \mu \| \delta_\tau \bm{\mathcal{E}}_h^{n+1} \|^2 \\
&\leq \tilde{C} \left( \frac{1}{\varepsilon} + 1 \right) \tau^3 | \delta_\tau v_h^{n+1} |_{\text{dG}}^2 + C \tau \int_{t_{n-1}}^{t_n} \| \partial_{tt} \mathbf{u} \|^2 dt+ \tilde{C} \left( \frac{1}{\varepsilon \mu} + 1 \right) \tau^4 \sum_{i=1}^{n+1} | \delta_\tau v_h^i |_{\text{dG}}^2 \\
&+ \tilde{C} \tau \|\delta_\tau \mathbf{V}_h^{n+1}\|_{\text{dG}}^2 
+ \tilde{C} \left( \frac{1}{\varepsilon} + 1 \right) h^{2r+2} \int_{t_n}^{t_{n+1}} | \partial_t \mathbf{u} |_{H^{r+1}(\mathcal{O})}^2 dt +\left(\frac{\tilde{C}}{\varepsilon}+1\right) \tau^2 \sum_{i=1}^{n+1} \|\delta_\tau \bm{\mathcal{E}}_h^i\|^2\\
&+ C \left( \frac{1}{\varepsilon \mu} + 1 \right) h^4 \int_{t_n}^{t_{n+1}} | \partial_t p |_{H^1(\mathcal{O})}^2 dt + \mathcal{L}_1 (\delta_\tau \mathbf{V}_h^{n+1}) + \mathcal{L}_2 (\delta_\tau \mathbf{V}_h^{n+1})\\
&+ \tilde{C} \tau^3 \int_0^{t_{n+1}} e^{-2\eta(t_{n+1-}t)}\left(\|\mathbf{u}(t)\|_{H^1(\mathcal{O})}^2+h^2\|\mathbf{u}(t)\|_{H^2(\mathcal{O})}^2+\|\mathbf{u}_t(t)\|_{H^1(\mathcal{O})}^2+h^2\|\mathbf{u}_t(t)\|_{H^2(\mathcal{O})}^2\right)dt\\
&+\tilde{C} \tau^3 \int_0^{t_n} e^{-2\eta(t_n-t)}\left(\|\mathbf{u}(t)\|_{H^1(\mathcal{O})}^2+h^2\|\mathbf{u}(t)\|_{H^2(\mathcal{O})}^2+\|\mathbf{u}_t(t)\|_{H^1(\mathcal{O})}^2+h^2\|\mathbf{u}_t(t)\|_{H^2(\mathcal{O})}^2\right)dt \\ 
& + \frac{1}{\tau} (\mathbf{J}^{n+1}, \delta_\tau \mathbf{V}_h^{n+1}) - \frac{1}{\tau} (\mathbf{J}^n, \delta_\tau \mathbf{V}_h^n). \label{8.266}
\end{aligned}
\end{equation}
For the nonlinear terms, we apply the same technique used in \cite{masri2023improved}. We choose $\varepsilon = \frac{1}{28}$, then use the estimates of Lemma 11, Lemma 14, and the condition \eqref{8.223} to obtain
\begin{equation}
\begin{aligned}
   &\sum_{n=1}^m|\mathcal{L}_1 (\delta_\tau \mathbf{V}_h^{n+1})| + \sum_{n=1}^m|\mathcal{L}_2 (\delta_\tau \mathbf{V}_h^{n+1})| \leq \frac{1}{7}\tau \mu \sum_{n=1}^m \|\delta_\tau \bm{\mathcal{E}}_h^{n+1}\|^2 \\&+C\left(1+\frac{1}{\mu}\right)\tau \sum_{n=1}^m\|\delta_\tau \mathbf{V}_h^{n+1}\|_\text{dG}^2+\left(\tilde{C}_T + \widehat{C}_T(1+\tilde{C})\right)(\tau+h^{2r}). \label{8.267}
\end{aligned}
\end{equation}
Next, combining the last two terms in \eqref{8.266} results in
\begin{equation}
\begin{aligned}
&\frac{1}{\tau} \left(\mathbf{J}^{m+1}, \delta_\tau \mathbf{V}_h^{m+1} \right) - \frac{1}{\tau} \left(\mathbf{J}^1, \delta_\tau \mathbf{V}_h^1 \right) \\
&\leq \frac{1}{8} \left( \|\delta_\tau \mathbf{V}_h^{m+1}\|^2_\text{dG} + \|\delta_\tau \mathbf{V}_h^1\|^2_\text{dG} \right) + C \tau \int_{t_m}^{t_{m+1}} \|\partial_{tt} \mathbf{u}\|^2 dt + C \tau \int_{t_0}^{t_1} \|\partial_{tt} \mathbf{u}\|^2 dt. \label{8.268}
\end{aligned}
\end{equation}
We sum \eqref{8.266} over $n=1$ to $m$ and use Lemma 11, Lemma 14, the bounds \eqref{8.267}, \eqref{8.268} the continuity and coercivity of $\mathcal{A}_d$, the assumption \eqref{8.223} to obtain
\begin{align*}
&\frac{1}{8} \|\delta_\tau \mathbf{V}_h^{m+1}\|^2_\text{dG} + \frac{\mu}{2} \tau \sum_{n=1}^m \|\delta_\tau \bm{\mathcal{E}}_h^{n+1}\|^2 \leq \left(\tilde{C}_T + \widehat{C}_T(1+\tilde{C})\right) (\tau + h^{2r}) \\
&+ C \|\delta_\tau \mathbf{V}_h^1\|^2_\text{dG} + \tilde{C} \tau \sum_{n=1}^m \|\delta_\tau \mathbf{V}_h^{n+1}\|^2_\text{dG} + \tilde{C}\tau^2 \sum_{n=1}^m \sum_{i=1}^{n+1} \|\delta_\tau \bm{\mathcal{E}}_h^i\|^2 \\
&+\tilde{C} \tau^3 \sum_{n=1}^m \int_0^{t_n}e^{-2\eta(t_n-t)}\left(\|\mathbf{u}(t)\|_{H^1(\mathcal{O})}^2+h^2\|\mathbf{u}(t)\|_{H^2(\mathcal{O})}^2+\|\mathbf{u}_t(t)\|_{H^1(\mathcal{O})}^2+h^2\|\mathbf{u}_t(t)\|_{H^2(\mathcal{O})}^2\right)dt\\
&+\tilde{C} \tau^3 \sum_{n=1}^m \int_0^{t_{n+1}}e^{-2\eta(t_{n+1}-t)}\left(\|\mathbf{u}(t)\|_{H^1(\mathcal{O})}^2+h^2\|\mathbf{u}(t)\|_{H^2(\mathcal{O})}^2+\|\mathbf{u}_t(t)\|_{H^1(\mathcal{O})}^2+h^2\|\mathbf{u}_t(t)\|_{H^2(\mathcal{O})}^2\right)dt.
\end{align*} 
Apply the bounds \eqref{269} and \eqref{270} and the discrete Gronwall's lemma to have
\begin{equation}
\begin{aligned}
\frac{1}{8} \|\delta_\tau \mathbf{V}_h^{m+1}\|^2_\text{dG} + \frac{\mu}{2} \tau \sum_{n=1}^m \|\delta_\tau \bm{\mathcal{E}}_h^{n+1}\|^2 \leq \left(\tilde{C}_T + \widehat{C}_T(1+\tilde{C})\right) (\tau + h^{2r}) \\
+ \tilde{C}_T \|\delta_\tau \mathbf{V}_h^1\|^2_\text{dG} . \label{290}
\end{aligned}
\end{equation}
Next, we consider the term $\|\delta_\tau \mathbf{V}_h^1\|^2_\text{dG}$. With $n=1$ in \eqref{8.175}, we use the following alternative bounds for \eqref{8.181} and the bound on $\tilde{R}_C(\mathbf{V}_h^n)$ \eqref{8.216}, derived by applying Young's inequality differently, to have
\begin{align*}
|\tilde{R}_C(\mathbf{V}_h^1)| &\leq 7 \varepsilon \mu \|\bm{\mathcal{E}}_h^1\|^2 + 2 \varepsilon \|\bm{\mathcal{E}}_h^0 - \bm{\mathcal{E}}_h^1\|^2 
+ C \left( \frac{h^2}{\varepsilon \mu} + 1 \right) \left( \tau^2 |v_h^1|_{\text{dG}}^2 + h^{2r+2} \right) \\
&\hspace{2em} + C \left( \frac{1}{\varepsilon \mu} + 1 \right) \left( h^2 (\|\bm{\mathcal{E}}_h^1\|^2 + \|\bm{\mathcal{E}}_h^0\|^2) + (\|\bm{\mathcal{E}}_h^0\|^2 + h^2) \|\tilde{\bm{\mathcal{E}}}_h^1\|_{\text{dG}}^2 \right) \\
&\hspace{2em} + C \tau^2 \int_{t_0}^{t_1} \|\partial_t \mathbf{u}\|^2 \, dt + \frac{C}{\varepsilon \mu}\|\mathbf{V}_h^1\|_{\text{dG}}^2+\frac{1}{16}\tau^{-1}\|\mathbf{V}_h^1\|_{\text{dG}}^2. \\
|(\tau(\partial_t \mathbf{u})^1 &- (\mathbf{u}^1 - \mathbf{u}^0), \mathbf{V}_h^1)| \leq C \tau^3 \int_{t_0}^{t_1} \|\partial_{tt} \mathbf{u}\|^2 dt+ \frac{1}{16} \|\mathbf{V}_h^1\|^2_\text{dG}.
\end{align*}
However, we retain other bounds and expressions in the same way as in the proof of Theorem 3. We use \eqref{8.223}, \eqref{8.255}, and the observation that $\mathbf{V}_h^0 = 0$ to obtain
\begin{align*}
\frac{1}{16} \|\mathbf{V}_h^1\|^2_\text{dG}
&\leq \tilde{C} \tau \left(h^{2r+2} + \tau^2 |v_h^1|^2_\text{dG} + \|\bm{\mathcal{E}}_h^1 - \bm{\mathcal{E}}_h^0\|^2 + h^2 (\|\tilde{\bm{\mathcal{E}}}_h^1\|^2_\text{dG} + \|\bm{\mathcal{E}}_h^1\|^2)\right) + \tilde{C} \tau \|\mathbf{V}_h^1\|^2_\text{dG}\\
&\leq \tilde{C} \tau^3.
\end{align*}
Hence, $\|\delta_\tau \mathbf{V}_h^1\|^2_\text{dG} \leq \tilde{C}_T \tau$. With this bound, \eqref{290} yields the result.
\end{proof}

\section{\textit{A priori} bound for the pressure}
\begin{theorem}
Under the same assumptions of Lemma 14, the following error estimate for the pressure holds. For $1 \leq m \leq N$:
\begin{equation*}
\tau \sum_{n=1}^{m} \| p^n - p_h^n \|^2 \leq \check{C}_T (\tau + h^{2r}),
\end{equation*}
where $\check{C}_T$ independent of $h$ and $\tau$ such that, $\check{C}_T = \tilde{C} \bar{C}_T + \tilde{C}$.
\end{theorem}
\begin{proof}
Let $\widetilde{\mathbf{M}}_h$ be a subspace of $\mathbf{M}_h$ such that
\begin{equation*}
\widetilde{\mathbf{M}}_h = \{ \mathbf{w}_h \in \mathbf{M}_h : \forall F \in \mathcal{F}_h, [\mathbf{w}_h] \cdot \mathbf{n}_F = 0 \}.
\end{equation*}
There is a positive constant $\beta^*$, independent of $h$, satisfying
\begin{equation*}
\inf_{g_h \in P_h} \sup_{\widetilde{\mathbf{u}}_h \in \widetilde{\mathbf{M}}_h} \frac{-\mathcal{b}(\widetilde{\mathbf{u}}_h, g_h)}{\| \widetilde{\mathbf{u}}_h \|_\text{dG} \| g_h \|} \geq \beta^*.
\end{equation*}
One can find the proof of this in Theorem 6.8 in \cite{riviere2008discontinuous}.
As a consequence of the above inf-sup condition, there exists a unique $\bm\zeta_h^n \in \widetilde{\mathbf{M}}_h$ such that
\begin{equation}
\mathcal{b}(\bm\zeta_h^n, p_h^n - \pi_h p^n) = \| p_h^n - \pi_h p^n \|^2, \quad \| \bm\zeta_h^n \|_\text{dG} \leq \frac{1}{\beta^*} \| p_h^n - \pi_h p^n \|, \quad \forall n \geq 0. \label{8.274}
\end{equation}
We substitute \eqref{eq:6.110} in \eqref{7.105}. Using \eqref{eq:6.111}, we obtain for all $\mathbf{w}_h \in \mathbf{M}_h$,
\begin{align*}
(\bm{\mathcal{E}}_h^n - \bm{\mathcal{E}}_h^{n-1}, \mathbf{w}_h) &+ \tau \mathcal{A}_C(\mathbf{V}_h^{n-1}; \mathbf{V}_h^{n-1}, \tilde{\bm{\mathcal{E}}}_h^n, \mathbf{w}_h) + \tau R_C(\mathbf{w}_h) + \tau \mu \mathcal{A}_d(\tilde{\bm{\mathcal{E}}}_h^n, \mathbf{w}_h) + \tau \mathcal{A}_d(q_r^n(\tilde{\bm{\mathcal{E}}}_h),\mathbf{w}_h) \\
&+ \tau R_S(\mathbf{w}_h) = \tau \mathcal{b}(\mathbf{w}_h, p_h^n - \pi_h p^n) - \tau \mathcal{b}(\mathbf{w}_h, p^n - \pi_h p^n) - \tau \mu \mathcal{A}_d(\Pi_h \mathbf{u}^n - \mathbf{u}^n, \mathbf{w}_h)\\
&- \tau \mathcal{A}_d(q_r^n(\Pi_h \mathbf{u} - \mathbf{u}), \mathbf{w}_h) 
+ \tau \delta \mu \mathcal{b}(\mathbf{w}_h, \nabla_h \cdot \tilde{\bm{\mathcal{E}}}_h^n - R_h([\tilde{\bm{\mathcal{E}}}_h^n])) + \tau \delta \mathcal{b}(\mathbf{w}_h, \nabla_h \cdot q_r^n(\tilde{\bm{\mathcal{E}}}_h) \\
&- R_h([q_r^n(\tilde{\bm{\mathcal{E}}}_h)]))  + R_t(\mathbf{w}_h).
\end{align*}
Choosing $\mathbf{w}_h = \bm\zeta_h^n$ in the above and using \eqref{8.274} results in
\begin{align}
\tau \| p_h^n - \pi_h p^n \|^2 = (\bm{\mathcal{E}}_h^n - \bm{\mathcal{E}}_h^{n-1}, \bm\zeta_h^n) + \tau \mathcal{A}_C(\mathbf{V}_h^{n-1}; \mathbf{V}_h^{n-1}, \tilde{\bm{\mathcal{E}}}_h^n, \bm\zeta_h^n) + \tau R_C(\bm\zeta_h^n) + \tau \mathcal{A}_d(q_r^n(\tilde{\bm{\mathcal{E}}}_h),\bm\zeta_h^n) \notag \\
 + \tau R_S(\bm\zeta_h^n) + \tau \mu \mathcal{A}_d(\tilde{\bm{\mathcal{E}}}_h^n, \bm\zeta_h^n) + \tau \mathcal{b}(\bm\zeta_h^n, p^n - \pi_h p^n) + \tau \mu \mathcal{A}_d(\Pi_h \mathbf{u}^n - \mathbf{u}^n, \bm\zeta_h^n)- \tau \mathcal{A}_d(q_r^n(\Pi_h \mathbf{u} - \mathbf{u}), \bm\zeta_h^n) \notag \\
- \tau \delta \mu \mathcal{b}(\bm\zeta_h^n, \nabla_h \cdot \tilde{\bm{\mathcal{E}}}_h^n - R_h([\tilde{\bm{\mathcal{E}}}_h^n])) - \tau \delta \mathcal{b}(\bm\zeta_h^n, \nabla_h \cdot q_r^n(\tilde{\bm{\mathcal{E}}}_h) - R_h([q_r^n(\tilde{\bm{\mathcal{E}}}_h)])) - R_t(\bm\zeta_h^n). \label{8.275}
\end{align}
We bound the first term in the RHS of the above equality by applying the CS inequality, \eqref{8.274}, and Young's inequality as
\begin{equation}
| (\bm{\mathcal{E}}_h^n - \bm{\mathcal{E}}_h^{n-1}, \bm\zeta_h^n) | \leq C\tau \| \delta_\tau \bm{\mathcal{E}}_h^n \|^2 + \frac{\tau}{16} \| p_h^n - \pi_h p^n \|^2. \label{eqn228}
\end{equation}
From Lemma 8 of \cite{masri2023improved}, we have
\begin{equation*}
|\mathcal{A}_C(\mathbf{V}_h^{n-1}; \mathbf{V}_h^{n-1}, \tilde{\bm{\mathcal{E}}}_h^n, \bm\zeta_h^n)| \leq C \| \mathbf{V}_h^{n-1} \|_{L^3(\mathcal{O})} \| \tilde{\bm{\mathcal{E}}}_h^n \|_\text{dG} \| \bm\zeta_h^n \|_\text{dG}.
\end{equation*}
Under the assumption on regularity (\textbf{A2}), using Lemma 11, \eqref{eq:6.96}, \eqref{8.157} and the assumption \eqref{8.223}, we have
\begin{align*}
\| \mathbf{V}_h^{n-1} \|_{L^3(\mathcal{O})}^2 &\leq 2 \| \bm{\mathcal{E}}_h^{n-1} \|_{L^3(\mathcal{O})}^2 + 2 \| \Pi_h \mathbf{u}^{n-1} \|_{L^3(\mathcal{O})}^2 \\
&\leq \tilde{C} h^{-d/3}(\tau + h^{2r}) + C \| \mathbf{u}^{n-1} \|_{W^{1,3}(\mathcal{O})}^2 \leq \tilde{C}.
\end{align*}
Thus, making use of Young's inequality and \eqref{8.274}, we obtain
\begin{equation}
|\mathcal{A}_c(\mathbf{V}_h^{n-1}; \mathbf{V}_h^{n-1}, \tilde{\bm{\mathcal{E}}}_h^n, \bm\zeta_h^n)| \leq \tilde{C} \| \tilde{\bm{\mathcal{E}}}_h^n \|_\text{dG}^2 + \frac{1}{16} \| p_h^n - \pi_h p^n \|^2.
\end{equation}
To obtain bound for the other nonlinear terms, we apply Lemma 8. For this, we note that
\begin{align*}
R_C(\bm\zeta_h^n) &= \mathcal{A}_c(\mathbf{V}_h^{n-1}; \mathbf{V}_h^{n-1}, \Pi_h \mathbf{u}^n - \mathbf{u}^n, \bm\zeta_h^n) + \mathcal{A}_c(\mathbf{V}_h^n; \bm{\mathcal{E}}_h^{n-1}, \mathbf{u}^n, \bm\zeta_h^n) \\
&\hspace{2em}+ \mathcal{A}_c(\mathbf{V}_h^n; \Pi_h \mathbf{u}^{n-1} - \mathbf{u}^{n-1}, \mathbf{u}^n, \bm\zeta_h^n) + \mathcal{A}_c(\mathbf{u}^n; \mathbf{u}^{n-1} - \mathbf{u}^n, \mathbf{u}^n, \bm\zeta_h^n) = \sum_{i=1}^4 L_i^n.
\end{align*}
Further, we split $L_1^n$ as follows:
\begin{align*}
L_1^n &= \mathscr{C} (\bm{\mathcal{E}}_h^{n-1}, \Pi_h \mathbf{u}^n - \mathbf{u}^n, \bm\zeta_h^n) + \mathscr{C} (\Pi_h \mathbf{u}^{n-1}, \Pi_h \mathbf{u}^n - \mathbf{u}^n, \bm\zeta_h^n) \\
&\hspace{2em} - \mathscr{U} (\mathbf{V}_h^{n-1}; \bm{\mathcal{E}}_h^{n-1}, \Pi_h \mathbf{u}^n - \mathbf{u}^n, \bm\zeta_h^n) - \mathscr{U} (\mathbf{V}_h^{n-1}; \Pi_h \mathbf{u}^{n-1}, \Pi_h \mathbf{u}^n - \mathbf{u}^n, \bm\zeta_h^n).
\end{align*}
We apply \eqref{4.118} for the first term, \eqref{4.120} for the second and fourth terms, and \eqref{4.121} for the third term, to obtain
\begin{equation}
|L_1^n| \leq C (\| \bm{\mathcal{E}}_h^{n-1} \| + h^r \| \mathbf{u}^n \|_{H^{r+1}(\mathcal{O})} ) \| \bm\zeta_h^n \|_\text{dG}. \notag
\end{equation}
We utilize \eqref{4.117} and \eqref{4.119} to bound the terms $L_2^n$ and $L_3^n$, respectively. For the term $L_4^n$ we proceed as follows:
\begin{equation}
L_4^n = \int_\mathcal{O} (\mathbf{u}^{n-1} - \mathbf{u}^n) \cdot \nabla u^n \cdot \bm\zeta_h^n \leq C \| \mathbf{u}^n - \mathbf{u}^{n-1} \| |\mathbf{u}^n|_{W^{1,3}(\mathcal{O})} \| \bm\zeta_h^n \|_\text{dG}. \notag
\end{equation}
Under the assumption on regularity (\textbf{A2}), we obtain
\begin{equation}
|R_C(\bm\zeta_h^n)| \leq C ( \| \bm{\mathcal{E}}_h^{n-1} \|^2 + h^{2r}) + C \tau \int_{t_{n-1}}^{t_n} \| \partial_t \mathbf{u} \|^2dt + \frac{1}{32} \| p_h^n - \pi_h p^n \|^2.
\end{equation}
Following arguments similar to the one used in \cite{masri2022discontinuous}, we arrive at
\begin{equation}
|R_t(\bm\zeta_h^n)| \leq \frac{1}{16} \tau \| p_h^n - \pi_h p^n \|^2 + C \tau^2 \int_{t_{n-1}}^{t_n} \| \partial_{tt} \mathbf{u} \|^2 dt + C h^{2r+2} \int_{t_{n-1}}^{t_n} \| \partial_t \mathbf{u} \|_{H^{r+1}(\mathcal{O})}^2 dt.
\end{equation}
Under the regularity assumption (\textbf{A2}), the bound for $R_S(\bm\zeta_h^n)$ can be obtained as
\begin{equation}
\begin{aligned}
&|R_C(\bm\zeta_h^n)| =\left|\mathcal{A}_\epsilon(q_r^n(\mathbf{u}),\tilde{\bm{\mathcal{E}}}_h^n) - \int_0^{t_n} \beta(t_n-s) \mathcal{A}_\epsilon(\mathbf{u}(s), \bm\zeta_h^n) ds\right| \\ 
& \leq \tilde{C} \tau^2 \int_0^{t_n} e^{-2\eta(t_n-t)}\left(\|\mathbf{u}(t)\|_{H^1(\mathcal{O})}^2+h^2\|\mathbf{u}(t)\|_{H^2(\mathcal{O})}^2+\|\mathbf{u}_t(t)\|_{H^1(\mathcal{O})}^2+h^2\|\mathbf{u}_t(t)\|_{H^2(\mathcal{O})}^2\right)dt\\
&\hspace{2em} + \frac{1}{32}\|\bm\zeta_h^n\|_{\text{dG}}^2. \label{eqn232}
\end{aligned}
\end{equation}
The remaining terms are dealt with using reasoning similar to before, the details are thus omitted. We obtain:
\begin{align*}
&|\mathcal{A}_d\textbf{}(\tilde{\bm{\mathcal{E}}}_h^n, \bm\zeta_h^n)| \leq C \|\tilde{\bm{\mathcal{E}}}_h^n\|_{\text{dG}} \|\bm\zeta_h^n\|_{\text{dG}} \leq C \|\tilde{\bm{\mathcal{E}}}_h^n\|_{\text{dG}} \|p_h^n - \pi_h p^n\|, \\
&|\mathcal{b}(\bm\zeta_h^n, p^n - \pi_h p^n)| \leq C h^r \|\bm\zeta_h^n\|_{\text{dG}} |p^n|_{H^r(\mathcal{O})} \leq C h^r |p^n|_{H^r(\mathcal{O})} \|p_h^n - \pi_h p^n\|, \\
&|\mathcal{b}(\bm\zeta_h^n, \nabla_h \cdot \tilde{\bm{\mathcal{E}}}_h^n - R_h([\tilde{\bm{\mathcal{E}}}_h^n]))| \leq C \|\tilde{\bm{\mathcal{E}}}_h^n\|_{\text{dG}} \|\bm\zeta_h^n\|_{\text{dG}} \leq C \|\tilde{\bm{\mathcal{E}}}_h^n\|_{\text{dG}} \|p_h^n - \pi_h p^n\|,\\
&|\mathcal{A}_d(\Pi_h \mathbf{u}^n - \mathbf{u}^n, \bm\zeta_h^n)| \leq C h^r |\mathbf{u}^n|_{H^{r+1}(\mathcal{O})} \|\bm\zeta_h^n\|_{\text{dG}} \leq C h^r |\mathbf{u}^n|_{H^{r+1}(\mathcal{O})} \|p_h^n - \pi_h p^n\|, \\
&|\mathcal{A}_d(q_r^n(\tilde{\bm{\mathcal{E}}}_h),\bm\zeta_h^n)|  \leq C \norm{q_r^n(\tilde{\bm{\mathcal{E}}}_h)}_\text{dG} \norm{\bm\zeta_h^n}_\text{dG} \leq \tilde{C} \tau \left(\sum_{i=1}^n \norm{\tilde{\bm{\mathcal{E}}}_h^i}_{\text{dG}}\right) \| p_h^n - \pi_h p^n \|, \\
&|\mathcal{b}(\bm\zeta_h^n, \nabla_h \cdot q_r^n(\tilde{\bm{\mathcal{E}}}_h) - R_h([q_r^n(\tilde{\bm{\mathcal{E}}}_h)]))| \leq C \norm{q_r^n(\tilde{\bm{\mathcal{E}}}_h)}_\text{dG} \norm{\bm\zeta_h^n}_\text{dG} \leq \tilde{C} \tau \left(\sum_{i=1}^n \norm{\tilde{\bm{\mathcal{E}}}_h^i}_{\text{dG}}\right) \| p_h^n - \pi_h p^n \|,\\
&|\mathcal{A}_d(q_r^n(\Pi_h \mathbf{u} - \mathbf{u}), \bm\zeta_h^n)| \leq \tilde{C}\tau \sum_{i=1}^n \|\Pi_h \mathbf{u}^i - \mathbf{u}^i\|_\text{dG} \|\bm\zeta_h^n\|_{\text{dG}} \leq \tilde{C} \tau h^r \sum_{i=1}^n |\mathbf{u}^i|_{H^{r+1}(\mathcal{O})} \|p_h^n - \pi_h p^n\|.
\end{align*}
We use Young's inequality in each of the above bounds, and then substitute these bounds along with the bounds \eqref{eqn228}-\eqref{eqn232} in \eqref{8.275}. Sum the resulting inequality over $n = 1$ to $m$ to have
\begin{align*}
&\frac{\tau}{2} \sum_{n=1}^m \|p_h^n - \pi_h p^n\|^2 \leq  \tilde{C} \tau \sum_{n=1}^m \|\tilde{\bm{\mathcal{E}}}_h^n\|^2_{\text{dG}} + C \tau^2 \int_0^T \left(\|\partial_{tt} \mathbf{u}\|^2 + \|\partial_t \mathbf{u}\|^2\right) dt \\
&+ C h^{2r+2} \int_0^T \|\partial_t \mathbf{u}\|^2_{H^{r+1}(\mathcal{O})} dt+ C \tau \sum_{n=1}^m \left(\|\bm{\mathcal{E}}_h^{n-1}\|^2 + h^{2r} + \|\delta_\tau \bm{\mathcal{E}}_h^n\|^2\right) \\
&+\tilde{C} \tau^3 \sum_{n=1}^m \int_0^{t_n}e^{-2\eta(t_n-t)}\left(\|\mathbf{u}(t)\|_{H^1(\mathcal{O})}^2+h^2\|\mathbf{u}(t)\|_{H^2(\mathcal{O})}^2+\|\mathbf{u}_t(t)\|_{H^1(\mathcal{O})}^2+h^2\|\mathbf{u}_t(t)\|_{H^2(\mathcal{O})}^2\right)dt.
\end{align*}
The final result is obtained by using Lemma 11, Lemma 15, the bound \eqref{166}, the bound \eqref{8.255}, and the triangle inequality.
\end{proof}

\section{Numerical validation}
This section provides numerical results that support the established theoretical results. To discretize the spatial domain, we employ the discontinuous piecewise polynomials of degree $r$ to approximate the discrete velocity and polynomials of degree $r-1$ to approximate the discrete pressure (i.e., $\mathbb{P}_r$-$\mathbb{P}_{r-1}$ mixed finite element spaces). Here, we have shown the results for $\mathbb{P}_2$–$\mathbb{P}_1$ and $\mathbb{P}_1$–$\mathbb{P}_0$ dG methods. For the temporal discretization, we apply the backward Euler method, which has first-order accuracy. In addition, we use the right rectangle rule to approximate the integral term, which also has first-order accuracy. The computational domain for space, denoted as $\mathcal{O}$, is chosen to be the square $(0, 1) \times (0, 1)$. The simulations for this example are carried out on the time interval $[0, 1]$, with a final time $T = 1$. \\
\textbf{Example: } In this example, we consider the following exact solution \\$(\mathbf{u}, p) = (\mathbf{U}_1(x, y, t), \mathbf{U}_2(x, y, t), p(x, y, t))$:
\begin{align*}
\mathbf{U}_1(x, y, t) &= x^3 (x - 1)^2 y^2 (y - 1)(5y - 3) (t+1), \\
\mathbf{U}_2(x, y, t) &= - x^2(x - 1)(5x - 3) y^3 (y - 1)^2 (t+1), \\
p(x, y, t) &= \sin{(\pi x)} \cos{(\pi y)} (t+1).
\end{align*}

We compute the forcing term $\mathbf{f}$ based on the exact solutions provided above. To begin with, we analyze the convergence rates in time by solving the problem with time step sizes $\tau \in \{1/2^3, 1/2^4, \dots, 1/2^7\}$. We fix the mesh resolution $h_F = 1/2^7$ for the $\mathbb{P}_2$–$\mathbb{P}_1$ dG scheme. We set $\epsilon = -1$, $\tilde{\sigma} =10$, $\sigma = 8$ on $\mathcal{F}_h^i$, and $\sigma = 16$ on $\mathcal{F}_h^b$. The other parameters are chosen to be $\mu=1$, $\gamma=0.1$, $\eta=0.1$. Let $\text{err}_{\tau}$ denote the error for a given time step size $\tau$, then the convergence rate is calculated as
$\ln(\text{err}_{\tau} / \text{err}_{\tau/2})/\ln(2)$. The errors and convergence rates in time are presented in Table 1. Our numerical results reveal a superconvergence in time for both velocity and pressure, which is higher than the expected first-order rate.

Next, we assess the convergence rates in space by solving the problem on a sequence of uniformly refined meshes. The time step sizes are fixed as follows: $\tau = 1/2^{10}$ for the $\mathbb{P}_1$–$\mathbb{P}_0$ dG method, $\tau = 1/2^{12}$ for the $\mathbb{P}_2$–$\mathbb{P}_1$ dG method. For the $\mathbb{P}_1$–$\mathbb{P}_0$ method, we set $\tilde{\sigma} = 10$, $\epsilon = -1$, $\sigma = 6$ on $\mathcal{F}_h^i$, and $\sigma = 12$ on $\mathcal{F}_h^b$. For the $\mathbb{P}_2$–$\mathbb{P}_1$ method, we set $\tilde{\sigma} = 10$, $\epsilon = -1$, $\sigma = 8$ on $\mathcal{F}_h^i$, and $\sigma = 16$ on $\mathcal{F}_h^b$.
The convergence rate is computed as
$\ln(\text{err}_{h_F} / \text{err}_{h_F/2})/\ln(2)$,
where $\text{err}_{h_F}$ represents the error associated with a mesh of resolution $h_F$. The errors and convergence rates in space are presented in Table 2, and we observe optimal convergence rates.
\begin{table}[ht]
    \centering
    \caption{Errors and order of convergence in time for velocity and pressure}
    \begin{tabular}{cccccc}
        \toprule
        $r$ & $\tau$ & $\|\mathbf{u}_h^{T} - \mathbf{u}(T)\|$ & rate & $\|p_h^{T} - p(T)\|$ & rate \\
        \midrule
        2 & $1/2^3$  & $7.629E-3$ & --    & $1.938E-1$ & -- \\
          & $1/2^4$  & $1.883E-3$ & 2.018 & $7.411E-2$ & 1.387 \\
          & $1/2^5$  & $4.825E-4$ & 1.965 & $2.380E-2$ & 1.640 \\
          & $1/2^6$  & $1.232E-4$ & 1.970 & $6.959E-3$ & 1.773 \\
          & $1/2^7$  & $3.111E-5$ & 1.985 & $1.953E-3$ & 1.833 \\
        \bottomrule
    \end{tabular}
\end{table}
\begin{table}[ht]
    \centering
    \caption{Errors and order of convergence in space for velocity and pressure}
    \begin{tabular}{cccccccc}
        \toprule
        $r$ & $h_F$ & $\|\mathbf{u}_h^{T} - \mathbf{u}(T)\|$ & rate & $\|\mathbf{u}_h^{T} - \mathbf{u}(T)\|_\text{dG}$ & rate & $\|p_h^{T} - p(T)\|$ & rate \\
        \midrule
        1 & $1/2^1$  & $2.224E-2$ & --    & $2.541E-1$& --   & $4.875E-1$ & --    \\
          & $1/2^2$  & $6.963E-3$ & 1.676 & $1.426E-1$& 0.833& $2.613E-1$ & 0.900 \\
          & $1/2^3$  & $1.717E-3$ & 2.020 & $7.087E-2$& 1.009& $1.322E-1$ & 0.983 \\
          & $1/2^4$  & $4.180E-4$ & 2.038 & $3.407E-2$& 1.056& $6.621E-2$ & 0.997 \\
          & $1/2^5$  & $1.103E-4$ & 1.922 & $1.694E-2$& 1.008& $3.312E-2$ & 0.999 \\
        \midrule
        2 & $1/2^1$  & $5.551E-3$ & --    & $1.565E-1$& --   & $1.622E-1$ & --    \\
          & $1/2^2$  & $9.590E-4$ & 2.533 & $4.554E-2$&1.781 & $4.090E-2$ & 1.988 \\
          & $1/2^3$  & $1.299E-4$ & 2.884 & $1.112E-2$&2.020 & $1.021E-2$ & 2.002 \\
          & $1/2^4$  & $1.681E-5$ & 2.950 & $2.787E-3$&2.011 & $2.548E-3$ & 2.003 \\
          & $1/2^5$  & $2.138E-6$ & 2.974 & $6.962E-4$&2.001 & $6.359E-4$ & 2.002 \\
        \bottomrule
    \end{tabular}
\end{table}

\section{Conclusion}
This article presents the error analysis for velocity and pressure resulting from a dG pressure correction splitting scheme for solving the Oldroyd model of order one. Existence, uniqueness of the discrete solutions, consistency, and stability of the schemes are established. We have derived the optimal \textit{a priori} bounds for the velocity in the dG norm and also in the $L^2$ norm by carefully designing the dual problem. To establish the error estimates for pressure, we need to derive the error bounds for the discrete time derivative of velocity. Theoretically, we observe that the estimate for the pressure is optimal in space but remains suboptimal in terms of time accuracy. However, our numerical results indicate  superconvergence in time for both velocity and pressure. Proving the superconvergence in time for velocity and pressure is highly non-trivial, which requires further investigation, and we leave this to future work.

\bibliographystyle{amsplain}
\bibliography{mybib}

\end{document}